\documentclass[mathpazo]{cicp}
\usepackage{subfigure}
\usepackage{color}
\usepackage{cases}
\newcommand\bcase{\begin{numcases}{}}
\newcommand\ecase{\end{numcases}}

\begin{document}
\title{
Fully Conservative Difference Schemes for the Rotation-Two-Component Camassa-Holm
System with Smooth/Nonsmooth Initial Data\footnote{The research was	supported by Zhejiang
Provincial Natural Science Foundation of China (Grant No. LZ23A010007).}}



 \author[Yan T et.~al.]{Tong Yan\affil{1},
       Jiwei Zhang\affil{2}, and Qifeng Zhang\affil{1}}
 \address{\affilnum{1}\ Department of Mathematics,
          Zhejiang Sci-Tech University,
          Hangzhou, 310018, P.R. China. \\
           \affilnum{2}\ School of Mathematics and Statistics, and Hubei Key Laboratory of Computational Science,
           Wuhan University, Wuhan, 430072, P.R. China}
 \emails{{\tt yantong0320@163.com} (T.~Yan), {\tt jiweizhang@whu.edu.cn} (J.~Zhang),
          {\tt zhangqifeng0504@gmail.com} (Q.~Zhang)}

\begin{abstract}
The rotation-two-component Camassa--Holm system, which possesses strongly nonlinear coupled terms and high-order differential terms, tends to have continuous nonsmooth solitary wave solutions, such as peakons, stumpons, composite waves and even chaotic waves. In this paper an accurate semi-discrete conservative difference scheme for the system is derived by taking advantage of its Hamiltonian invariants. We show that the semi-discrete numerical scheme preserves at least three discrete conservative laws: mass, momentum and energy. Furthermore, a fully discrete finite difference scheme is proposed without destroying anyone of the conservative laws. Combining a nonlinear iteration process and an efficient threshold strategy, the accuracy of the numerical scheme can be guaranteed. Meanwhile, the difference scheme can capture the formation and propagation of solitary wave solutions with satisfying long time behavior under the smooth/nonsmooth initial data. The numerical results reveal a new type of asymmetric wave breaking phenomenon under the nonzero rotational parameter.
\end{abstract}


\maketitle

\section{Introduction}
\label{sec1}
As one of the typical representatives of the nonlinearly dispersive partial differential equations, the Camassa--Holm (CH) equation
\begin{align}
	m_t+u m_x +2mu_x = 0,\quad {\rm with}\quad m = u-  u_{xx}\label{Problem1}
\end{align}
was proposed to simulate the evolution of shallow water waves \cite{CH1993}. Here $m$ represents the momentum related to the fluid velocity $u$.
The CH equation instead of the famous KdV equation allows traveling wave solutions in the explicit form of
$$u(x,t) = c' {\rm e}^{-|x-ct|}$$
with a sharp peak (peakons), where $c'$ is the amplitude and $c$ is the wave velocity. The CH equation has at least the following remarkable properties: complete integrability, infinitely conservation laws and bi-Hamiltonian structure \cite{Cons1997,Lene2005}.
In this paper, we focus on the generalized two-component case of the CH equation \eqref{Problem1} given by
\begin{align}
	&m_t + \sigma(um_x+2mu_x) = -3(1-\sigma)uu_x + Au_x - \mu u_{xxx}
	-(1-2\Omega A)\rho \rho_x + 2\Omega \rho (\rho u)_x,  \quad t \in (0,T], \label{eq1.3} \\
	&\rho_t + (\rho u)_x = 0, \quad  t \in (0,T] \label{eq1.4}
\end{align}
with the periodic boundary condition for $x\in {R}$. The system \eqref{eq1.3}--\eqref{eq1.4} is called the rotation-two-component Camassa--Holm (R2CH) system.

The R2CH system  was firstly proposed by Fan et al. \cite{FGL2016} in 2016, which
not only inherits most of the properties of the solutions of the Camassa--Holm equation,
but also adds a new variable $\rho$ to describe the height of the water waves. Moreover, it also introduces
a constant rotational speed $\Omega$ ($\Omega\in[0,1/4)$) to characterize the effect of the Earth's rotation on the shallow water waves.
Therefore, the R2CH system can depict the propagation phenomenon of shallow water waves more accurately.

Recalling $m = u- u_{xx}$, the R2CH system \eqref{eq1.3}--\eqref{eq1.4} can be rewritten as
\begin{align}
	&u_t - u_{xxt} - Au_x + 3uu_x = \sigma (2u_x u_{xx} + uu_{xxx}) - \mu u_{xxx} - (1-2\Omega A)\rho \rho_x + 2\Omega \rho (\rho u)_x,  \quad t \in (0,T],  \label{eq1.1}   \\
	&\rho_t + (\rho u)_x = 0, \quad  t \in (0,T].    \label{eq1.2}
\end{align}
Several significant works have been developed in theory after its first appearance.
For example, blow-up scenarios of strong solutions are discussed in \cite{Zhang2017,LPZ2019}, blow-up
criteria and wave-breaking phenomena are established in \cite{Moon2017}, solitary waves with singularities,
like peakons and cuspons are introduced in \cite{CFGL2017},  peakon weak solutions in distribution
sense are considered in \cite{FY2019}, the persistence properties of the system in weighted
$L^p$-spaces are investigated in \cite{YLQ2020} and the effect of the Coriolis force on traveling waves is studied in \cite{ZW2021}.

As we know, the system \eqref{eq1.1}--\eqref{eq1.2} has at least three conservative laws
\begin{align}
	&I_1(u,\rho) = \int_{R} \rho {\rm d}x, 
\\
	&I_2(u,\rho) = \int_{R} (u + \Omega \rho^2) {\rm d}x, 
\\
	&E(u,\rho) = \frac{1}{2}\int_R(u^2 + u_{x}^{2} + (1-2\Omega A)\rho^2) {\rm d}x.
\end{align}
In addition, for a special case of $\sigma=1$ and $\Omega=0$, the system \eqref{eq1.1}--\eqref{eq1.2} has another three-order conservative law, which is described by
	\begin{align*}
		H = \int_{R} (u^3 + uu_x^2 - Au^2 -\mu u_x^2 + u\rho^2) {\rm d}x,
	\end{align*}
see also \cite{CLQ2011,ZX2013}. The detailed derivation is postponed in the appendix.

While theoretical studies have been well carried out in many aspects, numerical works for the R2CH system
\eqref{eq1.1}--\eqref{eq1.2} are still few in the literature and in urgent need for development.
At present, the only numerical scheme is finite difference discretization based on a bilinear operator, see e.g., \cite{ZLZ2022}.
However, the numerical simulation in \cite{ZLZ2022} only covers the smooth initial values with conditional stability.
In the case of nonsmooth initial data, it is not only theoretically difficult to analyze, but also difficult to capture the nonsmooth solutions. In order to develop more accurate numerical algorithms, this paper extends the excellent ideas in \cite{LP2016} to solve the system \eqref{eq1.1}--\eqref{eq1.2}.

Li and Vu-Quoc pointed out in \cite{LV1995}: ``\emph{in some areas, the ability to preserve some
	invariant properties of the original differential equation is a criterion to judge the success of a numerical simulation}.'' Hence, one of our principal targets in deriving numerical scheme is to preserve the specially intrinsic conservative structure of the system.
Other targets include the study of novel wave-breaking phenomena during the collision and evolution under the nonsmooth initial data in
a long time simulation.

The study of the high-order discrete conservation law is a very challenging subject. As Liu and Xing pointed out in \cite{LX2016}: ``\emph{it appears a rather difficult task to preserve all conservation laws}.'' In this paper, we make a tentative attempt to simulate a third-order discrete quantity $H^n$ in the last example in Section \ref{Sec5.3}. It seems that $H^n$ is very close to the conserved quantity $H$.
However, a rigorous theoretical interpretation is still lacking to determine whether the quantity $H^n$ is a discrete conservation law.

The main contributions of the paper are concretely summarized as follows.
\begin{itemize}
	\item The present difference scheme possesses at least three discrete conservative laws, enabling it to accurately capture the behaviors of solutions to the R2CH system under smooth/nonsmooth initial conditions.
	\item The difference scheme demonstrates several new phenomena with nice resolution for the first time, including the phenomena of the short-wave breaking and interaction of the long-wave propagation.
	\item The numerical accuracy of the difference scheme can be guaranteed and oscillations of nonsmooth solutions can also be effectively eliminated by a threshold technique.
	\item The difference scheme improves the numerical results in the literature and has potential impacts for predicting the propagation of solutions of other shallow water wave equations.
\end{itemize}

The rest of the paper is arranged as follows. In Section \ref{sec2},
a semi-discrete finite difference scheme is derived, which preserves the semi-discrete mass, momentum and the energy.
After that, a fully discrete finite difference scheme is established in Section \ref{sec3}, which preserves all the conservative laws in the discrete counterpart. Algorithm implementation in combination of a fixed point iteration method is carried out in Section \ref{sec4}. Numerical examples including the dam-break problem, singe peakon, multipeakon and peakon anti-peakon interaction problems demonstrate the discrete conservative laws and nice evolution of the solutions in Section \ref{sec5} followed by a brief conclusion in Section \ref{sec6}.

\section{Semi-discrete scheme}\label{sec2}
\setcounter{equation}{0}
To establish a conservative finite difference scheme for solving \eqref{eq1.3}--\eqref{eq1.4} with the periodic boundary condition on the computational domain $[0,\,L]$, we first discretize the system in space and derive a spatial semi-discrete scheme, which is shown to preserve conservation laws. Let $h = L/M$ be the spatial stepsize for a given integer $M$ and define the spatial variable by $x_i = ih$, $i = 0,1,\cdots,M$. Denote $u_i(t) = u(x_i,t)$ and $m_i(t) = m(x_i,t)$. Put forward on this basis, we propose the semi-discrete conservative finite difference scheme by
\begin{align}
	& \frac{{\rm d}}{{\rm d}t}m_i(t) + \frac{\sigma}{2h}\big[(m_{i+1}(t)u_{i+1}(t)-m_{i-1}(t)u_{i-1}(t))  + m_i(t)(u_{i+1}(t)-u_{i-1}(t))\big] \nonumber\\
	&\quad= -\frac{3(1-\sigma)}{4h}\big[u_{i+1}^{2}(t)-u_{i-1}^{2}(t)\big] + \frac{A}{2h}\big[u_{i+1}(t)-u_{i-1}(t)\big]  \nonumber\\
	&\quad \quad- \frac{\mu}{{2h^3}} \big[u_{i+2}(t)-2u_{i+1}(t)+2u_{i-1}(t)-u_{i-2}(t)\big] - \frac{(1-2\Omega A)}{{4h}}\big[\rho_{i+1}^2(t)-\rho_{i-1}^2(t)\big]  \nonumber\\
	&\quad \quad + \frac{\Omega }{2h}\rho_i(t)\big[(u_{i+1}(t)+u_i(t))(\rho_{i+1}(t)+\rho_i(t))-(u_{i-1}(t)+u_{i}(t))(\rho_{i-1}(t)+\rho_{i}(t))\big],  \label{eq2.1} \\
	&\frac{{\rm d}}{{\rm d}t}\rho_{i}(t) + \frac{1}{4h} \big[(u_{i+1}(t)+u_i(t))(\rho_{i+1}(t)+\rho_i(t))-(u_{i}(t)+u_{i-1}(t))(\rho_{i}(t)+\rho_{i-1}(t))\big] = 0,  \label{eq2.2} \\
	&m_i(t)  = u_i(t) - \frac{1}{2h}\big[v_{i+1}(t)-v_{i-1}(t)\big],  \label{eq2.3}
\end{align}
where $
v_i(t) = \frac1{2h}\big[u_{i+1}(t)-u_{i-1}(t)\big]$ and $i=1,2,\cdots,M$.

With the periodicity of the boundary conditions for discrete mesh functions $u_i$ and $\rho_i$, we have
\begin{align*}
	u_i = u_{i+M}, \quad \rho_i = \rho_{i+M}, \quad i = 0,1,\cdots,M-1.  
\end{align*}
Moreover, the periodicity for $u$ ensures the periodicity of the intermediate variable $m$.
Throughout the whole paper, we omit the spatial stepsize $h$ in the semi-discrete and discrete conservative laws.
On this basis, we define semi-discrete mass, momentum, and energy as follows
\begin{align}
	&I_1(t) = \sum_{i=1}^{M}\rho_{i}(t),  \label{eq2.5} \\
	&I_2(t) = \sum_{i=1}^{M}\left[u_i(t) + \Omega \rho_{i}^2(t)\right], \label{eq2.6} \\
	&E(t) = \frac{1}{2}\sum_{i=1}^{M}\bigg[u_{i}^{2}(t) + \Big(\frac{u_{i+1}(t)-u_{i-1}(t)}{2h}\Big)^2 + (1-2\Omega A)\rho_{i}^2(t)\bigg].  \label{eq2.7}
\end{align}
Therefore, we show that the semi-discrete scheme \eqref{eq2.1}--\eqref{eq2.3} is conservative in the following sense.
\begin{theorem} \label{Theorem2_1}
	Consider the semi-discrete scheme \eqref{eq2.1}--\eqref{eq2.3} of the R2CH system \eqref{eq1.3}--\eqref{eq1.4}. Then the semi-discrete mass, momentum, and the energy with $\sigma = 1$ are conservative  for the R2CH system \eqref{eq1.3}--\eqref{eq1.4} in the following sense
	\begin{align*}
		\frac{{\rm d}}{{\rm d}t}I_1(t) = 0, \quad \frac{{\rm d}}{{\rm d}t}I_2(t) = 0,\quad\frac{{\rm d}}{{\rm d}t}E(t) = 0.  
	\end{align*}
\end{theorem}
\begin{proof}
	(\textbf{I}). In combination of \eqref{eq2.2} and \eqref{eq2.5}, we have
	\begin{align*}
		\frac{{\rm d}}{{\rm d}t}I_1(t) = \sum_{i=1}^{M}\frac{{\rm d}}{{\rm d}t}\rho_{i}(t) = \sum_{i=1}^{M}\frac{(u_{i-1}+u_{i})(\rho_{i-1}+\rho_{i})- (u_{i+1}+u_i)(\rho_{i+1}+\rho_{i})}{4h} = 0,
	\end{align*}
	where the periodicity of the boundary conditions for $u$ and $\rho$ is used.
	
	(\textbf{II}). Similarly, we have
	\begin{align*}
		&\; \sum_{i=1}^{M}\frac{m_{i+1}(t)u_{i+1}(t)-m_{i-1}(t)u_{i-1}(t)}{2h} = 0,  \\
		&\;	\sum_{i=1}^{M}m_i(t)\frac{u_{i+1}(t)-u_{i-1}(t)}{2h} = 0,   \\
		&\;	\sum_{i=1}^{M}\frac{u_{i+1}^2(t)-u_{i-1}^2(t)}{4h} = 0, \\
		&\;	\sum_{i=1}^{M}\frac{u_{i+1}(t)-u_{i-1}(t)}{2h} = 0, \\
		&\;	\sum_{i=1}^{M}\frac{u_{i+2}(t)-2u_{i+1}(t)+2u_{i-1}(t)-u_{i-2}(t)}{2h^3} = 0, \\
		&\;	\sum_{i=1}^{M}\frac{\rho_{i+1}^2(t)-\rho_{i-1}^2(t)}{4h} = 0.
	\end{align*}
	Thus, summing over $i$ from $1$ to $M$ on both sides of \eqref{eq2.1}, we have
	\begin{align}
		\sum_{i=1}^{M}\frac{{\rm d}}{{\rm d}t}m_i(t)
		= 2\Omega\sum_{i=1}^{M}\rho_i(t)\frac{(u_{i+1}(t)+u_i(t))(\rho_{i+1}(t)+\rho_i(t))-(u_{i-1}(t)+u_{i}(t))(\rho_{i-1}(t)+\rho_{i}(t))}{4h}.    \label{eq2.9}
	\end{align}
	Taking the derivative on both sides of \eqref{eq2.3} with respect to $t$ and summing over $i$ from 1 to $M$, we have
	\begin{align}
		\sum_{i=1}^{M}\frac{{\rm d}}{{\rm d}t}m_i(t) = \sum_{i=1}^{M}\left(\frac{{\rm d}}{{\rm d}t}u_i(t) - \frac{v'_{i+1}(t)-v'_{i-1}(t)}{2h}\right) = \sum_{i=1}^{M}\frac{{\rm d}}{{\rm d}t}u_i(t).  \label{eq2.10}
	\end{align}
	Using \eqref{eq2.2}, \eqref{eq2.9} and \eqref{eq2.10}, we have
	\begin{align*}
		\frac{{\rm d}}{{\rm d}t}I_2(t) = &\; \sum_{i=1}^{M}\frac{{\rm d}}{{\rm d}t}u_i(t) + 2\Omega\sum_{i=1}^{M}\rho_{i}(t)\frac{{\rm d}}{{\rm d}t}\rho_i(t)    \nonumber \\
		= &\; 2\Omega\sum_{i=1}^{M}\rho_i(t)\frac{(u_{i+1}(t)+u_i(t))(\rho_{i+1}(t)+\rho_i(t))-(u_{i-1}(t)+u_{i}(t))(\rho_{i-1}(t)+\rho_{i}(t))}{4h}  \nonumber \\
		\quad&\; + 2\Omega\sum_{i=1}^{M}\rho_i(t)\frac{(u_{i-1}(t)+u_i(t))(\rho_{i-1}(t)+\rho_i(t))-(u_{i+1}(t)+u_{i}(t))(\rho_{i+1}(t)+\rho_{i}(t))}{4h}\nonumber \\
		= &\;0.
	\end{align*}
	
	(\textbf{III}). Finally, we show that $\frac{{\rm d}}{{\rm d}t}E(t) = 0.$ By observing \eqref{eq2.3}, we have
	\begin{align}
		\frac{{\rm d}}{{\rm d}t}E(t) =&\; \sum_{i=1}^{M}\big[ u_{i}(t)u'_{i}(t) + v_i(t)v'_{i}(t) + (1-2\Omega A)\rho_{i}(t)\rho'_{i}(t)\big]  \nonumber \\
		=&\; \sum_{i=1}^{M} \Big[u_i(t)\big(u'_{i}(t) - \frac{v'_{i+1}(t)-v'_{i-1}(t)}{2h}\big) + (1-2\Omega A)\rho_{i}(t)\rho'_{i}(t)\Big]     \nonumber \\
		=&\; \sum_{i=1}^{M} u_i(t)m'_{i}(t) + (1-2\Omega A)\sum_{i=1}^{M}\rho_{i}(t)\rho'_{i}(t).   \label{eq2.11}
	\end{align}
	Noticing that
	\begin{align}
		&\; \sum_{i=1}^{M} u_i(t)\Big(\frac{m_{i+1}(t)u_{i+1}(t)-m_{i-1}(t)u_{i-1}(t)}{2h}  + m_i(t)\frac{u_{i+1}(t)-u_{i-1}(t)}{2h}\Big)    \nonumber   \\
		= &\; \sum_{i=1}^{M}\Big(\frac{m_{i+1}(t)u_{i+1}(t)u_i(t) - m_{i-1}(t)u_{i-1}(t)u_i(t)}{2h}  + m_i(t)u_i(t)v_i(t) \Big)   \nonumber \\
		= &\; \sum_{i=1}^{M}\Big(\frac{m_{i}(t)u_{i}(t)u_{i-1}(t) - m_{i}(t)u_{i}(t)u_{i+1}(t)}{2h} + m_i(t)u_i(t)v_i(t) \Big)   \nonumber \\
		= &\; \sum_{i=1}^{M} \Big(-m_i(t)u_i(t)v_i(t)  + m_i(t)u_i(t)v_i(t) \Big)
		=  0.                        \label{eq2.12}
	\end{align}
	
	Eq. \eqref{eq2.11} follows from \eqref{eq2.1}, \eqref{eq2.2} and \eqref{eq2.12} that
	\begin{align*}
		&\; \sum_{i=1}^{M} u_i(t)m'_{i}(t) + (1-2\Omega A)\sum_{i=1}^{M}\rho_{i}(t)\rho'_{i}(t)   \nonumber \\
		= &\; - \frac{\sigma}{2h} \sum_{i=1}^{M} u_i(t)\Big[\big(m_{i+1}(t)u_{i+1}(t)-m_{i-1}(t)u_{i-1}(t)\big)
		+ m_i(t)\big(u_{i+1}(t)-u_{i-1}(t)\big)\Big]  \nonumber \\
		&\; -\frac{3(1-\sigma)}{4h}\sum_{i=1}^{M}\Big[u_i(t)(u_{i+1}^{2}(t)-u_{i-1}^{2}(t))\Big] + \frac{A}{2h}
		\sum_{i=1}^{M}\Big[u_i(t)(u_{i+1}(t)-u_{i-1}(t))\Big]  \nonumber \\
		&\; -\frac{\mu}{2h^3}\sum_{i=1}^{M}u_i(t)\Big[u_{i+2}(t)-2u_{i+1}(t)+2u_{i-1}(t)-u_{i-2}(t)\Big]\notag\\
		&\; -\frac{(1-2\Omega A)}{4h}\sum_{i=1}^{M}u_i(t)\Big[\rho_{i+1}^2(t)-\rho_{i-1}^2(t)\Big] \nonumber \\
		&\; + 2\Omega\sum_{i=1}^{M}u_i(t)\rho_{i}(t)\frac{(u_{i+1}(t)+u_i(t))(\rho_{i+1}(t)+\rho_i(t))
			-(u_{i-1}(t)+u_{i}(t))(\rho_{i-1}(t)+\rho_{i}(t))}{4h}  \nonumber \\
		&\; + (1-2\Omega A)\sum_{i=1}^{M}\rho_i(t)\frac{(u_{i-1}(t)+u_i(t))(\rho_{i-1}(t)+\rho_i(t))
			-(u_{i+1}(t)+u_{i}(t))(\rho_{i+1}(t)+\rho_{i}(t))}{4h}  \nonumber \\
		=&\; 0,
	\end{align*}
	which completes the proof.
\end{proof}

\section{Fully discrete difference scheme}\label{sec3}
\setcounter{equation}{0}
In the previous section, a semi-discrete finite difference scheme is established for the R2CH system \eqref{eq1.3}--\eqref{eq1.4}. Below, a new temporal discretization is presented to guarantee the long time accurate calculation without destroying the conservation properties of the original system. To this end, let $\tau = T/N$ the time stepsize and $t^n = n\tau$ $(n = 0,1,\cdots,N)$ the partition of time variable with a given integer $N$. Denoting $m_i^n = m(x_i,t^n)$, $u_i^n = u(x_i,t^n)$, $\rho_i^n = \rho(x_i,t^n)$, then \eqref{eq2.1}--\eqref{eq2.3} can be discretized implicitly by
\begin{align}
	&\frac{m_{i}^{*}-m_{i}^{n}}{\tau/2} + \frac{\sigma}{2h} \big[(m_{i+1}^{*}u_{i+1}^{*}-m_{i-1}^{*}u_{i-1}^{*}) + m_{i}^{*}(u_{i+1}^{*}-u_{i-1}^{*})\big]     \nonumber \\
	&\quad= -\frac{3(1-\sigma)}{4h}\big[(u_{i+1}^{*})^2 - (u_{i-1}^{*})^2\big] + \frac{A}{2h} \big(u_{i+1}^{*}-u_{i-1}^{*}\big)       \nonumber\\
	&\quad\quad- \frac{\mu}{2h^3} \big(u_{i+2}^{*}-2u_{i+1}^{*}+2u_{i-1}^{*}-u_{i-2}^{*}\big)-\frac{(1-2\Omega A)}{4h}\big[(\rho_{i+1}^{*})^2 - (\rho_{i-1}^{*})^2\big]  \nonumber \\
	&\quad\quad+\frac{\Omega}{2h}\rho_{i}^{*}\big[(u_{i+1}^{*}+u_{i}^{*})(\rho_{i+1}^{*}+\rho_{i}^{*}) - (u_{i-1}^{*}+u_{i}^{*})(\rho_{i-1}^{*}+\rho_{i}^{*})\big], \label{eq3.1}  \\
	&\frac{\rho_{i}^{*}-\rho_{i}^{n}}{\tau/2} + \frac{1}{4h}\big[(u_{i+1}^{*}+u_{i}^{*})(\rho_{i+1}^{*}+\rho_{i}^{*}) - (u_{i-1}^{*}+u_{i}^{*})(\rho_{i-1}^{*}+\rho_{i}^{*})\big] = 0, \label{eq3.2}
\end{align}
\begin{align}
  &m_i^n = u_i^n - \frac1{2h}\big(v_{i+1}^n-v_{i-1}^n\big),\quad v_i^n = \frac1{2h}(u_{i+1}^n-u_{i-1}^n),\label{eq3.2b}\\
	&u_i^{n+1} = 2u_{i}^{*} - u_{i}^{n},\quad m_{i}^{n+1} = 2m_{i}^{*} - m_{i}^{n}, \quad \rho_{i}^{n+1} = 2\rho_{i}^{*} - \rho_{i}^{n},  \label{eq3.3}
\end{align}
where $i=1,\cdots,M$ and $n=0,1,\cdots,N-1$.
Thus, the discrete mass, momentum and energy are defined as
\begin{align}
	&I_1^n = \sum_{i=1}^{M}\rho_{i}^n,  \label{eq2.5(b)} \\
	&I_2^n = \sum_{i=1}^{M}\left[u_i^n + \Omega (\rho_{i}^n)^2\right], \label{eq2.6(b)} \\
	&E^n = \frac{1}{2}\sum_{i=1}^{M}\bigg[(u_{i}^{n})^2 + \Big(\frac{u_{i+1}^n-u_{i-1}^n}{2h}\Big)^2 + (1-2\Omega A)(\rho_{i}^n)^2\bigg]. \label{eq2.7(b)}
\end{align}
The newly established discrete scheme will be shown to preserve the above three conserved quantities for the original system \eqref{eq1.1}--\eqref{eq1.2}. Then we have the following theorem.
\begin{theorem}\label{Theorem3_1}
	Consider the difference scheme \eqref{eq3.1}--\eqref{eq3.3} of the R2CH system \eqref{eq1.3}--\eqref{eq1.4}. Then the discrete mass, momentum, and energy with $\sigma = 1$ are conservative for the R2CH system \eqref{eq1.3}--\eqref{eq1.4}  in the following sense
	\begin{align*}
		I_{1}^{n+1} = I_{1}^{n}, \quad I_{2}^{n+1} = I_{2}^{n}, \quad E^{n+1} = E^{n}.   
	\end{align*}
\end{theorem}
\begin{proof}
	(\textbf{I}). Firstly, we show that $I_{1}^{n+1} = I_{1}^{n}$. Noticing  \eqref{eq3.2} and \eqref{eq3.3}, it readily knows that
	\begin{align}
		\frac{\rho_{i}^{n+1} - \rho_{i}^{n}}{\tau} + \frac{(u_{i+1}^{*}+u_{i}^{*})(\rho_{i+1}^{*}+\rho_{i}^{*}) - (u_{i-1}^{*}+u_{i}^{*})(\rho_{i-1}^{*}+\rho_{i}^{*})}{4h}  = 0.   \label{eq3.5}
	\end{align}
	Summing \eqref{eq3.5} with respect to $i$ from 1 to $M$ and utilizing the periodicity, we have
	\begin{align*}
		&\; \sum_{i=1}^{M} \Big[\frac{\rho_{i}^{n+1} - \rho_{i}^{n}}{\tau} + \frac{(u_{i+1}^{*}+u_{i}^{*})(\rho_{i+1}^{*}+\rho_{i}^{*}) - (u_{i-1}^{*}+u_{i}^{*})(\rho_{i-1}^{*}+\rho_{i}^{*})}{4h}\Big] \nonumber \\
		= &\; \sum_{i=1}^{M}\frac{\rho_{i}^{n+1} - \rho_{i}^{n}}{\tau}
		= \; \frac{I_{1}^{n+1} - I_{1}^{n}}{\tau}
		= \; 0,
	\end{align*}
	which implies $I_{1}^{n+1} = I_{1}^{n}$.
	
	(\textbf{II}). Next we show $I_{2}^{n+1} = I_{2}^{n}$.
	In combination of \eqref{eq3.1} and \eqref{eq3.2b}, and noticing the periodicity, we have
	\begin{align}
		&\;	\sum_{i=1}^{M}\frac{u_{i}^{n+1} - u_{i}^{n}}{\tau}  \nonumber \\
		= &\; \sum_{i=1}^{M}\frac{m_{i}^{n+1} - m_{i}^{n}}{\tau} + \sum_{i=1}^{M}\frac{(v_{i+1}^{n+1} - v_{i-1}^{n+1}) - (v_{i+1}^{n} - v_{i-1}^{n})}{2h\tau}   \nonumber \\
		= &\; \Omega \sum_{i = 1}^{M} \rho_{i}^{*}\frac{(u_{i+1}^{*}+u_{i}^{*})(\rho_{i+1}^{*}+\rho_{i}^{*}) - (u_{i-1}^{*}+u_{i}^{*})(\rho_{i-1}^{*}+\rho_{i}^{*})}{2h}.   \label{eq3.6}
	\end{align}
	Again using \eqref{eq3.5}, we find that
	\begin{align}
		&\;	\sum_{i=1}^{M}\frac{(\rho_{i}^{n+1})^2 - (\rho_{i}^{n})^2}{\tau}
		=2\sum_{i=1}^{M}\frac{\rho_{i}^{n+1} + \rho_{i}^{n}}{2}\cdot \frac{\rho_{i}^{n+1} - \rho_{i}^{n}}{\tau}  \nonumber  \\
		= &\; \sum_{i=1}^{M} \rho_{i}^{*}\frac{(u_{i-1}^{*}+u_{i}^{*})(\rho_{i-1}^{*}+\rho_{i}^{*}) - (u_{i+1}^{*}+u_{i}^{*})(\rho_{i+1}^{*}+\rho_{i}^{*})}{2h}.  \label{eq3.7}
	\end{align}
	Combining \eqref{eq3.6} with \eqref{eq3.7}, we have
	\begin{align*}
		\frac{I_{2}^{n+1} - I_{2}^{n}}{\tau}
		= \sum_{i=1}^{M}\frac{u_{i}^{n+1} - u_{i}^{n}}{\tau}  + \Omega \sum_{i=1}^{M}\frac{(\rho_{i}^{n+1})^2 - (\rho_{i}^{n})^2}{\tau}
		= 0,
	\end{align*}
	which indicates that $I_{2}^{n+1} = I_{2}^{n}$.
	
	(\textbf{III}). Finally, we prove $E^{n+1} = E^{n}$. Using \eqref{eq2.12}, we have
	\begin{align*}
		\sum_{i=1}^{M} u_i^*\Big(\frac{m_{i+1}^*u_{i+1}^*-m_{i-1}^*u_{i-1}^*}{2h}  + m_i^*\frac{u_{i+1}^*-u_{i-1}^*}{2h}\Big) = 0.        
	\end{align*}
	Using the periodicity, we have
	\begin{align}
		\sum_{i=1}^{M} \rho_{i}^{*} \frac{(u_{i-1}^{*}+u_{i}^{*})(\rho_{i-1}^{*}+\rho_{i}^{*}) -(u_{i+1}^{*} + u_{i}^{*})(\rho_{i+1}^{*} + \rho_{i}^{*})}{4h}
		= \sum_{i=1}^{M} (\rho_i^*)^2 \frac{u_{i-1}^* - u_{i+1}^*}{4h}.  \label{eq3.9}
	\end{align}
	Multiplying \eqref{eq3.1} with $u_{i}^{*}$ and summing over $i$ from 1 to $M$, and rearranging the corresponding result,
	we have
	\begin{align*}
		0 =&\; \sum_{i=1}^{M}\frac{m_{i}^{n+1} - m_{i}^{n}}{\tau}  u_{i}^{*} + (1-2\Omega A)\sum_{i=1}^{M}\frac{(\rho_{i+1}^{*})^2 - (\rho_{i-1}^{*})^2}{4h}  u_{i}^{*}   \nonumber   \\
		= &\; \sum_{i=1}^{M} \bigg(\frac{u_{i}^{n+1} - u_{i}^{n}}{\tau}u_{i}^{*} - \frac{v_{i+1}^{n+1}-v_{i-1}^{n+1}-v_{i+1}^{n}+v_{i-1}^n}{2h\tau}u_{i}^{*}
		\bigg) + (1-2\Omega A) \sum_{i=1}^{M}(\rho_{i}^{*})^2 \frac{u_{i-1}^{*} - u_{i+1}^{*}}{4h}   \nonumber \\
		= &\; \sum_{i=1}^{M} \bigg[\frac{u_{i}^{n+1} - u_{i}^{n}}{\tau}u_{i}^{*} + \Big(\frac{v_{i}^{n+1}-v_i^{n}}{\tau}\Big) \Big(\frac{u_{i+1}^* - u_{i-1}^*}{2h}\Big)\bigg]   \nonumber \\
		&\; +  (1-2\Omega A) \sum_{i=1}^{M} \rho_{i}^{*} \frac{(u_{i-1}^{*}+u_{i}^{*})(\rho_{i-1}^{*}+\rho_{i}^{*}) -(u_{i+1}^{*} + u_{i}^{*})(\rho_{i+1}^{*} + \rho_{i}^{*})}{4h}\notag\\
		= &\; \sum_{i=1}^{M} \bigg[\Big( \frac{u_{i}^{n+1} - u_{i}^{n}}{\tau}\Big) \Big(\frac{u_{i}^{n+1} + u_{i}^{n}}{2}\Big) + \Big(\frac{v_{i}^{n+1} - v_{i}^{n}}{\tau}\Big) \Big( \frac{v_{i}^{n+1} + v_{i}^{n}}{2}\Big) \bigg]     \nonumber \\
		&\; +  (1-2\Omega A) \sum_{i=1}^{M} \frac{\rho_i^{n+1}+\rho_i^{n}}{2} \cdot \frac{\rho_i^{n+1}-\rho_i^n}{\tau}\nonumber\\
		 = & \; \frac{E^{n+1} - E^{n}}{\tau},
	\end{align*}
	where the periodicity, \eqref{eq3.3} and $\sigma = 1$ are used in the first equality,  the periodicity and \eqref{eq3.2b} are used in the second equality, the periodicity and \eqref{eq3.9} are used in the third equality, and \eqref{eq3.5} is used in the penultimate equality. This completes the proof.
\end{proof}
\section{Algorithm implementation}\label{sec4}
\setcounter{equation}{0}
Denote $\boldsymbol{u} = (u_1,u_2,\cdots,u_M)^T$, $\boldsymbol{m} = (m_1,m_2,\cdots,m_M)^T$ and $\boldsymbol{\rho} = (\rho_1,\rho_2,\cdots,\rho_M)^T$. Recalling \eqref{eq3.2b},
we have a linear system of equations $\boldsymbol{m} = \boldsymbol{B}\textbf{u}$, where $\boldsymbol{B}$ is a
symmetric circulant matrix defined by
$$\textbf{B} = {\rm \textbf{circ}}(c[0],c[1],\cdots,c[M-1])$$
with $c[0] = 1+\frac1{2h^2}$, $c[1] = 0$, $c[2] = -\frac{1}{4h^2}$, $c[3]= \cdots = c[M-3]=0$, $c[M-2]=-\frac{1}{4h^2}$ and $c[M-1] = 0$.
Obviously, $\boldsymbol{B}$ is a dominant diagonal matrix, which allows us to express the vector $\boldsymbol{u}$ in terms of $\boldsymbol{m}$ as $\boldsymbol{u} = \boldsymbol{B}^{-1}\boldsymbol{m}$.
To avoid ambiguity we let $(\boldsymbol{B}^{-1}\boldsymbol{m}^*)_{i}$ denote the $i$th element of the vector.
Consequently, \eqref{eq3.1}--\eqref{eq3.3} can be rewritten equivalently in terms of $m^*$ as follows
\begin{align}
	&\frac{m_{i}^{*}-m_{i}^{n}}{\tau/2} + \frac{\sigma}{2h} \big[(m_{i+1}^{*}(\boldsymbol{B}^{-1}\boldsymbol{m}^*)_{i+1}-m_{i-1}^{*}(\boldsymbol{B}^{-1}\boldsymbol{m}^*)_{i-1}) + m_{i}^{*}((\boldsymbol{B}^{-1}\boldsymbol{m}^*)_{i+1}-(\boldsymbol{B}^{-1}\boldsymbol{m}^*)_{i-1})\big]    \nonumber \\
	&\quad= -\frac{3(1-\sigma)}{4h}\big[(\boldsymbol{B}^{-1}\boldsymbol{m}^*)_{i+1}^2 - (\boldsymbol{B}^{-1}\boldsymbol{m}^*)_{i-1}^2\big] + \frac{A}{2h} \big[(\boldsymbol{B}^{-1}\boldsymbol{m}^*)_{i+1}-(\boldsymbol{B}^{-1}\boldsymbol{m}^*)_{i-1}\big]     \nonumber \\
	&\quad\quad - \frac{\mu}{2h^3} \big[(\boldsymbol{B}^{-1}\boldsymbol{m}^*)_{i+2}-2(\boldsymbol{B}^{-1}\boldsymbol{m}^*)_{i+1}
	+2(\boldsymbol{B}^{-1}\boldsymbol{m}^*)_{i-1}-(\boldsymbol{B}^{-1}\boldsymbol{m}^*)_{i-2}\big]  \nonumber \\
	&\quad\quad -\frac{(1-2\Omega A)}{4h}\big[(\rho_{i+1}^{*})^2 - (\rho_{i-1}^{*})^2\big]   + \frac{\Omega}{2h} \rho_{i}^{*}\big[((\boldsymbol{B}^{-1}\boldsymbol{m}^*)_{i+1}+(\boldsymbol{B}^{-1}\boldsymbol{m}^*)_{i})(\rho_{i+1}^{*}+\rho_{i}^{*}) \notag\\
	&\quad\quad- ((\boldsymbol{B}^{-1}\boldsymbol{m}^*)_{i-1}+(\boldsymbol{B}^{-1}\boldsymbol{m}^*)_{i})(\rho_{i-1}^{*}+\rho_{i}^{*})\big],\label{eq4.28} \\
	&\frac{\rho_{i}^{*}-\rho_{i}^{n}}{\tau/2} + \frac{1}{4h} \big[((\boldsymbol{B}^{-1}\boldsymbol{m}^*)_{i+1}+(\boldsymbol{B}^{-1}\boldsymbol{m}^*)_{i})(\rho_{i+1}^{*}+\rho_{i}^{*}) \notag
	\\&\quad\quad- ((\boldsymbol{B}^{-1}\boldsymbol{m}^*)_{i-1}+(\boldsymbol{B}^{-1}\boldsymbol{m}^*)_{i})(\rho_{i-1}^{*}+\rho_{i}^{*})\big] = 0, \label{eq4.29} \\
	&\; m_{i}^{n+1} = 2m_{i}^{*} - m_{i}^{n}, \quad \rho_{i}^{n+1} = 2\rho_{i}^{*} - \rho_{i}^{n}.\label{eq4.30}
\end{align}
At present, we have a nonlinear system of equations with variables $\boldsymbol{m}^*$ and $\boldsymbol{\rho}^*$ only, which can be solved by a fixed point iteration method.

The algorithm flow chart of the scheme \eqref{eq4.28}--\eqref{eq4.30} is given as follows. Specifically, to solve the solutions at the $(n+1)$th time level, we suppose that $\boldsymbol{u}^{n,l}  =
\Big(u_1^{n,l},u_2^{n,l},\cdots,u_M^{n,l}\Big)^T$, $\boldsymbol{\rho}^{n,l}  =
\Big(\rho_1^{n,l},\rho_2^{n,l},\cdots,\rho_M^{n,l}\Big)^T$ and $\boldsymbol{m}^{n,l}  =
\Big(m_1^{n,l},m_2^{n,l},\cdots,m_M^{n,l}\Big)^T$ have been determined, then {\bf Algorithm 1} is utilized to solve the scheme \eqref{eq4.28}--\eqref{eq4.30}.
\begin{center}
	{\bf Algorithm 1}: The iterative process for solving \eqref{eq4.28}--\eqref{eq4.30} \\
	\begin{tabular}{rll}\hline
		1. & Set the tolerance error \emph{tol}. \cr
		2. & Give the initial guess $\boldsymbol{u}^{*,l} := \boldsymbol{u}^{n,l}$.\cr
		3. & For $n = 1,2,\cdots,N$ Do:\cr
		4. & \qquad Compute $\boldsymbol{\rho}^{*,l+1}$, $\boldsymbol{m}^{*,l+1}$, then obtain $\boldsymbol{u}^{*,l+1}$ \cr
		5. & EndDo   \cr
		6. & If $\Vert \boldsymbol{u}^{*,l} - \boldsymbol{u}^{*,l+1} \Vert_{\infty} > \emph{tol}$, $\boldsymbol{u}^{*,l} := \boldsymbol{u}^{*,l+1}$, GoTo 2;  \cr
		7. & else, $\boldsymbol{u}^{n+1} := 2 \boldsymbol {u}^{*,l+1} - \boldsymbol{u}^{n,l}$; $\boldsymbol{\rho}^{n+1} := 2 \boldsymbol{\rho}^{*,l+1} - \boldsymbol{\rho}^{n,l}$, $\boldsymbol{m}^{n+1} = \boldsymbol{Bu}^{n+1}$ .  \cr
		\hline
	\end{tabular}
\end{center}

\begin{remark}
	Provided the solution of the system \eqref{eq1.1}--\eqref{eq1.2} is less regular, then numerical oscillation will occur during the calculation. To eliminate this phenomenon, an efficient strategy is utilized  by adding local numerical viscosity. The following two viscous terms
	\begin{align*}
		&\; R_i^{u} = \frac{\epsilon_{i}^{u}}{2h}(u_{i+1}^{*}-2u_i^{*}+u_{i-1}^{*}),  \quad \text{and} \quad \; R_i^{\rho} = \frac{\epsilon_{i}^{\rho}}{2h}(\rho_{i+1}^{*}-2\rho_i^{*}+\rho_{i-1}^{*}),
	\end{align*}
	where
	\begin{align*}
		\epsilon_i^{u} =  \left\{
		\begin{array}{ll}
			1, \quad |u_{i+1}^{*} - 2u_{i}^{*} + u_{i-1}^{*}| \geqslant \varepsilon h, \\
			0, \quad otherwise,
		\end{array}
		\right.
		\quad\text{and}   \quad
		\epsilon_i^{\rho} =  \left\{
		\begin{array}{ll}
			1, \quad |\rho_{i+1}^{*} - 2\rho_{i}^{*} + \rho_{i-1}^{*}| \geqslant \varepsilon h, \\
			0, \quad otherwise.
		\end{array}
		\right.
	\end{align*}
	are added to the right-hand side of \eqref{eq3.1} and \eqref{eq3.2}, respectively.
	The above factor $\varepsilon$ is a threshold (usually very small, e.g., $\varepsilon=10^{-5}$) which can determine where the slope of the solution becomes unbounded.
	It is worth mentioning that the factor $\varepsilon$ varies from case to case in the simulation of the discontinuous solutions in the section \ref{sec5_2} below.
\end{remark}

\section{Numerical tests}\label{sec5}
\setcounter{equation}{0}
In this section, several examples for the benchmark problems are provided to verify the convergence, conservation laws and the performance of the proposed scheme \eqref{eq3.1}--\eqref{eq3.3}.
To test the convergence orders, the posterior error estimate is utilized in temporal direction and spatial direction.
To be more specific, for sufficient small $\tau$, we denote
\begin{align*}
	\|{\rm F}_{u}(h)\|_{\infty} = \max\limits_{1\leqslant i \leqslant M, 1 \leqslant n \leqslant N} \big| u_{i}^{n}(h,\tau) - u_{2i}^{n}(h/2,\tau)\big|, \quad{\rm Ord}_{\infty}^{h} = \log_2\frac{\|{\rm F}_{u}(h)\|_{\infty}}{\|{\rm F}_{u}(h/2)\|_{\infty}},
\end{align*}
and for sufficient small $h$, denote
\begin{align*}
	\|{\rm G}_{u}(\tau)\|_{\infty} = \max\limits_{1\leqslant i \leqslant M, 1 \leqslant n \leqslant N}\big| u_{i}^{n}(h,\tau) - u_{i}^{2n}(h,\tau/2)\big|, \quad {\rm Ord}_{\infty}^{\tau} = \log_2\frac{\|{\rm G}_{u}(\tau)\|_{\infty}}{\|{\rm G}_{u}(\tau/2)\|_{\infty}}.
\end{align*}

\subsection{Part I: smooth initial data}\label{sec5_1}
\begin{example}[\textbf{Dam-break problem} \cite{LP2016,ZLZ2022}] The R2CH system with the smooth initial data  \label{Exam5.1}
	\begin{align*}
		u(x,0) = 0, \quad \rho(x,0) = 1 + \tanh(x+a) - \tanh(x-a)  
	\end{align*}
	are considered, where $a$ is a dam-breaking parameter. The exact solution for the problem is unknown.
\end{example}
To test the convergence orders and conservation, we consider the following four cases:
\begin{itemize}
	\item \textbf{Case} (\uppercase\expandafter{\romannumeral1}). $a=0.1$, $A = 0$, $\mu = 0$, $\sigma = 1$, $\Omega = 0$, on the domain $[-6,6] \times[0,20]$.
	\item \textbf{Case} (\uppercase\expandafter{\romannumeral2}). $a=4$, $A = 0$, $\mu = 0$, $\sigma = 1$, $\Omega = 0$, on the domain $[-12\pi,12\pi]\times[0,2]$.
	\item \textbf{Case} (\uppercase\expandafter{\romannumeral3}). $a=0.1$, $A = 0.1$, $\mu = 0.1$, $\sigma = 1$, $\Omega = 73 \times10^{-6}$, on the domain $[-8,8]\times[0,1]$.
	\item  \textbf{Case} (\uppercase\expandafter{\romannumeral4}). $a=4$, $A = 1$, $\mu = 1$, $\sigma = 1$, $\Omega = 73 \times 10^{-6}$, on the domain $[-12\pi,12\pi]\times[0,2]$.
\end{itemize}

\textbf{(Convergence)} We first verify the convergence orders for the above four cases.
The temporal convergence orders of velocity $u(x,t)$ and height $\rho(x,t)$ are respectively listed in Table \ref{table1} and Table \ref{table2}, which show  the second-order convergence by fixing spatial grid $M=100$ and refining $N$. Similarly, Tables \ref{table3}--\ref{table4} illustrate the second-order spatial convergence order by fixing  $N=4000$ and refining $M$.

\textbf{(Conservation)} The conservations in Theorem \ref{Theorem3_1} are demonstrated in Table \ref{table5}, which shows our scheme indeed preserves three conserved quantities defined in \eqref{eq2.5(b)}--\eqref{eq2.7(b)}. It is worth noting that the total momentum in \textbf{Cases} (\uppercase\expandafter{\romannumeral1}) and (\uppercase\expandafter{\romannumeral2}) is zero, while it does not vanish in \textbf{Cases} (\uppercase\expandafter{\romannumeral3}) and (\uppercase\expandafter{\romannumeral4}) since $\Omega$ is nonzero in \eqref{eq2.6(b)}.

\textbf{(Portraits of solutions at different instants)} In addition, an applicable reference solution is necessary to evaluate the behavior of the numerical solution. To generate a reference solution, we take a refined grid $M=3200$ in space. Figures \ref{fig1}--\ref{fig2} show the behaviors of the predicted velocities $u(x,t)$ and heights $\rho(x,t)$ at time $t=2$ for \textbf{Cases} (\uppercase\expandafter{\romannumeral2}) and  (\uppercase\expandafter{\romannumeral4}). We observe that the solutions in \textbf{Case} (\uppercase\expandafter{\romannumeral2}) are symmetric, while the solutions in  \textbf{Case} (\uppercase\expandafter{\romannumeral4}) are asymmetric. Actually the symmetry depends heavily on the selection of parameters. It can be clearly seen that the conservative scheme \eqref{eq3.1}--\eqref{eq3.3} performs well in depicting the dam break solution of the R2CH system \eqref{eq1.3}--\eqref{eq1.4} even using a relatively rough grid. Figures \ref{fig3}--\ref{fig4} show the evolution of the predicted dam break solutions of the R2CH system over a long period of time. We find that even up to $t=50$, the solution in Figure \ref{fig3} still stretches leftward and rightward in a symmetric form, while the evolution trend of the solutions in Figure \ref{fig4} is clearly less symmetric.

\begin{table}[htb!]
	\begin{center}
		\renewcommand{\arraystretch}{1.25}
		\tabcolsep 0pt
		\caption{Numerical errors, temporal convergence orders of velocity $u(x,t)$
			of the scheme \eqref{eq3.1}--\eqref{eq3.3}
			with $M = 100$ for \textbf{Case} (\uppercase\expandafter{\romannumeral1}), \textbf{Case} (\uppercase\expandafter{\romannumeral2}), \textbf{Case} (\uppercase\expandafter{\romannumeral3}) and \textbf{Case} (\uppercase\expandafter{\romannumeral4}). }   \label{table1}
		\def\temptablewidth{0.88\textwidth}
		\rule{\temptablewidth}{1pt}
		{\footnotesize
			\begin{tabular*}{\temptablewidth}{@{\extracolsep{\fill}}ccccccccc}
				&\multicolumn{2}{c}{\textbf{Case} (\uppercase\expandafter{\romannumeral1})}
				&\multicolumn{2}{c}{\textbf{Case} (\uppercase\expandafter{\romannumeral2})}
				&\multicolumn{2}{c}{\textbf{Case} (\uppercase\expandafter{\romannumeral3})}
				&\multicolumn{2}{c}{\textbf{Case} (\uppercase\expandafter{\romannumeral4})}    \\
				\cline{2-3}\cline{4-5} \cline{6-7} \cline{8-9}
				$\quad N$ &$\|{\rm F}_{u}(\tau)\|_\infty$&${ {\rm Ord}_{\infty}^{\tau}}$&  $\|{\rm F}_{u}(\tau)\|_\infty$ & ${{\rm Ord}_{\infty}^{\tau}} $  &$\|{\rm F}_{u}(\tau)\|_\infty$ &${ {\rm Ord}_{\infty}^{\tau}} $  &$\|{\rm F}_{u}(\tau)\|_{\infty}$&${ {\rm Ord}_{\infty}^{\tau}} $ \\
				\hline
				$\quad 100 $ &$9.0981{\rm e}-04$ & $*$     &$3.3262{\rm e}-04$ &$*$      &$1.7801{\rm e}-07$ &$*$      &$8.3334{\rm e}-04$ &$*$       \quad\\
				$\quad 200 $ &$2.2919{\rm e}-04$ &$1.9890$ &$8.3267{\rm e}-05$ &$1.9982$ &$4.4503{\rm e}-08$ &$2.0000$ &$2.0881{\rm e}-04$ &$1.9967$   \quad\\
				$\quad 400 $ &$5.7402{\rm e}-05$ &$1.9973$ &$2.0827{\rm e}-05$ &$1.9999$ &$1.1126{\rm e}-08$ &$2.0000$ &$5.2245{\rm e}-05$ &$1.9988$   \quad\\
				$\quad 800 $ & $1.4357{\rm e}-05$&$1.9993$ &$5.1994{\rm e}-06$ &$1.9997$ &$2.7815{\rm e}-09$ &$2.0000$ &$1.3056{\rm e}-05$ &$2.0006$   \quad\\
				$\quad 1600$ & $3.5897{\rm e}-06$&$1.9998$ &$1.3006{\rm e}-06$ &$1.9995$ &$6.9536{\rm e}-10$ &$2.0000$ &$3.2636{\rm e}-06$ &$ 2.0001$  \quad\\
		\end{tabular*}}
		\rule{\temptablewidth}{1pt}
	\end{center}
\end{table}

\begin{table}[htb!]
	\begin{center}
		\renewcommand{\arraystretch}{1.25}
		\tabcolsep 0pt
		\caption{Numerical errors, temporal convergence orders in height  $\rho(x,t)$
			of the scheme \eqref{eq3.1}--\eqref{eq3.3}
			with $M=100$ for \textbf{Case} (\uppercase\expandafter{\romannumeral1}), \textbf{Case} (\uppercase\expandafter{\romannumeral2}), \textbf{Case} (\uppercase\expandafter{\romannumeral3}) and \textbf{Case} (\uppercase\expandafter{\romannumeral4}). }  \label{table2}
		\def\temptablewidth{0.88\textwidth}
		\rule{\temptablewidth}{1pt}
		{\footnotesize
			\begin{tabular*}{\temptablewidth}{@{\extracolsep{\fill}}ccccccccc}
				&\multicolumn{2}{c}{\textbf{Case} (\uppercase\expandafter{\romannumeral1})}
				&\multicolumn{2}{c}{\textbf{Case} (\uppercase\expandafter{\romannumeral2})}
				&\multicolumn{2}{c}{\textbf{Case} (\uppercase\expandafter{\romannumeral3})}
				&\multicolumn{2}{c}{\textbf{Case} (\uppercase\expandafter{\romannumeral4})}\\
				\cline{2-3}	\cline{4-5} \cline{6-7} \cline{8-9}
				$\quad N$ &$\|{\rm F}_{\rho}(\tau)\|_\infty$  &${{\rm Ord}_{\infty}^{\tau}}$ &$\|{\rm F}_{\rho}(\tau)\|_\infty$ &${{\rm Ord}_{\infty}^{\tau}}$ &$\|{\rm F}_{\rho}(\tau)\|_\infty$ &${ {\rm Ord}_{\infty}^{\tau}} $  &$\|{\rm F}_{\rho}(\tau)\|_{\infty}$&${ {\rm Ord}_{\infty}^{\tau}} $ \\
				\hline
				$\quad100$ &$1.4094{\rm e}-03$           &$*$      &$3.0830{\rm e}-04$ &$*$      &$5.2605{\rm e}-07$ &$*$      &$3.7110{\rm e}-04$ & $*$     \quad\\
				$\quad200$ &$3.5345{\rm e}-04$ &$1.9955$ &$7.7121{\rm e}-05$ &$1.9991$ &$1.3151{\rm e}-07$ &$2.0000$ &$9.2969{\rm e}-05$ &$1.9970$  \quad\\
				$\quad400$ &$8.8512{\rm e}-05$ &$1.9976$ &$1.9283{\rm e}-05$ &$1.9997$ &$3.2879{\rm e}-08$ &$2.0000$ &$2.3258{\rm e}-05$ &$1.9990$  \quad\\
				$\quad800$ &$2.2133{\rm e}-05$ &$1.9997$ &$4.8210{\rm e}-06$ &$1.9999$ &$8.2193{\rm e}-09$ &$2.0001$ &$5.8135{\rm e}-06$ &$2.0002$ \quad\\
				$\quad1600$&$5.5335{\rm e}-06$ &$1.9999$ &$1.2060{\rm e}-06$ &$1.9990$ &$2.0549{\rm e}-09$ &$2.0000$ &$1.4532{\rm e}-06$ &$2.0001$ \quad\\
		\end{tabular*}}
		\rule{\temptablewidth}{1pt}
	\end{center}
\end{table}

\begin{table}[htb!]
	\begin{center}
		\renewcommand{\arraystretch}{1.25}
		\tabcolsep 0pt
		\caption{Numerical errors, spatial convergence orders in velocity  $u(x,t)$
			of the scheme \eqref{eq3.1}--\eqref{eq3.3}
			with $N=4000$ for \textbf{Case} (\uppercase\expandafter{\romannumeral1}), \textbf{Case} (\uppercase\expandafter{\romannumeral2}), \textbf{Case} (\uppercase\expandafter{\romannumeral3}) and \textbf{Case} (\uppercase\expandafter{\romannumeral4}). }   \label{table3}
		\def\temptablewidth{0.88\textwidth}
		\rule{\temptablewidth}{1pt}
		{\footnotesize
			\begin{tabular*}{\temptablewidth}{@{\extracolsep{\fill}}ccccccccc}
				&\multicolumn{2}{c}{\textbf{Case} (\uppercase\expandafter{\romannumeral1})}
				&\multicolumn{2}{c}{\textbf{Case} (\uppercase\expandafter{\romannumeral2})}
				&\multicolumn{2}{c}{\textbf{Case} (\uppercase\expandafter{\romannumeral3})}
				&\multicolumn{2}{c}{\textbf{Case} (\uppercase\expandafter{\romannumeral4})}    \\
				\cline{2-3}   \cline{4-5}  \cline{6-7} \cline{8-9}
				$\quad M$ &$\|{\rm F}_{u}(h)\|_\infty$  &${{\rm Ord}_{\infty}^{h}}$ &$\|{\rm F}_{u}(h)\|_\infty$& ${{\rm Ord}_{\infty}^{h}}$ &$\|{\rm F}_{u}(h)\|_\infty$ &${ {\rm Ord}_{\infty}^{h}} $  &$\|{\rm F}_{u}(h)\|_{\infty}$&${ {\rm Ord}_{\infty}^{h}} $ \\
				\hline
				$\quad 100$ &$6.2259{\rm e}-04$  &$*$      &$1.2876{\rm e}-01$  &$*$      &$1.9975{\rm e}-04$  &$*$      &$1.1547{\rm e}-01$  &$*$     \quad\\
				$\quad 200$ &$1.5547{\rm e}-04$  &$2.0016$ &$3.7311{\rm e}-02$  &$1.7871$ &$4.9396{\rm e}-05$  &$2.0157$ &$4.6363{\rm e}-02$  &$1.3165$   \quad\\
				$\quad 400$ &$3.8856{\rm e}-05$  &$2.0004$ &$1.2058{\rm e}-02$  &$1.6296$ &$1.2327{\rm e}-05$  &$2.0025$ &$1.4034{\rm e}-02$  &$1.7240$  \quad\\
				$\quad 800$ &$9.7152{\rm e}-06$  &$1.9998$ &$3.3317{\rm e}-03$  &$1.8557$ &$3.0826{\rm e}-06$  &$1.9997$ &$3.7075{\rm e}-03$  &$1.9205$ \quad\\
				$\quad 1600$&$2.4288{\rm e}-06$  &$2.0000$ &$8.5597{\rm e}-04$  &$1.9606$ &$7.7051{\rm e}-07$  &$2.0003$ &$9.3181{\rm e}-04$  &$1.9923$ \quad\\
		\end{tabular*}}
		\rule{\temptablewidth}{1pt}
	\end{center}
\end{table}

\begin{table}[htb!]
	\begin{center}
		\renewcommand{\arraystretch}{1.25}
		\tabcolsep 0pt
		\caption{Numerical errors, spatial convergence orders in height  $\rho(x,t)$
			of the scheme \eqref{eq3.1}--\eqref{eq3.3}
			with $N = 4000$ for \textbf{Case} (\uppercase\expandafter{\romannumeral1}), \textbf{Case} (\uppercase\expandafter{\romannumeral2}), \textbf{Case} (\uppercase\expandafter{\romannumeral3}) and \textbf{Case} (\uppercase\expandafter{\romannumeral4}). }   \label{table4}
		\def\temptablewidth{0.88\textwidth}
		\rule{\temptablewidth}{1pt}
		{\footnotesize
			\begin{tabular*}{\temptablewidth}{@{\extracolsep{\fill}}ccccccccc}
				&\multicolumn{2}{c}{\textbf{Case} (\uppercase\expandafter{\romannumeral1})}
				&\multicolumn{2}{c}{\textbf{Case} (\uppercase\expandafter{\romannumeral2})}
				&\multicolumn{2}{c}{\textbf{Case} (\uppercase\expandafter{\romannumeral3})}
				&\multicolumn{2}{c}{\textbf{Case} (\uppercase\expandafter{\romannumeral4})}    \\
				\cline{2-3} \cline{4-5} \cline{6-7} \cline{8-9}
				$\quad M$&$\|{\rm F}_{\rho}(h)\|_\infty$  &${{\rm Ord}_{\infty}^{h}}$    &$\|{\rm F}_{\rho}(h)\|_\infty$ &${{\rm Ord}_{\infty}^{h}}$  &$\|{\rm F}_{\rho}(h)\|_\infty$ &${ {\rm Ord}_{\infty}^{h}}$  &$\|{\rm F}_{\rho}(h)\|_{\infty}$&${ {\rm Ord}_{\infty}^{h}} $ \\
				\hline
				$\quad 100$&$8.9486{\rm e}-04$  &$*$      &$1.4496{\rm e}-01$  &$*$       &$2.4458{\rm e}-04$   &$*$       &$2.4052{\rm e}-01$  &$*$      \quad\\
				$\quad 200$&$2.2263{\rm e}-04$  &$2.0070$ &$8.6727{\rm e}-02$  &$0.7411$  &$6.0340{\rm e}-05$   &$2.0191$  &$1.8170{\rm e}-01$  &$0.4046$   \quad\\
				$\quad 400$&$5.5586{\rm e}-05$  &$2.0018$ &$3.4910{\rm e}-02$  &$1.3129$  &$1.5037{\rm e}-05$   &$2.0046$  &$7.7859{\rm e}-02 $ &$1.2226$  \quad\\
				$\quad 800$&$1.3894{\rm e}-05$  &$2.0003$ &$1.0946{\rm e}-02$  &$1.6732$  &$3.7620{\rm e}-06$   &$1.9990$  &$2.5178{\rm e}-02$  &$1.6289$ \quad\\
				$\quad 1600$&$3.4736{\rm e}-06$ &$1.9999$ &$2.9327{\rm e}-03$  &$1.9001$  &$9.4032{\rm e}-07$   &$2.0003$  &$7.0786{\rm e}-03$  &$1.8305$ \quad\\
		\end{tabular*}}
		\rule{\temptablewidth}{1pt}
	\end{center}
\end{table}
\begin{table}[tbh!]
	\begin{center}
		\renewcommand{\arraystretch}{1.10}
		\tabcolsep 0pt
		\caption{Numerical conserved quantities $E^n$, $I_1^n$ and $I_2^n$ at time $t^n$.}
		\def\temptablewidth{0.8\textwidth}
		\rule{\temptablewidth}{1pt}            \label{table5}
		{\footnotesize
			\begin{tabular*}{\temptablewidth}{@{\extracolsep{\fill}}cccc}
				&\multicolumn{3}{c}{\textbf{Case} (I), $h=0.06$, $\tau = 0.01$}  \\ 
				\cline{2-4}
				$\quad t_n\quad$  &$E^n$  &$I_1^n$ & $I_2^n$\\
				\hline
				$\quad0\quad$   &$107.1098478543294\quad$   &$206.6665840981688\quad$  &$0$\;\\
				$\quad2\quad$   &$107.1098478543293\quad$   &$206.6665840981688\quad$  &$-0.000000000000005\;$\\
				$\quad4\quad$   &$107.1098478543292\quad$   &$206.6665840981687\quad$  &$-0.000000000000002\;$\\
				$\quad6\quad$   &$107.1098478543292\quad$   &$206.6665840981687\quad$  &$~~~0.000000000000022\;$\\
				$\quad8\quad$   &$107.1098478543292\quad$   &$206.6665840981686\quad$  &$~~~0.000000000000013\;$\\
				$\quad10\quad$  &$107.1098478543292\quad$   &$206.6665840981686\quad$  &$~~~0.000000000000017\;$\\
				\hline
				&\multicolumn{3}{c}{\textbf{Case} (II), $h=0.5$, $\tau = 0.005$}\\ 
				\cline{2-4}
				$\quad t_n\quad$    &  $E^n$ & $I_1^n$  & $I_2^n$ \\
				\hline
				$\quad0\quad $   &$134.6831098503824\quad$    &$181.8309885794262\quad$    &$~~~0$     \;       \\
				$\quad2\quad$    &$134.6831098503929\quad$    &$181.8309885794291\quad$    &$~~~0.000000000000020\;$ \\
				$\quad4\quad$    &$134.6831098503961\quad$    &$181.8309885794310\quad$    &$~~~0.000000000000039\;$  \\
				$\quad6\quad$    &$134.6831098503970\quad$    &$181.8309885794320\quad$    &$~~~0.000000000000074\;$ \\
				$\quad8\quad$    &$134.6831098503970\quad$    &$181.8309885794321\quad$    &$~~~0.000000000000009\;$ \\
				$\quad10\quad$   &$134.6831098503968\quad$    &$181.8309885794320\quad$    &$~~~0.000000000000006\;$ \\
				\hline
				&\multicolumn{3}{c}{\textbf{Case} (III), $h=0.16$, $\tau = 0.0025$} \\ 
				\cline{2-4}
				$\quad t_n\quad$    &  $E^n$ & $I_1^n$  & $I_2^n$ \\
				\hline
				$\quad0\quad$    &$52.66545441047811\quad$  &$102.4999994287402\quad$     &$~~~0.007689268607251\;$      \\
				$\quad1\quad$    &$52.66545441059597\quad$  &$102.4999994287401\quad$     &$~~~0.007689268554501\;$      \\
				$\quad2\quad$    &$52.66545441119767\quad$  &$102.4999994287401\quad$     &$~~~0.007689262987543\;$      \\
				$\quad3\quad$    &$52.66545441165016\quad$  &$102.4999994287401\quad$     &$~~~0.007689144120514\;$      \\
				$\quad4\quad$    &$52.66545441384896\quad$  &$102.4999994287401\quad$     &$~~~0.007687960991156\;$      \\
				$\quad5\quad$    &$52.66545444157879\quad$  &$102.4999994287402\quad$     &$~~~0.007680802034650\;$      \\
				\hline
				& \multicolumn{3}{c}{\textbf{Case} (IV), $h=0.45$, $\tau = 0.0025$} \\ 
				\cline{2-4}
				$\quad t_n\quad$    &  $E^n$ & $I_1^n$  & $I_2^n$ \\
				\hline
				$\quad0\quad$    &$152.6185732846562\quad$   &$206.0751204356079\quad$     &$~~~0.022285565392107\;$      \\
				$\quad1\quad$    &$152.6185714348163\quad$   &$206.0751204356105\quad$     &$~~~0.022285565391962\;$      \\
				$\quad2\quad$    &$152.6185618398274\quad$   &$206.0751204356114\quad$     &$~~~0.022285565214827\;$      \\
				$\quad3\quad$    &$152.6185527930804\quad$   &$206.0751204356117\quad$     &$~~~0.022285564293643\;$      \\
				$\quad4\quad$    &$152.6185466151580\quad$   &$206.0751204356121\quad$     &$~~~0.022285563402993\;$      \\
				$\quad5\quad$    &$152.6185421518980\quad$   &$206.0751204356124\quad$     &$~~~0.022285563022157\;$      \\
		\end{tabular*}}
		\rule{\temptablewidth}{1pt}
	\end{center}
\end{table}

\begin{figure}[htbp]
	\vspace{-8mm}
	\subfigure[$M=200$]{\centering
		\includegraphics[width=0.35\textwidth]{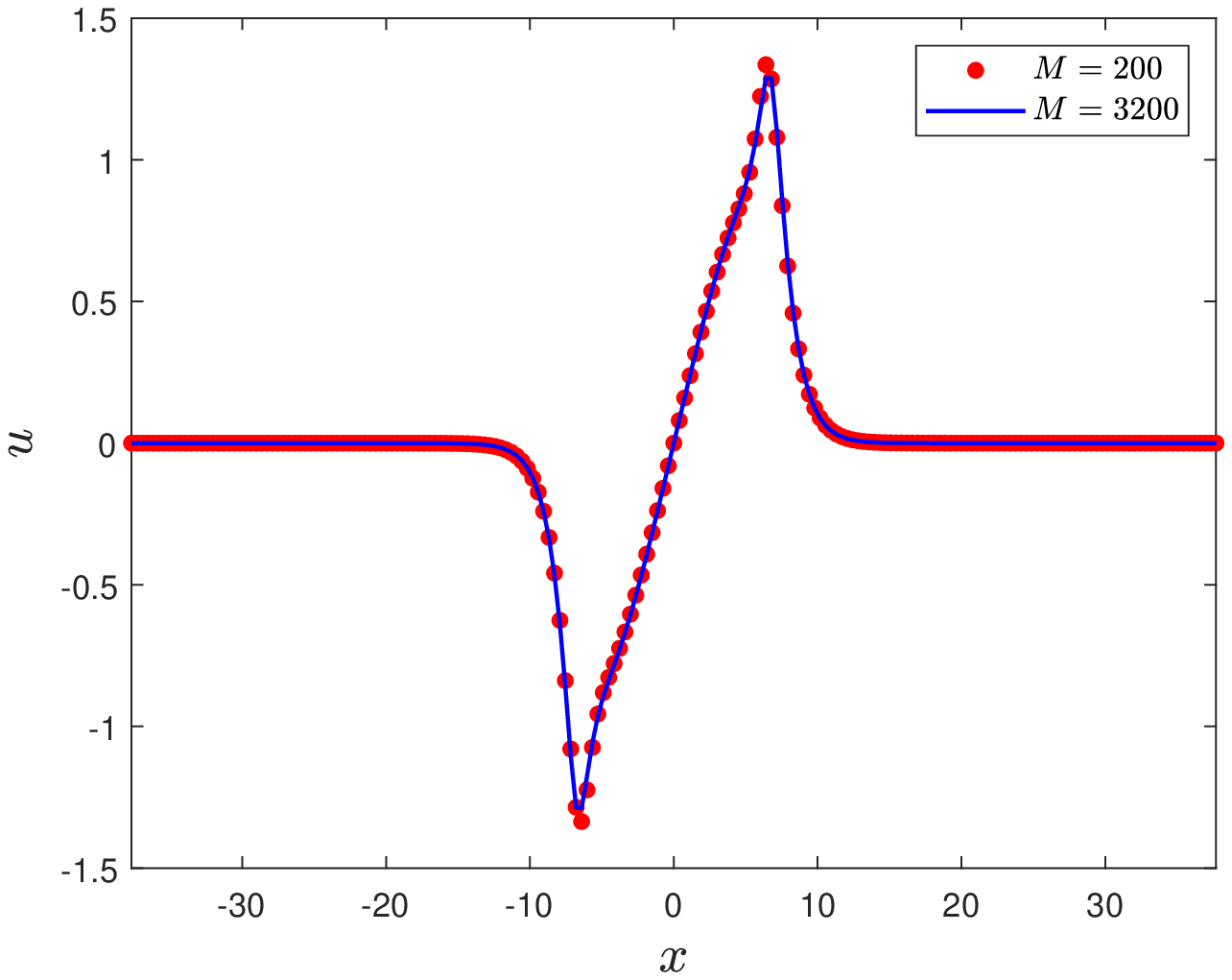}
	}\hspace{-5.5mm}\subfigure[$M=400$]{\centering
		\includegraphics[width=0.35\textwidth]{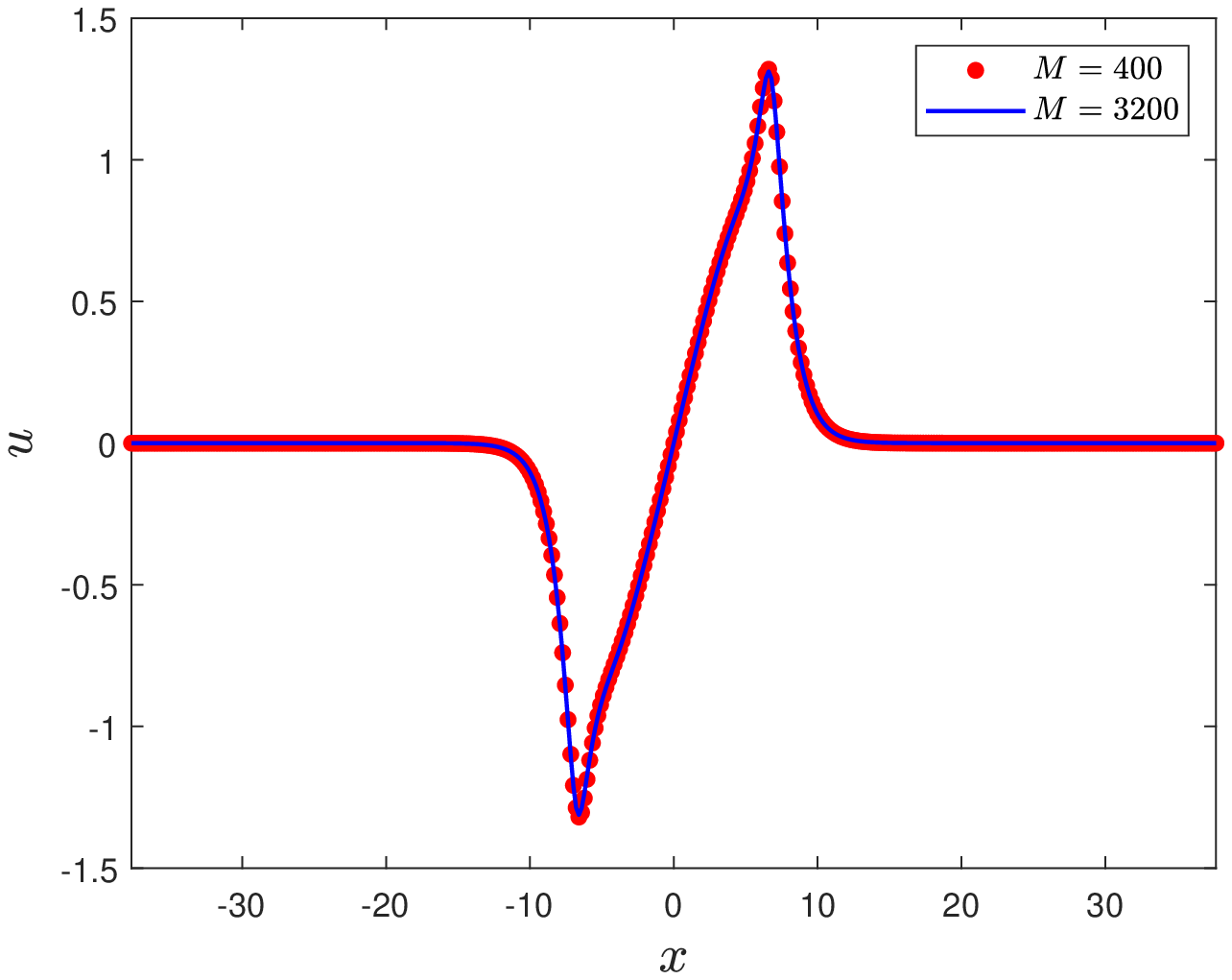}
	}\hspace{-5.5mm}\subfigure[$M=800$]{\centering
		\includegraphics[width=0.35\textwidth]{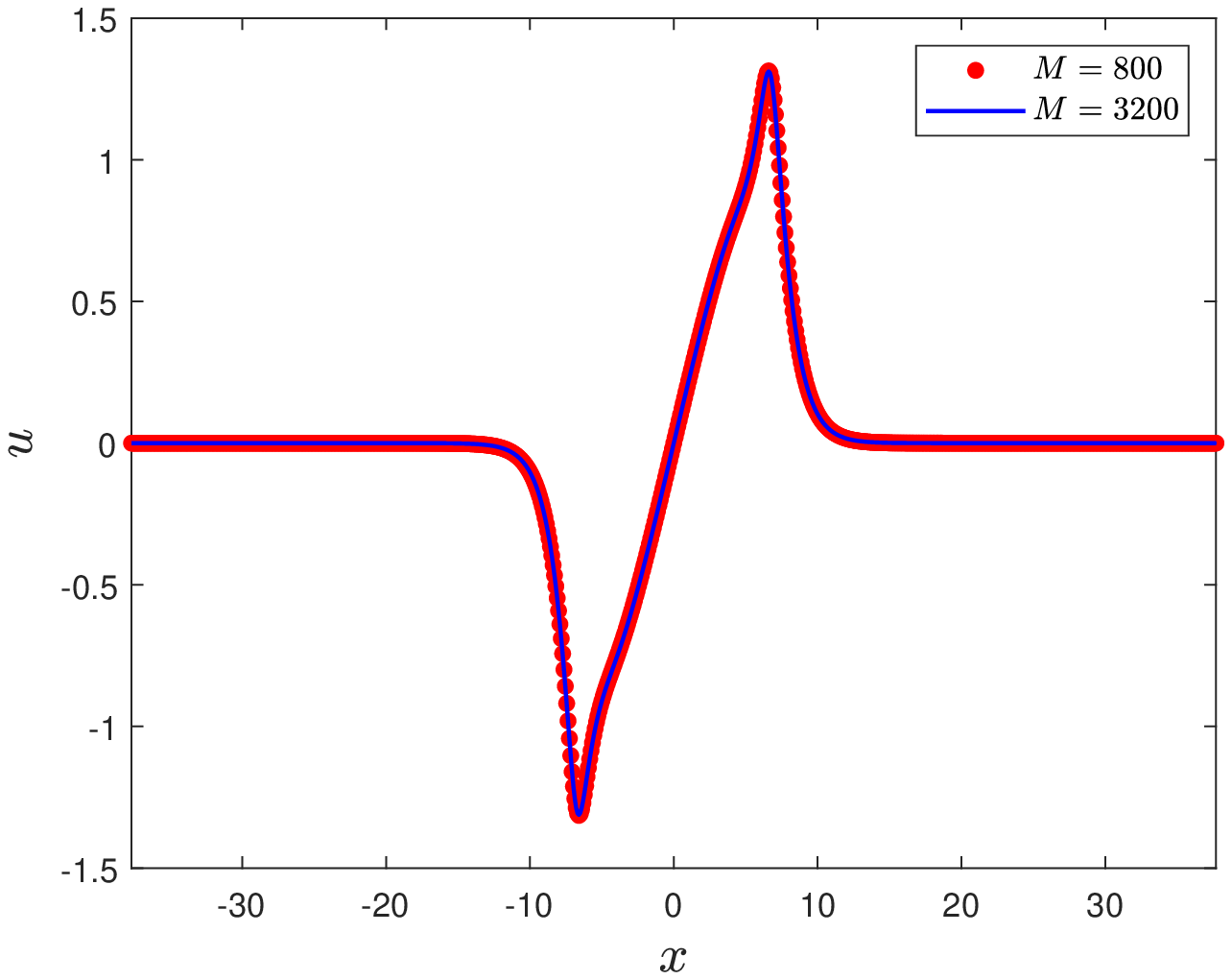}
	}
	\subfigure[$M=200$]{\centering
		\includegraphics[width=0.35\textwidth]{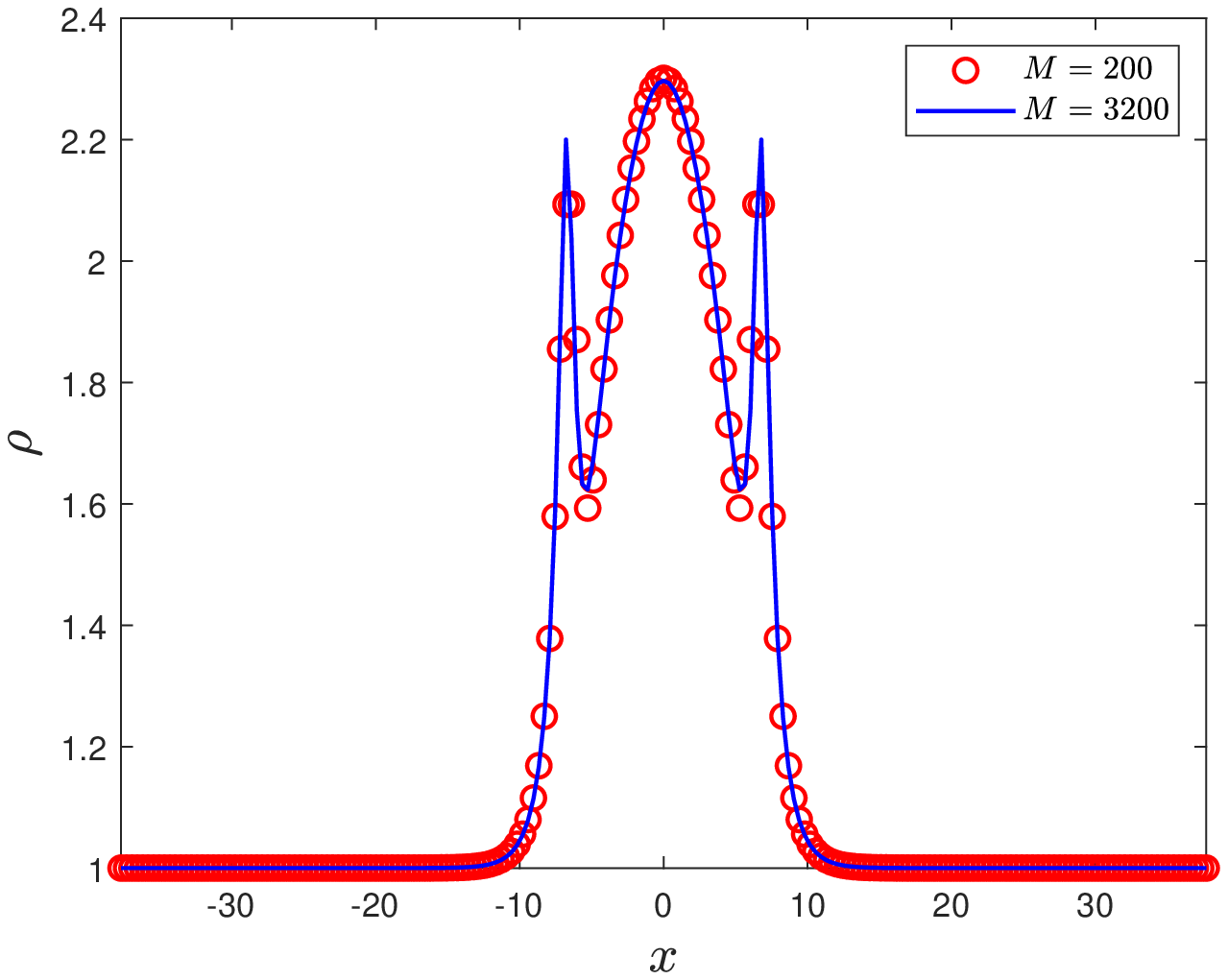}
	}\hspace{-5.5mm}\subfigure[$M=400$]{\centering
		\includegraphics[width=0.35\textwidth]{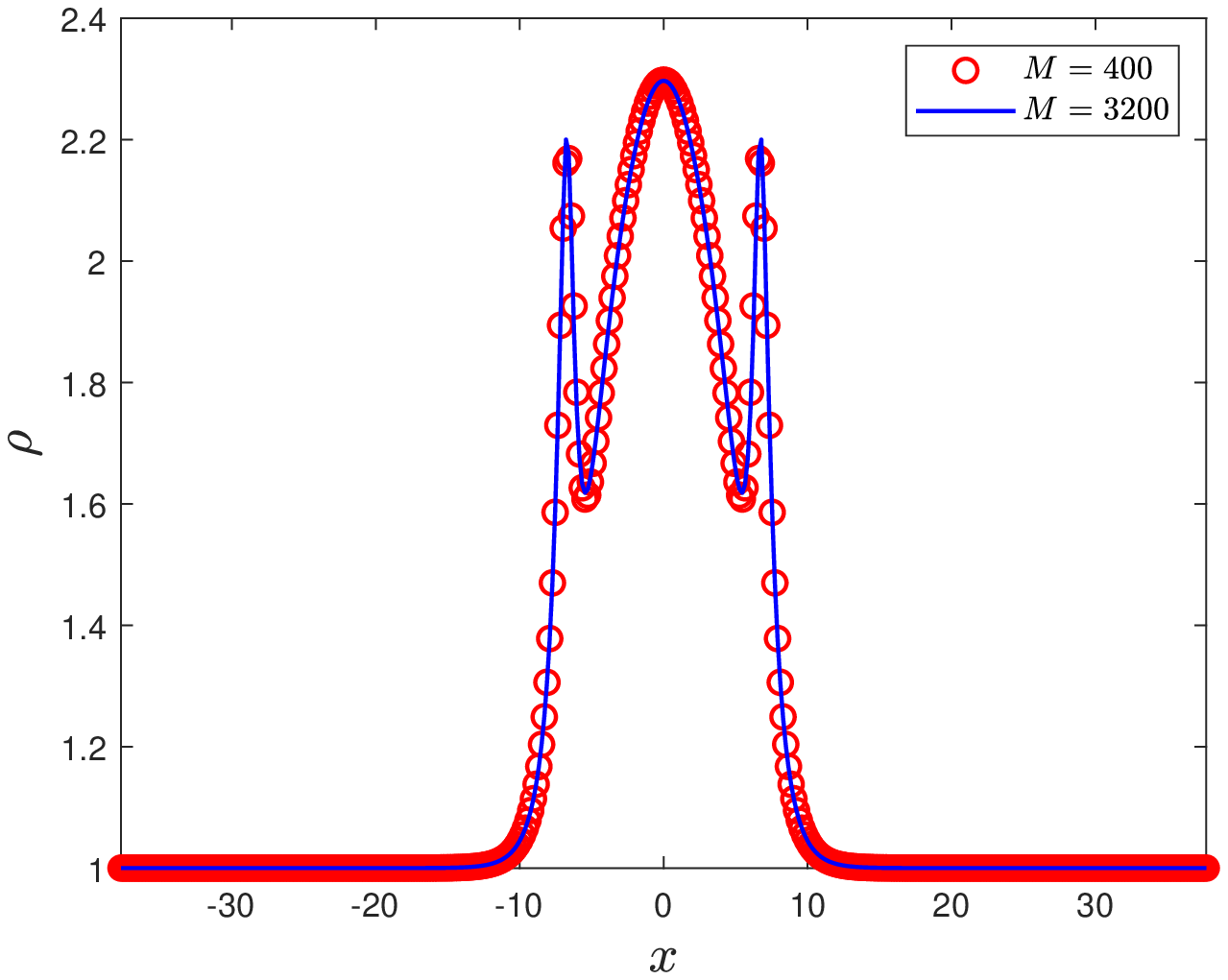}
	}\hspace{-5.5mm}\subfigure[$M=800$]{\centering
		\includegraphics[width=0.35\textwidth]{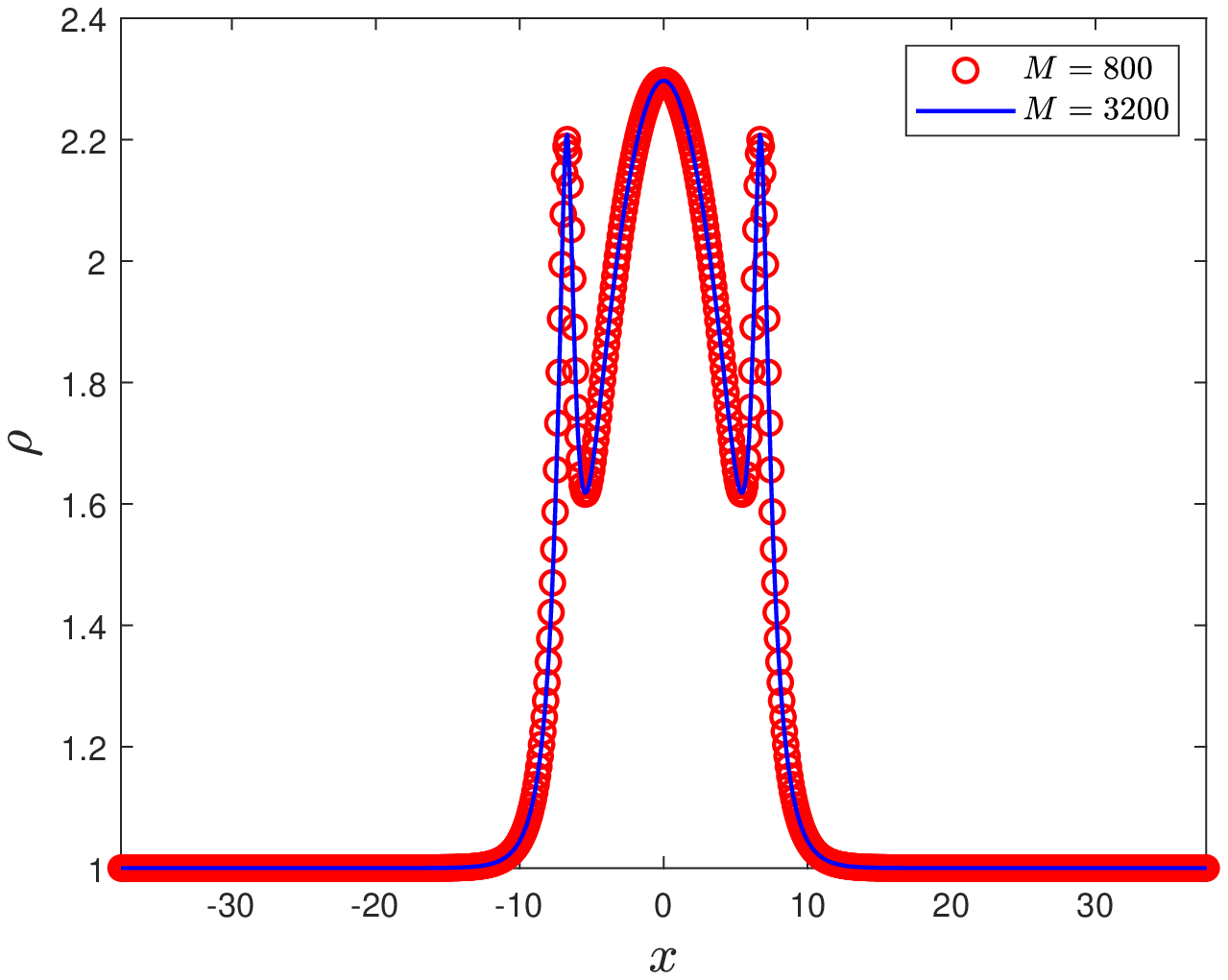}
	}
	\caption{The profiles of the predicted dam-break solutions of velocities $u(x,t)$ and heights $\rho(x,t)$ using different spatial grid points at $t = 2$ in \textbf{Case} (\uppercase\expandafter{\romannumeral2}). }	\label{fig1}
\end{figure}

	\begin{figure}[htbp]
	\vspace{-3mm}
	\subfigure[$M=200$]{\centering
		\includegraphics[width=0.35\textwidth]{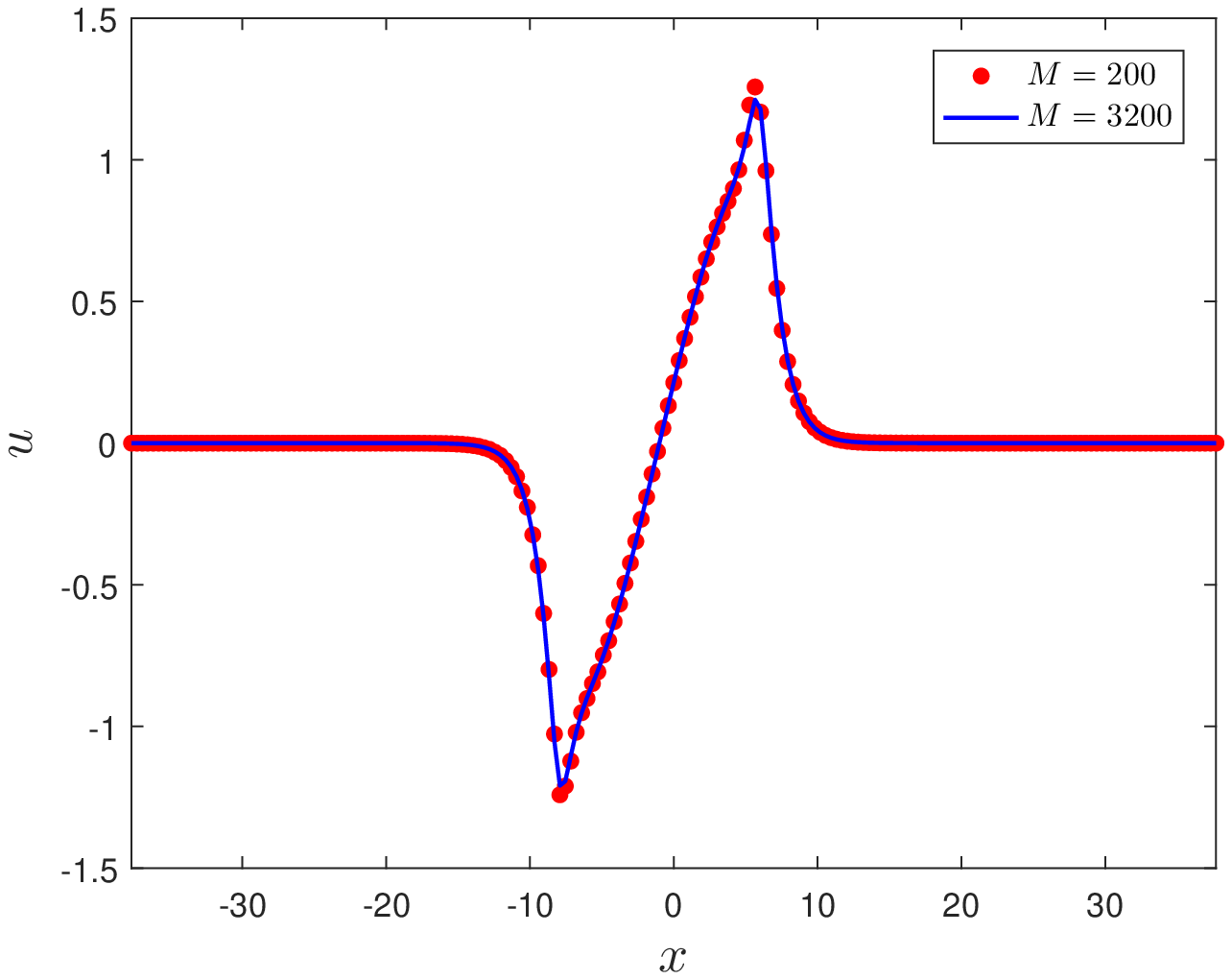}
	}\hspace{-5.5mm}\subfigure[$M=400$]{\centering
		\includegraphics[width=0.35\textwidth]{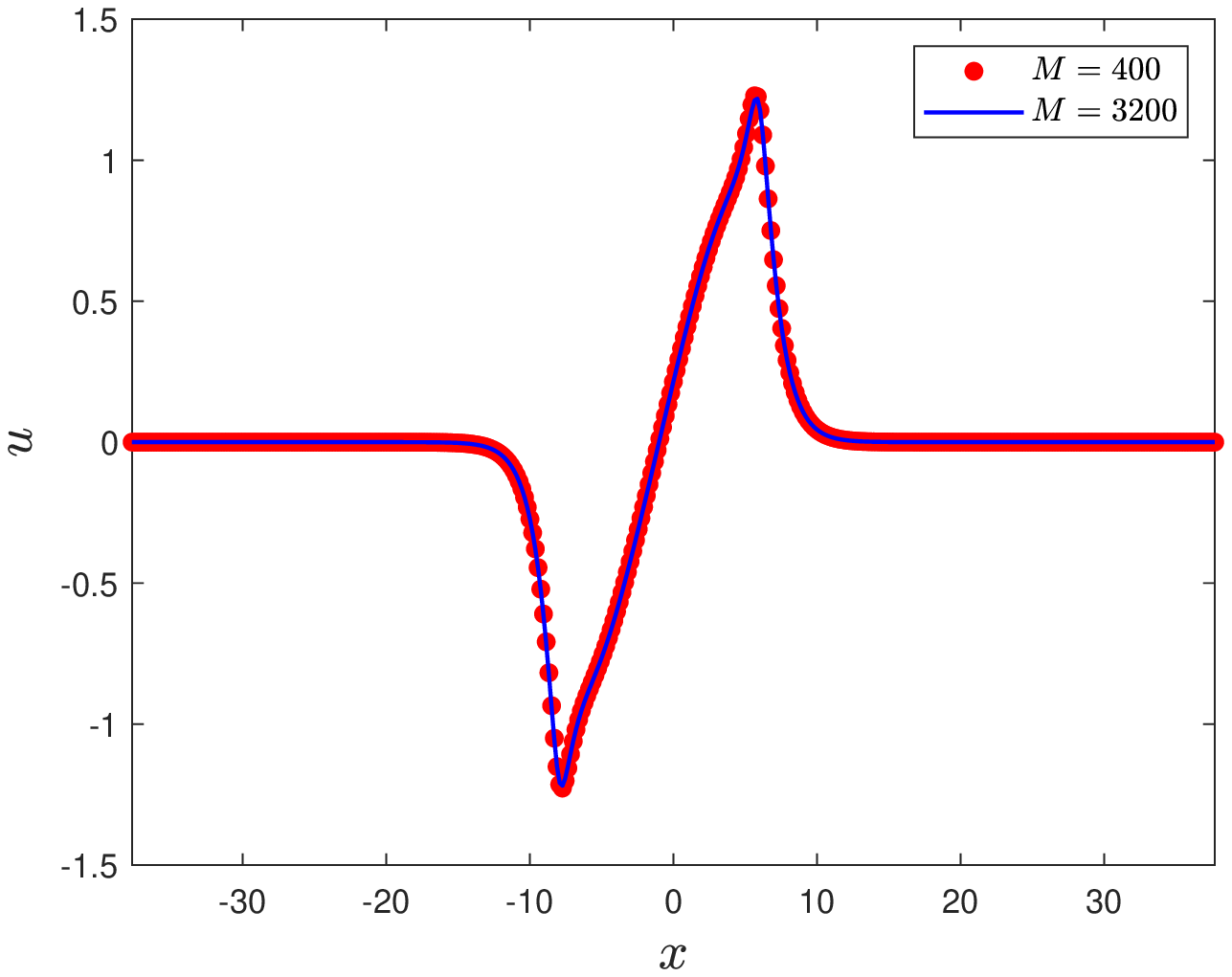}
	}\hspace{-5.5mm}\subfigure[$M=800$]{\centering
		\includegraphics[width=0.35\textwidth]{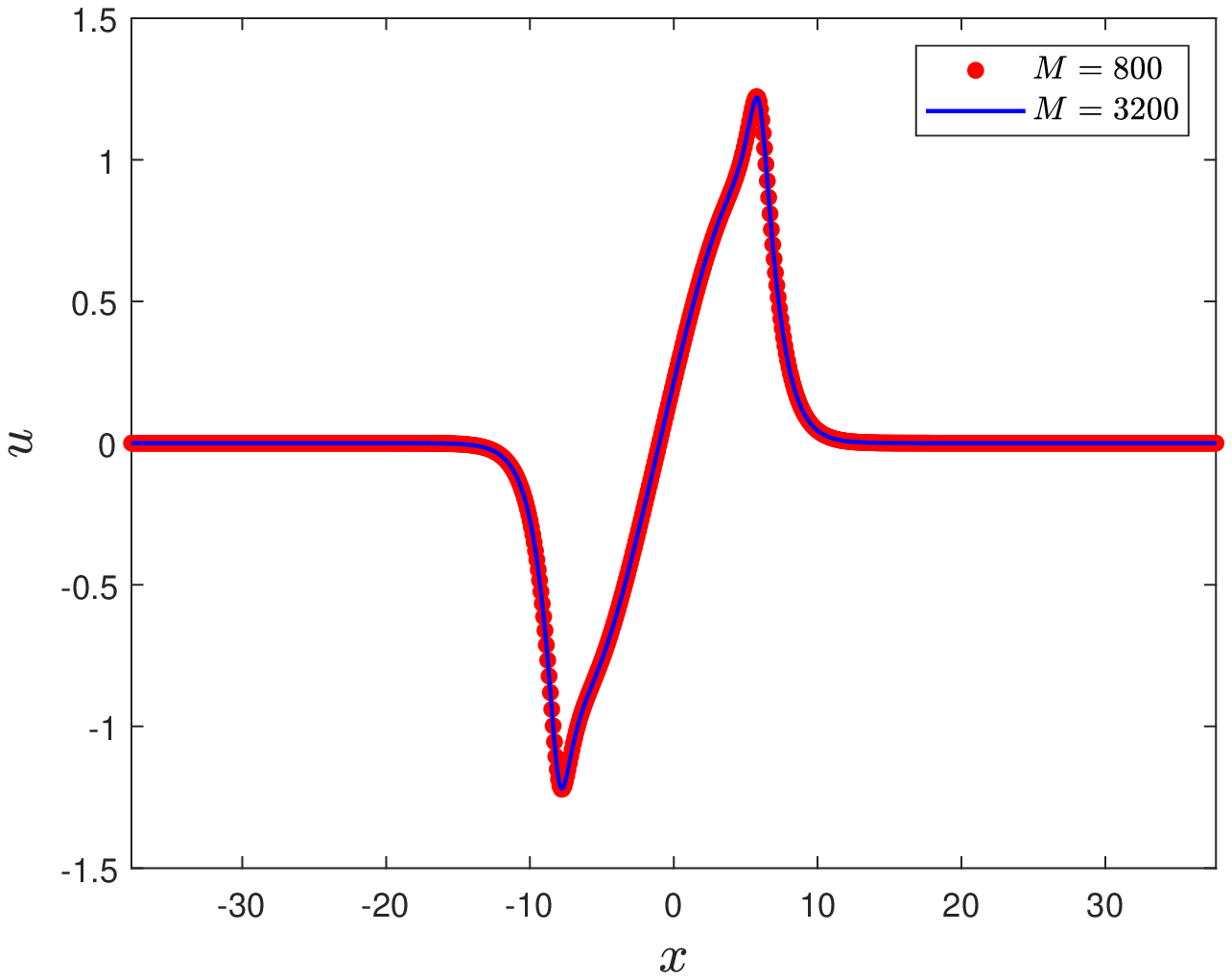}
	}
	\subfigure[$M=200$]{\centering
		\includegraphics[width=0.35\textwidth]{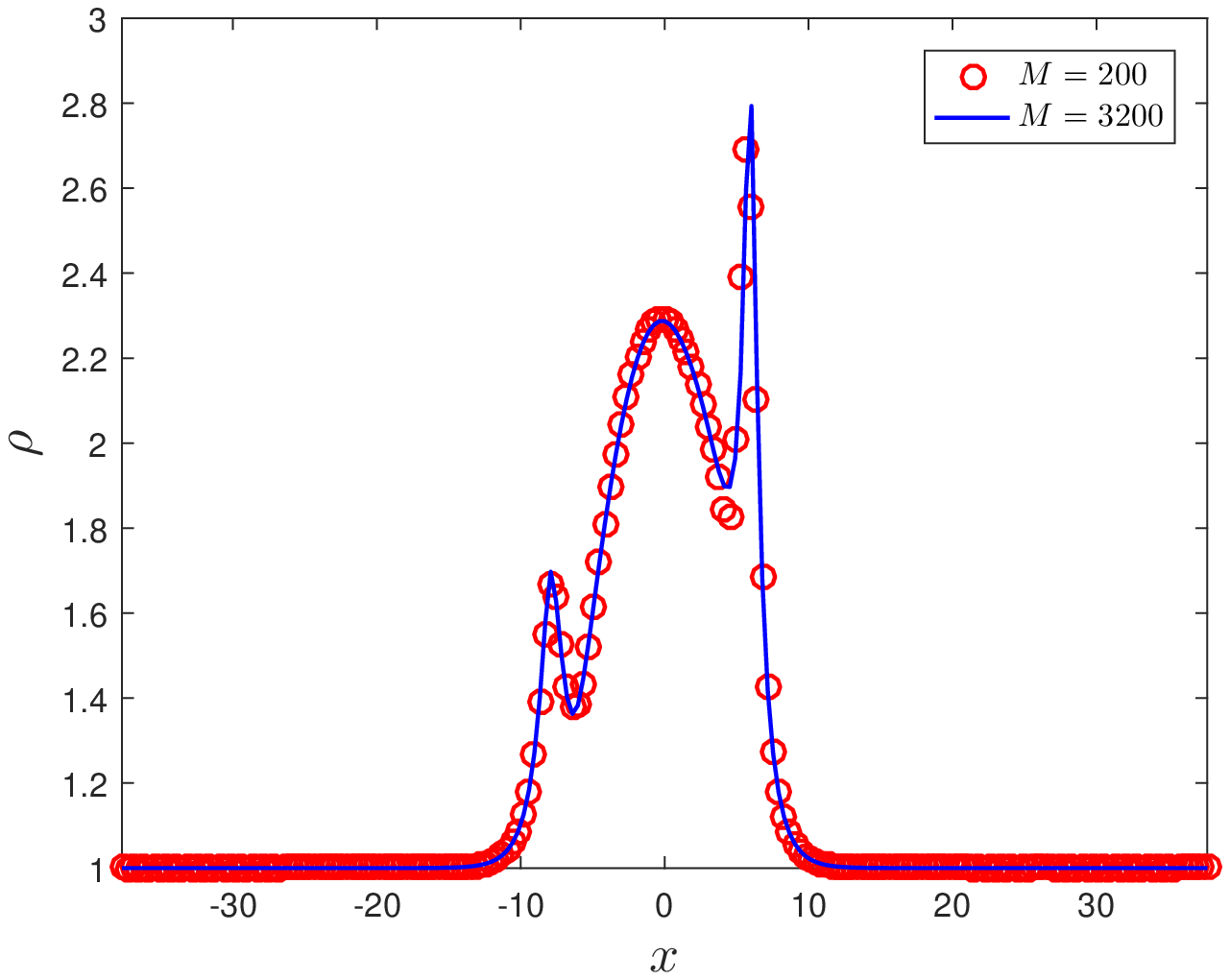}
	}\hspace{-5.5mm}\subfigure[$M=400$]{\centering
		\includegraphics[width=0.35\textwidth]{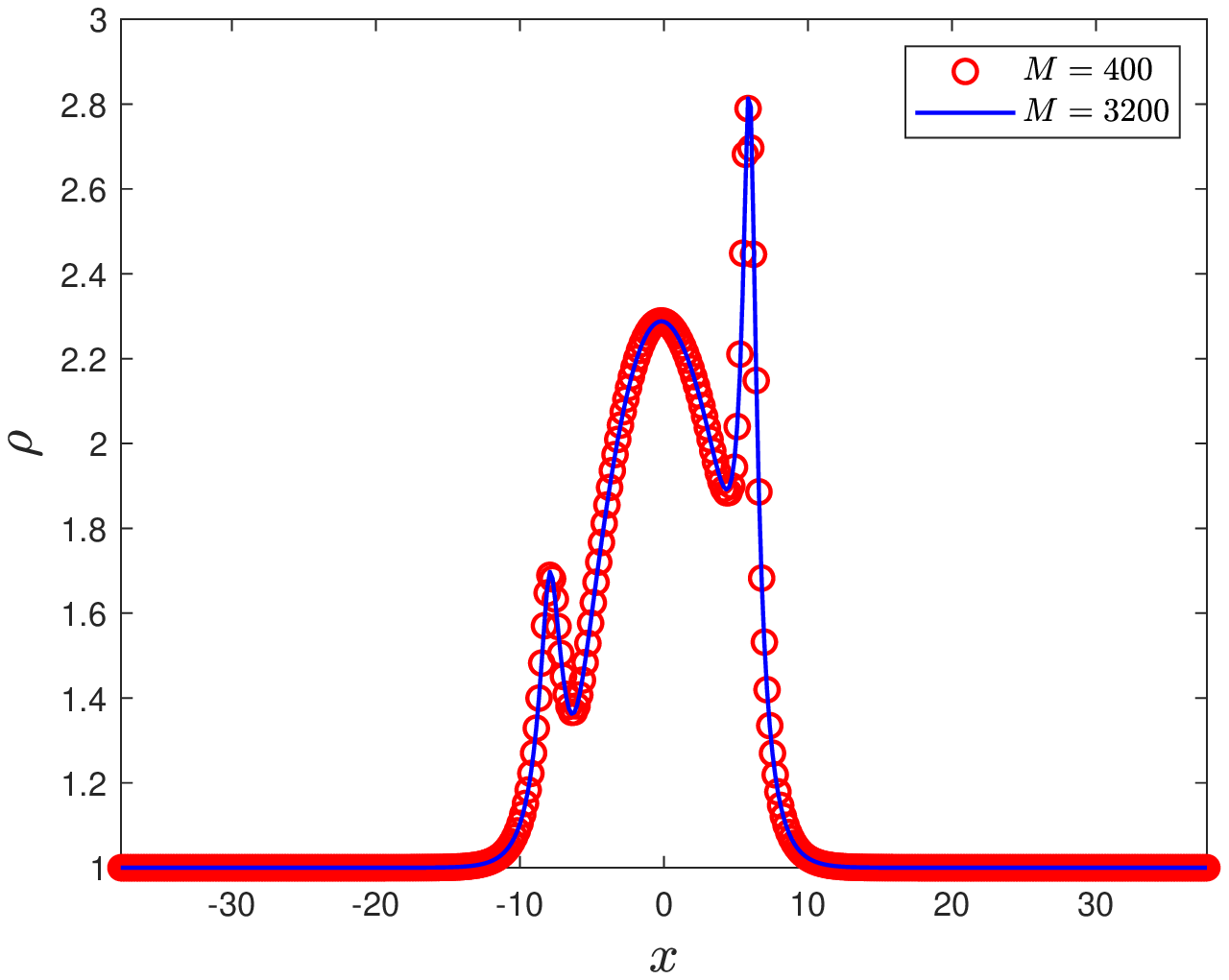}
	}\hspace{-5.5mm}\subfigure[$M=800$]{\centering
		\includegraphics[width=0.35\textwidth]{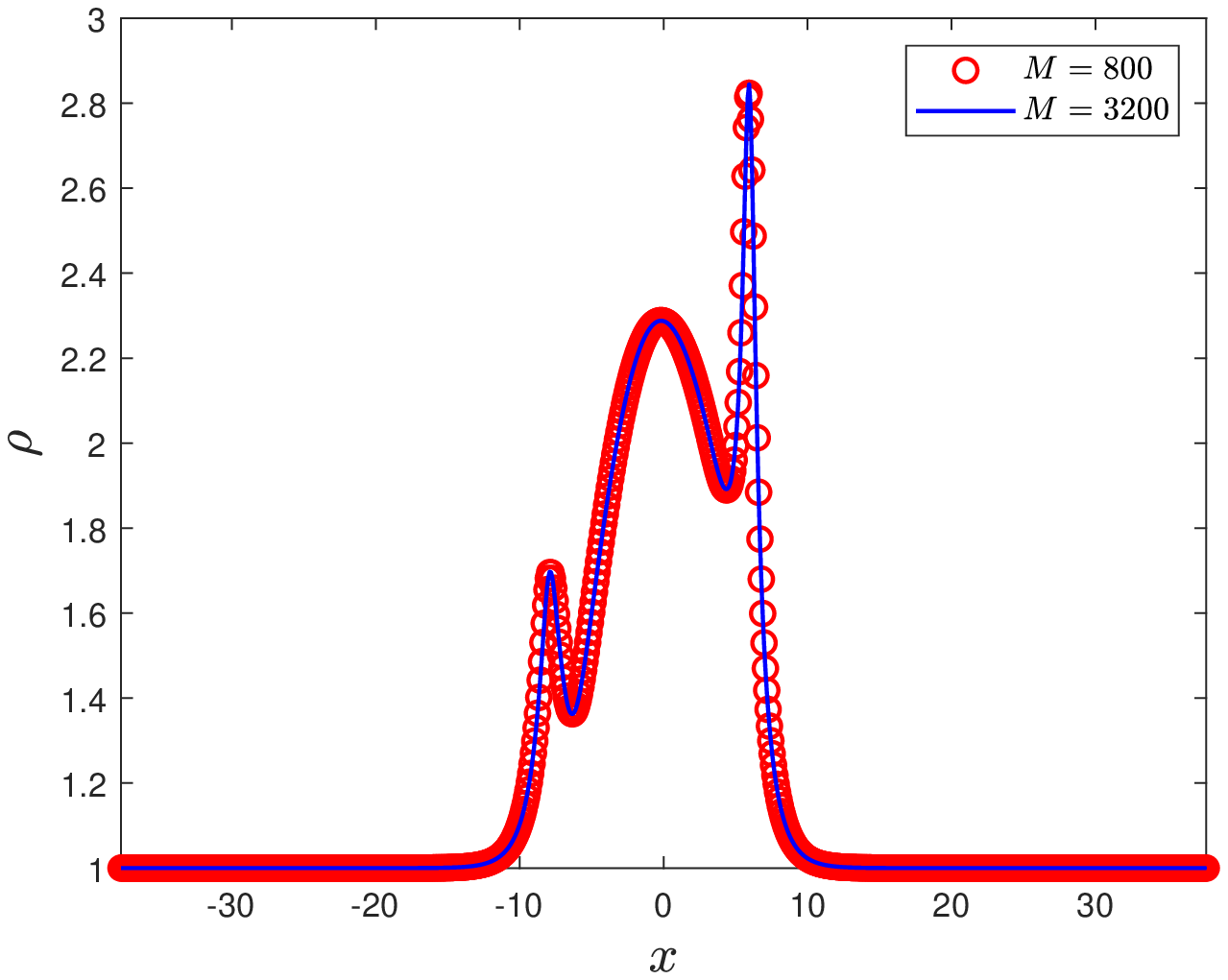}
	}
	\caption{The profiles of the predicted dam-break solutions of velocities $u(x,t)$ and heights $\rho(x,t)$ using different spatial grid points at $t = 2$ in \textbf{Case} (\uppercase\expandafter{\romannumeral4}). }  \label{fig2}
\end{figure}

	\begin{figure}[htbp]
	\vspace{-8mm}
	\centering
	\subfigure[Velocity variable, view(45,70)]{\centering
		\includegraphics[width=0.45\textwidth]{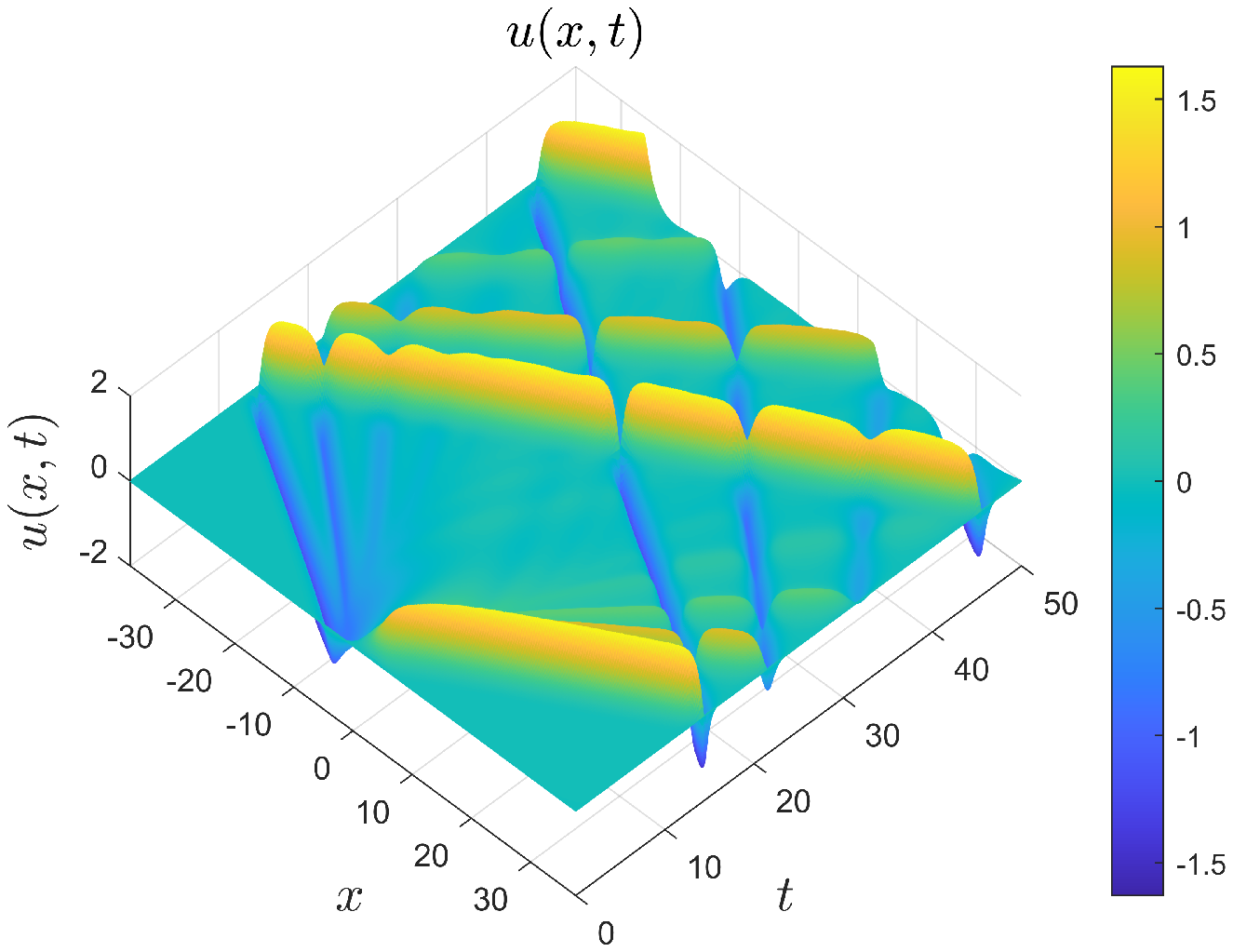}
	}
	\subfigure[Height variable, view(45,70)]{\centering
		\includegraphics[width=0.45\textwidth]{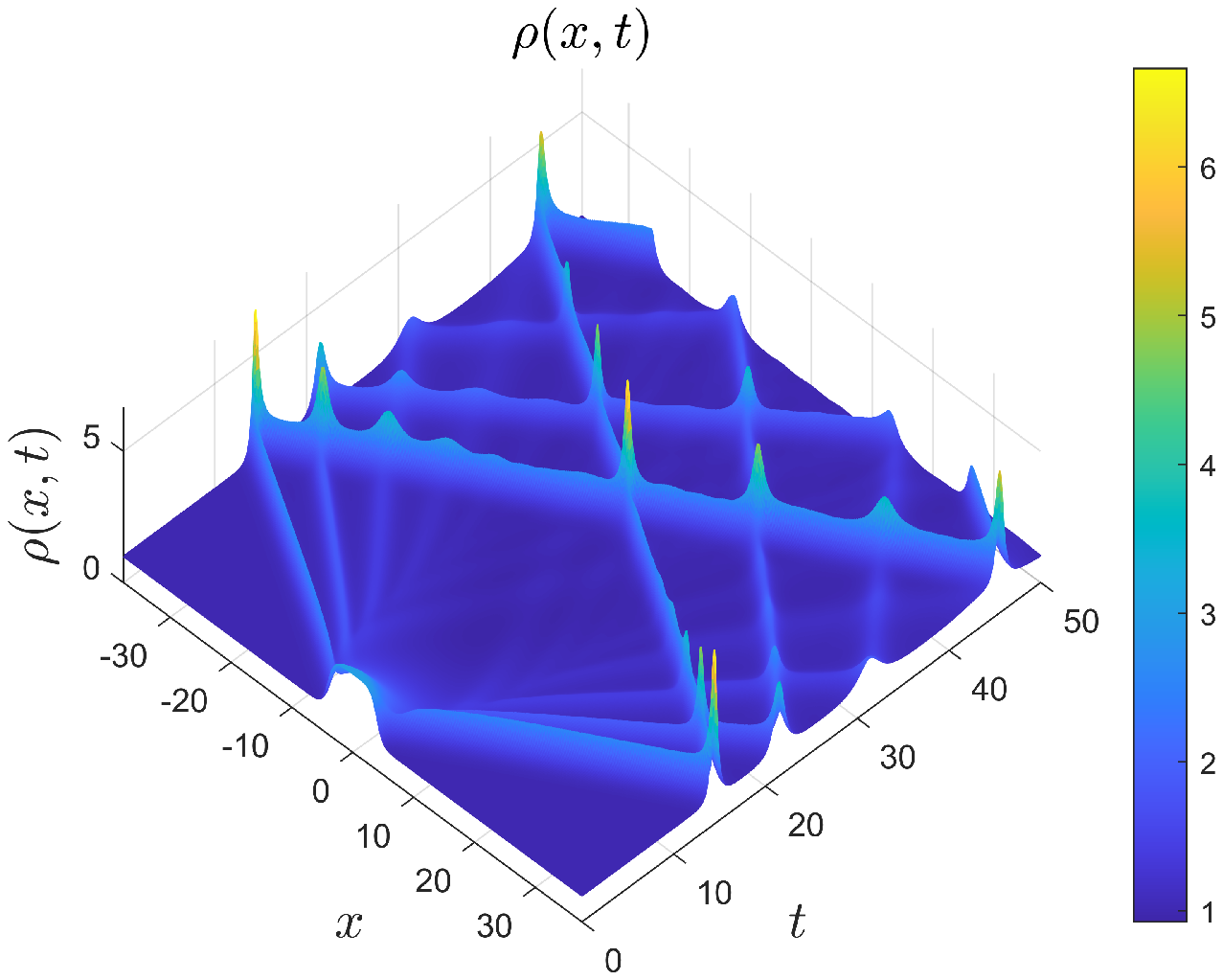}
	}
	\caption{The predicted dam-break solutions for the R2CH system in \textbf{Case} (\uppercase\expandafter{\romannumeral2}) show evolution of the velocity $u(x,t)$ and the height $\rho(x,t)$ with $t = 50$. } \label{fig3}
\end{figure}
\begin{figure}[htbp]
	\centering
	\subfigure[Velocity variable, view(45,70)]{\centering
		\includegraphics[width=0.45\textwidth]{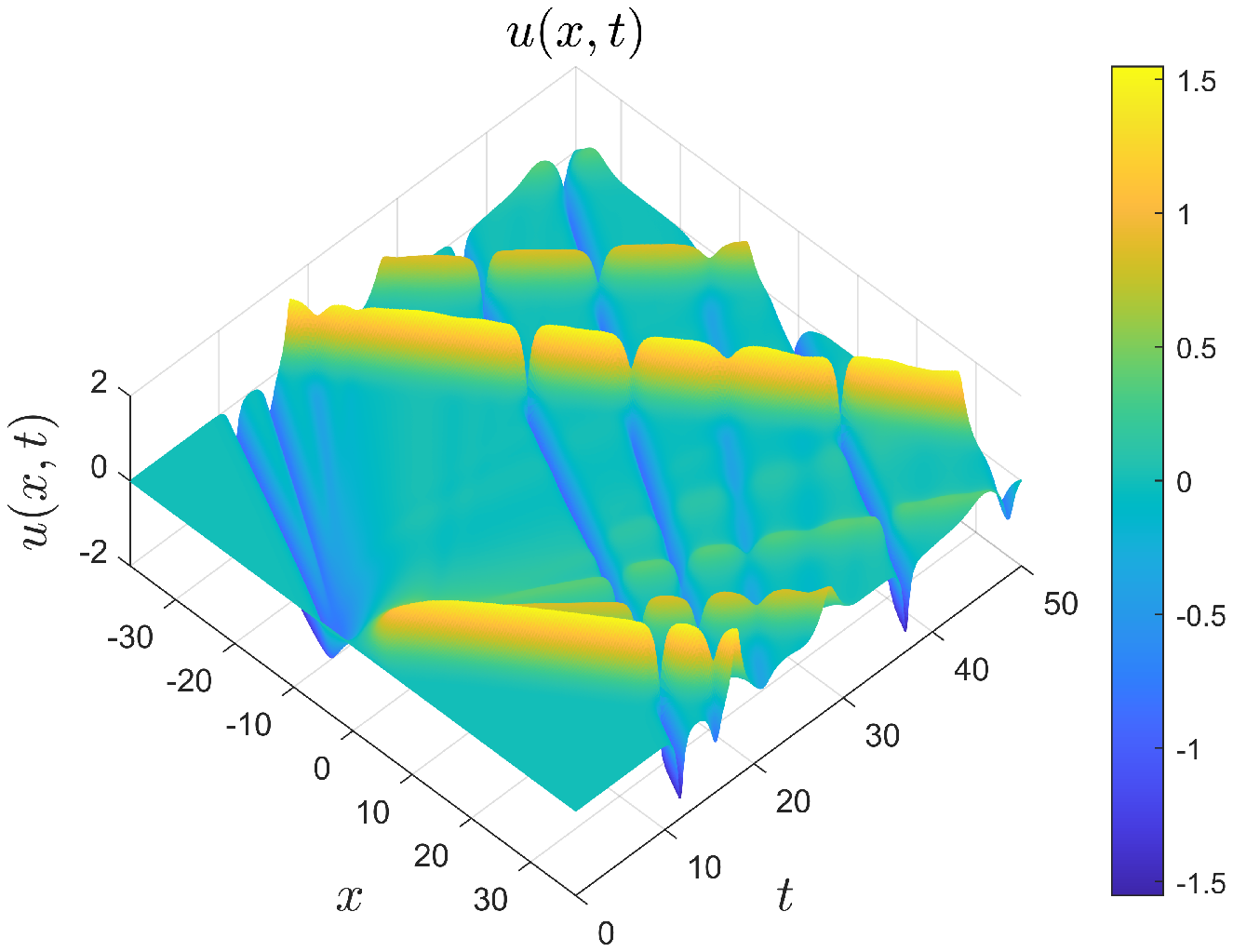}
	}
	\subfigure[Height variable, view(45,70)]{\centering
		\includegraphics[width=0.45\textwidth]{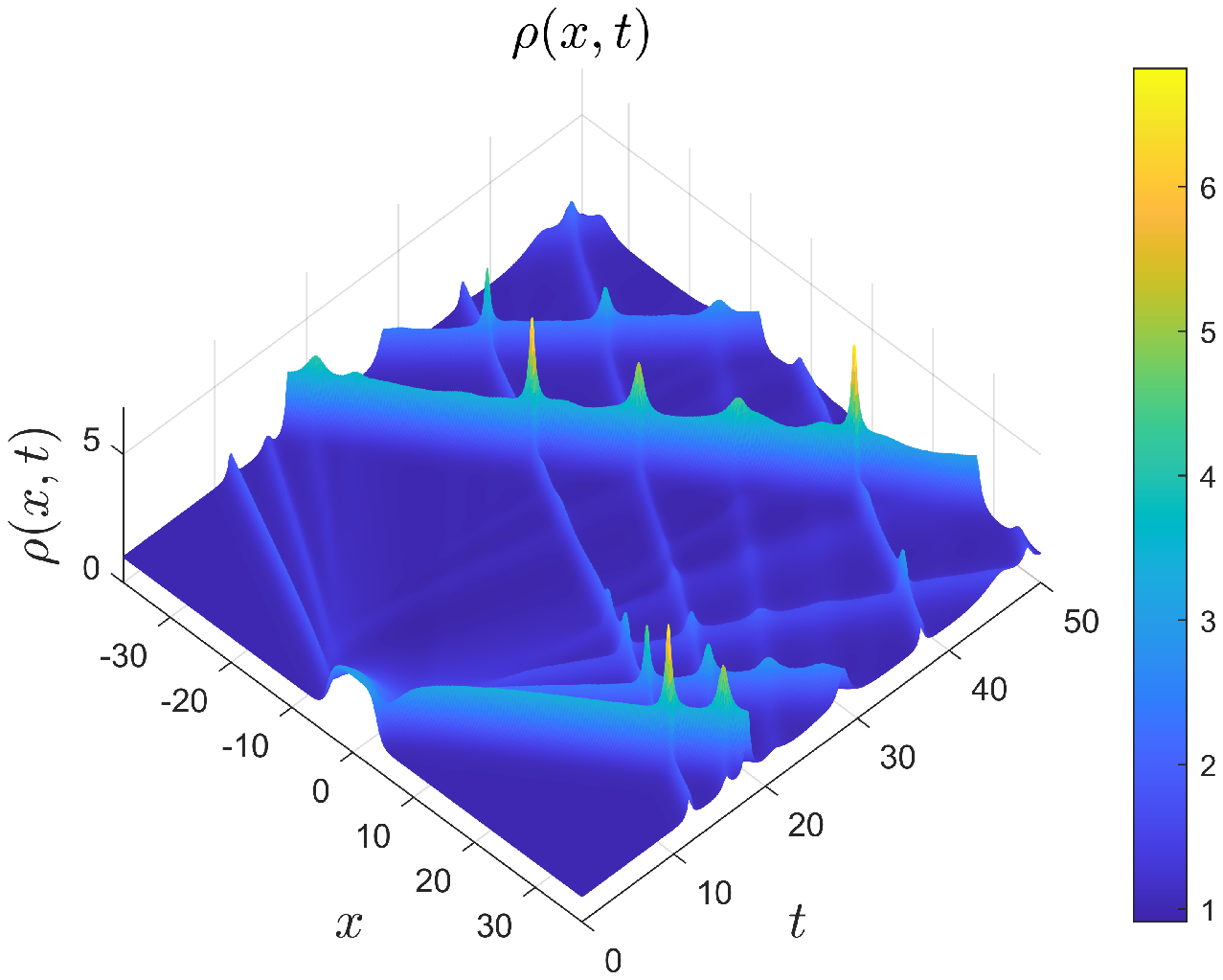}
	}
	\caption{The predicted dam-break solutions for the R2CH system in \textbf{Case} (\uppercase\expandafter{\romannumeral4}) show evolution of the velocity $u(x,t)$ and the height $\rho(x,t)$ with $t = 50$. } \label{fig4}
\end{figure}
	\subsection{Part II: nonsmooth initial data}\label{sec5_2}
In this part, we conduct numerical simulation for the R2CH system with nonsmooth initial data.
The parameters of the numerical examples come from the following three cases
\begin{itemize}
	\item \textbf{Case} (\uppercase\expandafter{\romannumeral1}). $A = 0$, $\mu = 0$, $\sigma = 1$, $\Omega = 0$;
	\item \textbf{Case} (\uppercase\expandafter{\romannumeral2}).  $A = 0$, $\mu = 0$, $\sigma = 1$, $\Omega = 0.1$;
	\item  \textbf{Case} (\uppercase\expandafter{\romannumeral3}). $A = 1$, $\mu = 1$, $\sigma = 1$, $\Omega = 73 \times 10^{-6}$.
\end{itemize}

\begin{example}[\textbf{Three-peakon interaction for the CH equation}\cite{JWG2020}, \cite{LP2016}, \cite{XS2008}, \cite{ZST2011}] We now test a degenerate R2CH system \eqref{eq1.1}--\eqref{eq1.2} by taking $\rho=A=\mu=0$ and $\sigma=1$, which will reduce to the classical CH equation. Consider the three-peakon interaction with the following initial condition
	\begin{align*}
		u(x,0) = \phi_1(x) + \phi_2(x) + \phi_3(x),
	\end{align*}
	where
	\begin{equation*}
		\phi_i(x) = \left\{
		\begin{array}{ll}
			\displaystyle\frac{c_i}{\cosh(L/2)}\cosh(x-x_i),\quad |x-x_i| \leqslant L/2,  \\ [4\jot]
			\displaystyle \frac{c_i}{\cosh(L/2)}\cosh(L-(x-x_i)),\quad |x-x_i| > L/2, \\
		\end{array}
		\right. \quad i = 1,\,2,\,3.
	\end{equation*}
	The parameters are given by $c_1 = 2$, $c_2 = 1$, $c_3 = 0.8$, $x_1 = -5$, $x_2 = -3$, $x_3 = -1$  and the computational domain is $[0,L]$ with $L = 30$.
\end{example}
In the calculation, we take the spatial stepsize $h = L/2048$ and temporal stepsize
$\tau = 1/10000$ to simulate this interaction at $t = 0,\,1,\,2,\,3,\,4,\,6,\,8,\,10$, respectively. Figure \ref{fig5} shows the moving peak interaction at different instants of time. It is clear to see that the proposed scheme performs well in resolving the complex interaction among multiple peakons for the CH equation. The results obtained here are comparable to those obtained by a fourth-order HIEQ-GM method in \cite{JWG2020} and a MSWCM-AD30 method in \cite{ZST2011}.
\begin{figure}[htbp]
	\vspace{-10mm}
	\centering
	\subfigure{\centering
		\includegraphics[width=0.4\textwidth]{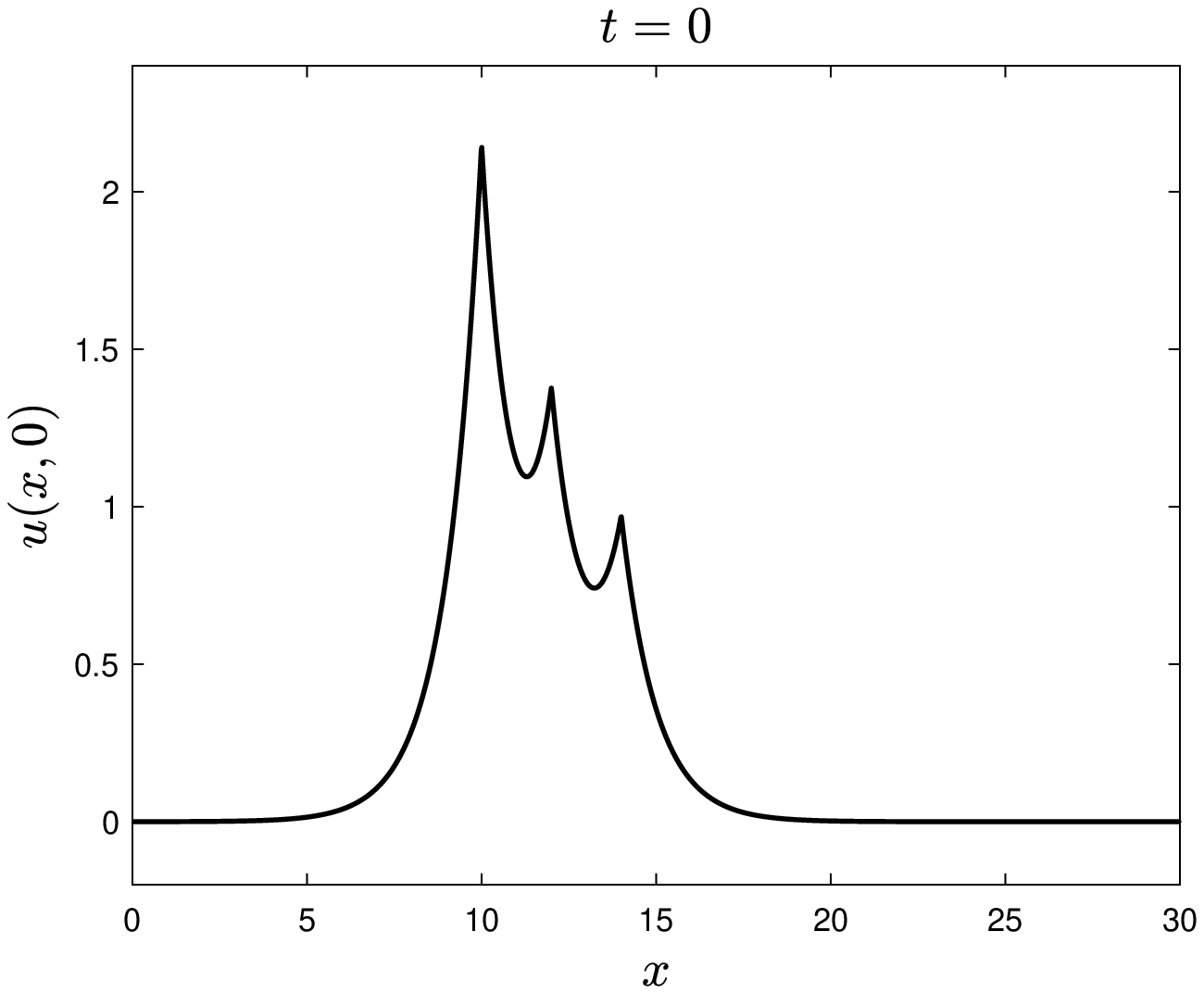}
	}\hspace{1.5mm}\subfigure{\centering
		\includegraphics[width=0.4\textwidth]{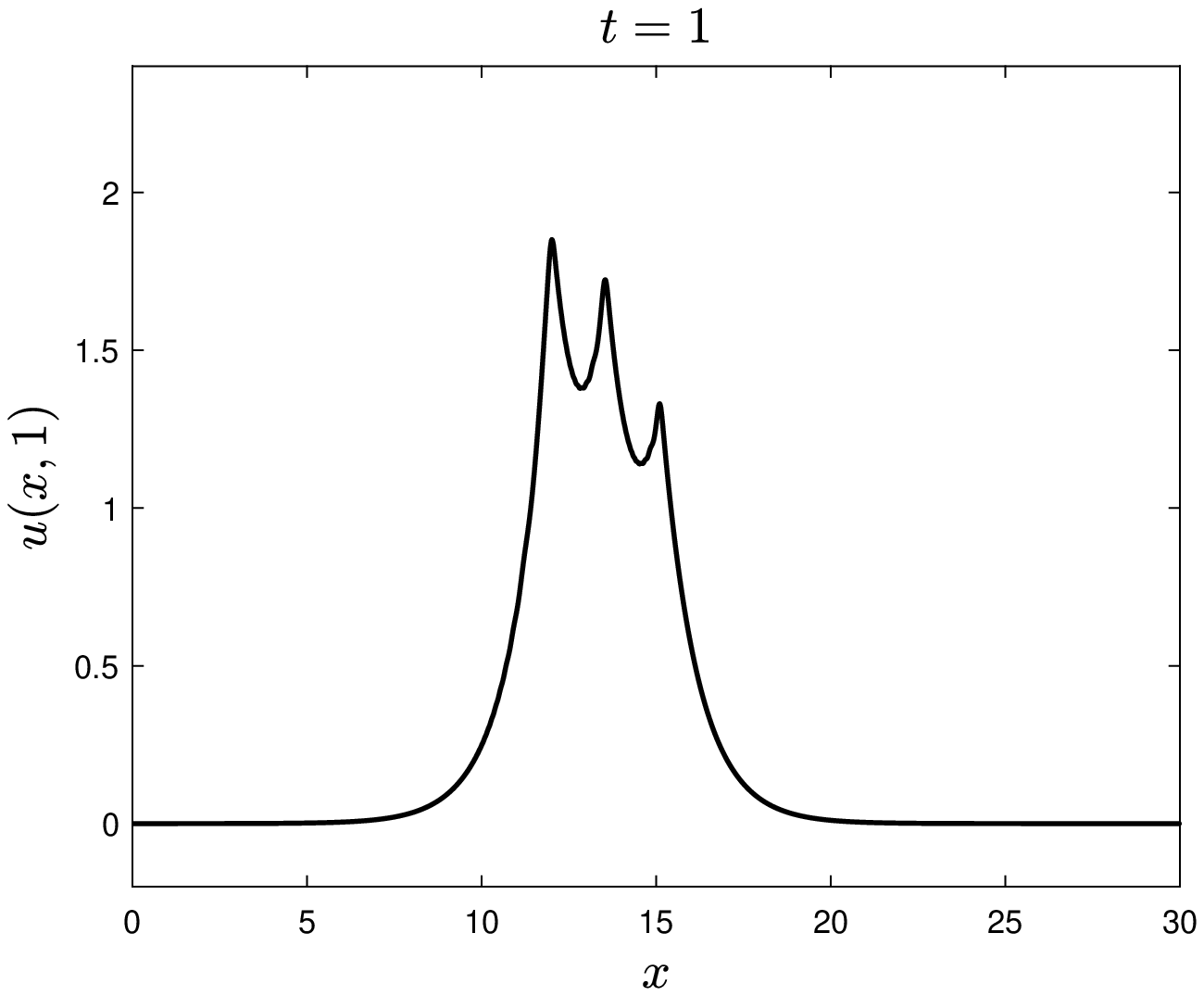}
	}\\
	\subfigure{\centering
		\includegraphics[width=0.4\textwidth]{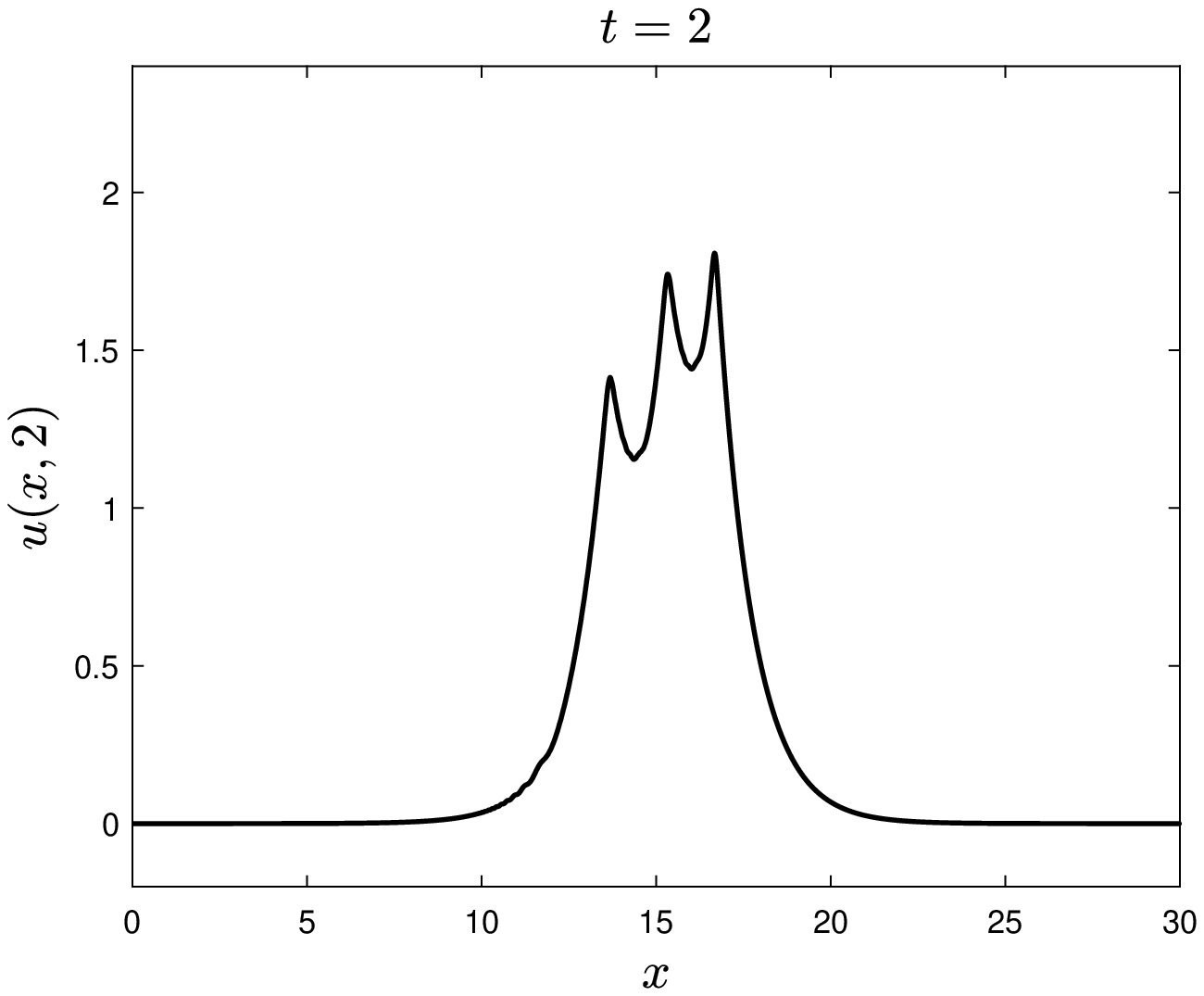}
	}\hspace{1.5mm}\subfigure{\centering
		\includegraphics[width=0.4\textwidth]{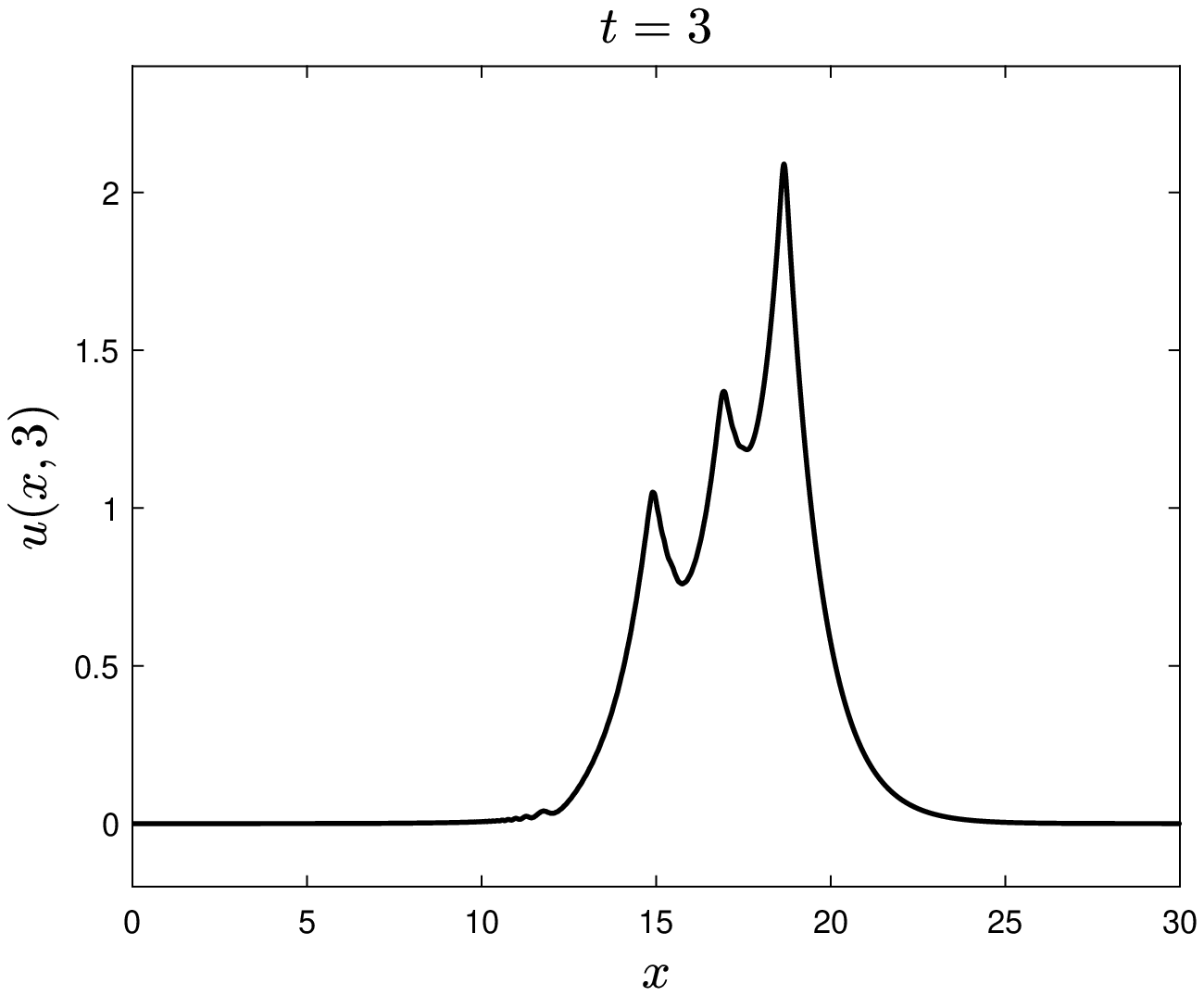}
	}\\
	\subfigure{\centering
		\includegraphics[width=0.4\textwidth]{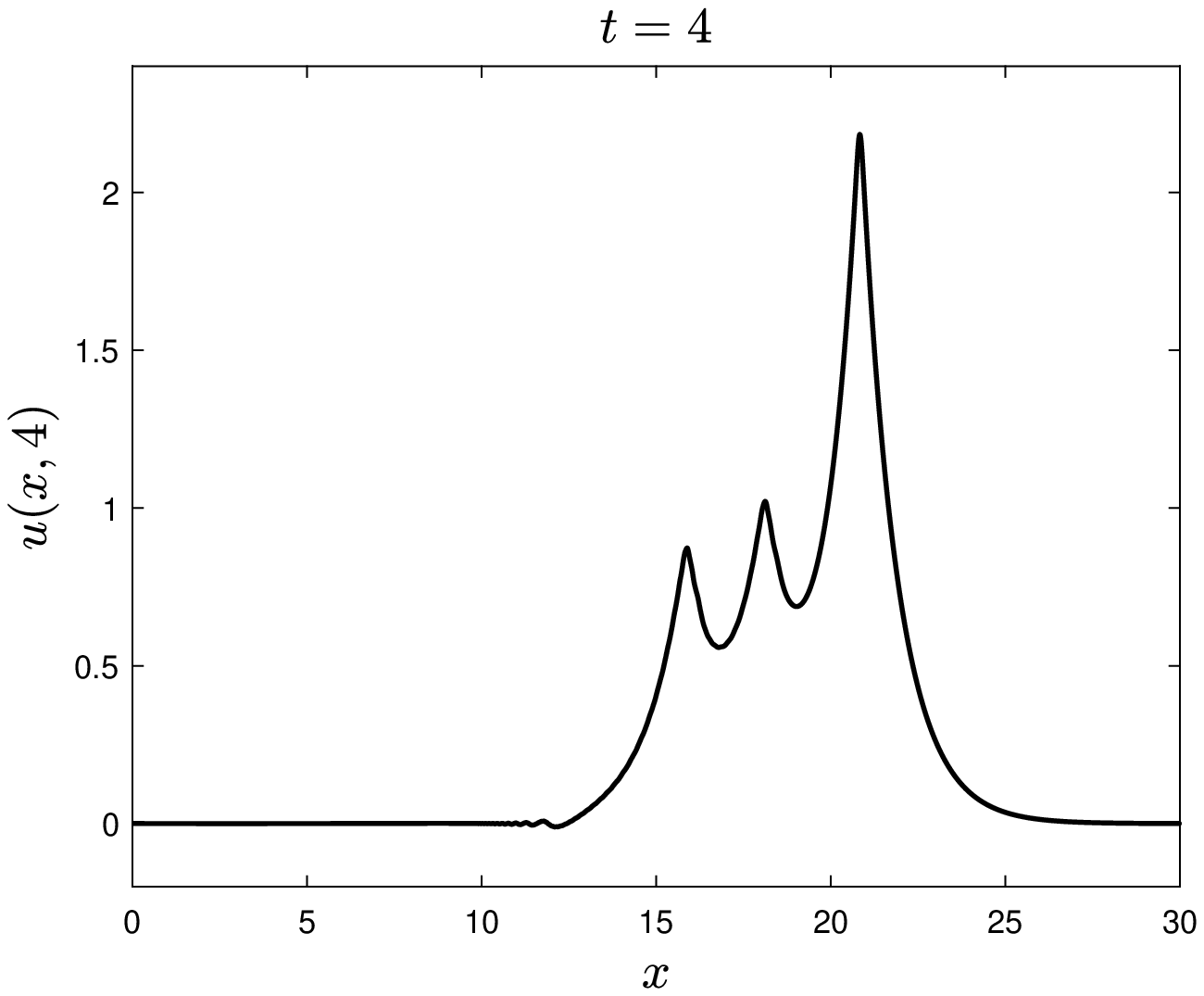}
	}\hspace{1.5mm}\subfigure{\centering
		\includegraphics[width=0.4\textwidth]{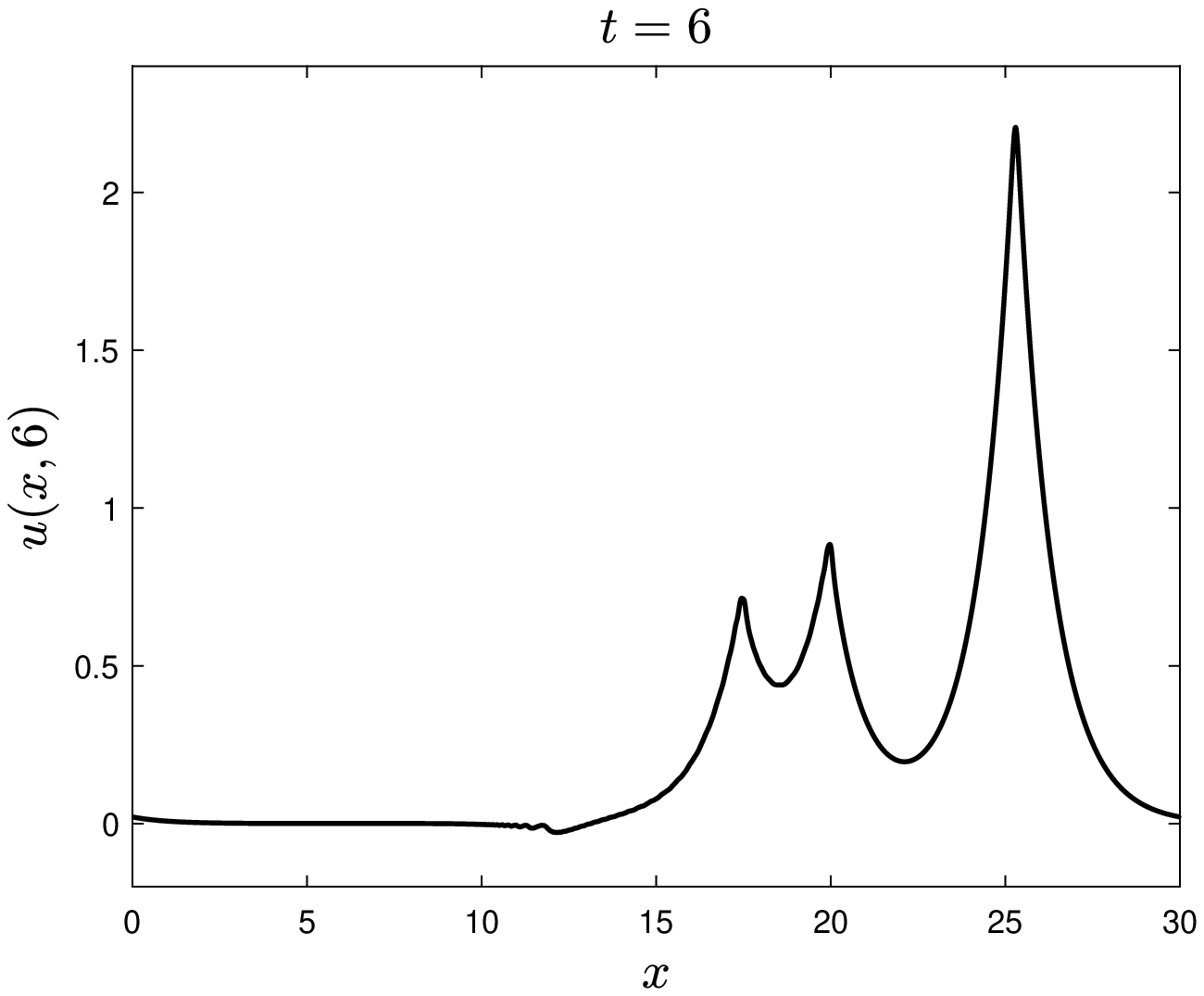}
	}\\
	\subfigure{\centering
		\includegraphics[width=0.4\textwidth]{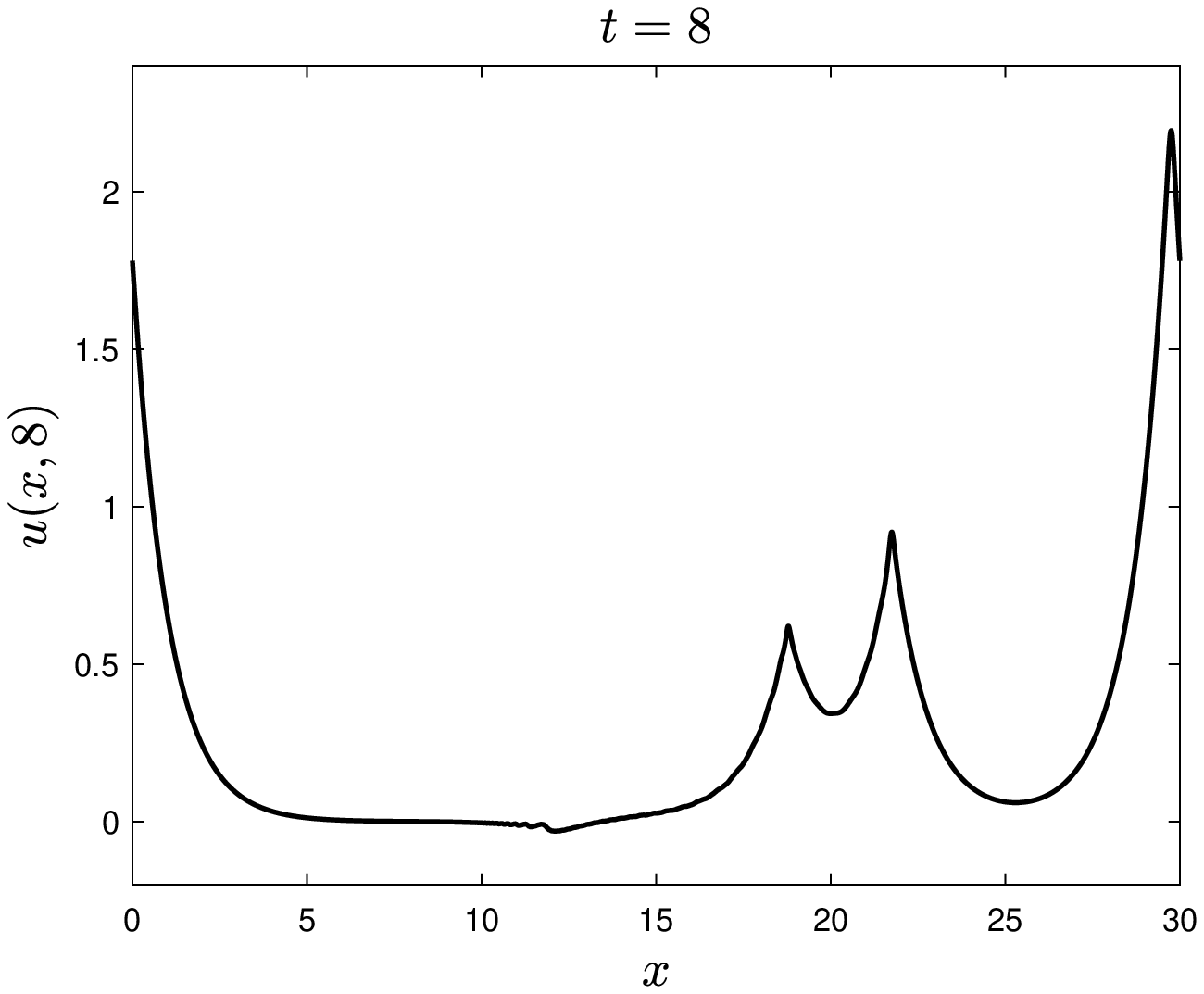}
	}\hspace{1.5mm}\subfigure{\centering
		\includegraphics[width=0.4\textwidth]{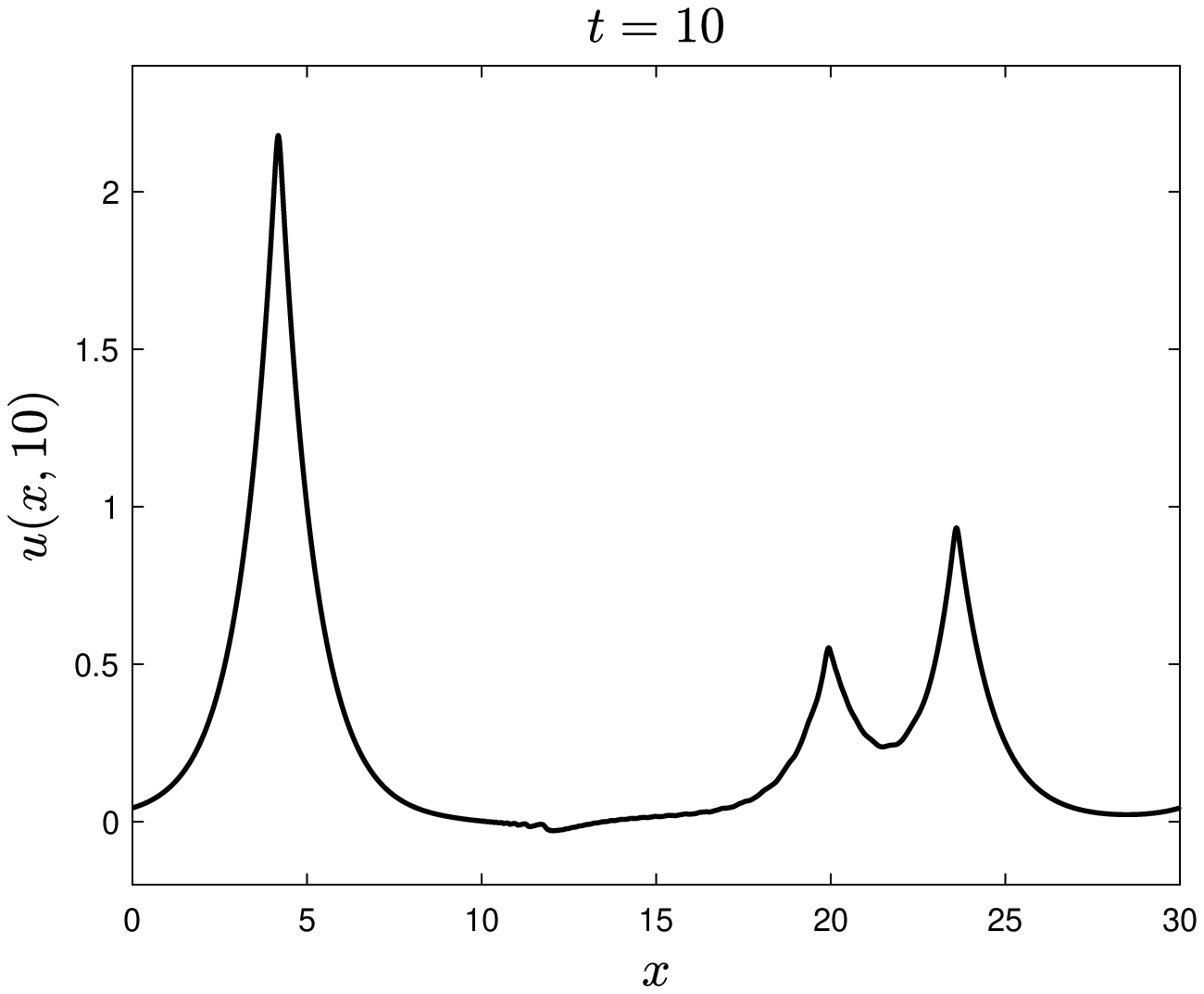}
	}\caption{Three-peakon interaction of the CH equation at $t=0,\,1,\,2,\,3,\,4,\,6,\,8,\,10$, respectively.}	
	\label{fig5}
\end{figure}
\begin{example}[\textbf{Single peakon}\cite{CKL2020}] We now consider a single peakon solution of the  R2CH system \eqref{eq1.3}--\eqref{eq1.4} with the initial conditions
	\begin{align*}
		u(x,0) = e^{-|x-x_0|}, \quad \rho (x,0) = 0.5,
	\end{align*}
	where $x_0 = 10$ is used and the computational domain is prescribed as $[0,20]$.
\end{example}
The parameters in \textbf{Case} (\uppercase\expandafter{\romannumeral1}) and \textbf{Case} (\uppercase\expandafter{\romannumeral3}) are taken respectively to show the behavior of the single peakon. We take spatial stepsize $h=0.025$ and temporal stepsize $\tau=0.0005$ in the calculation. Figures \ref{fig6} and \ref{fig7} depict the predicted solution profiles of the velocity variable $u(x,t)$ and height variable $\rho(x,t)$ at $t=1$, $3$ and $5$, respectively, under \textbf{Case} (\uppercase\expandafter{\romannumeral1}) and \textbf{Case} (\uppercase\expandafter{\romannumeral3}). In contrast, we observe that the selection of parameters has remarkable impact on the profiles of solutions to the system \eqref{eq1.3}--\eqref{eq1.4}. Figure \ref{fig8} shows the computed results of discrete conserved quantities \eqref{eq2.5(b)}--\eqref{eq2.7(b)} at time $t=5$ for \textbf{Case} (\uppercase\expandafter{\romannumeral1}) and \textbf{Case} (\uppercase\expandafter{\romannumeral3}). One can easily observe that the scheme \eqref{eq3.1}--\eqref{eq3.3} guarantees the conservation of mass, momentum and energy for different parameters. Moreover, One can also see that the rotational parameter $\Omega$ indeed changes the conserved quantities $E^n$ and $I_2^n$, which is consistent with \eqref{eq2.6(b)}--\eqref{eq2.7(b)}.

\begin{figure}[htbp]
	\centering
	\subfigure[Velocities]{\centering
		\includegraphics[width=0.40\textwidth]{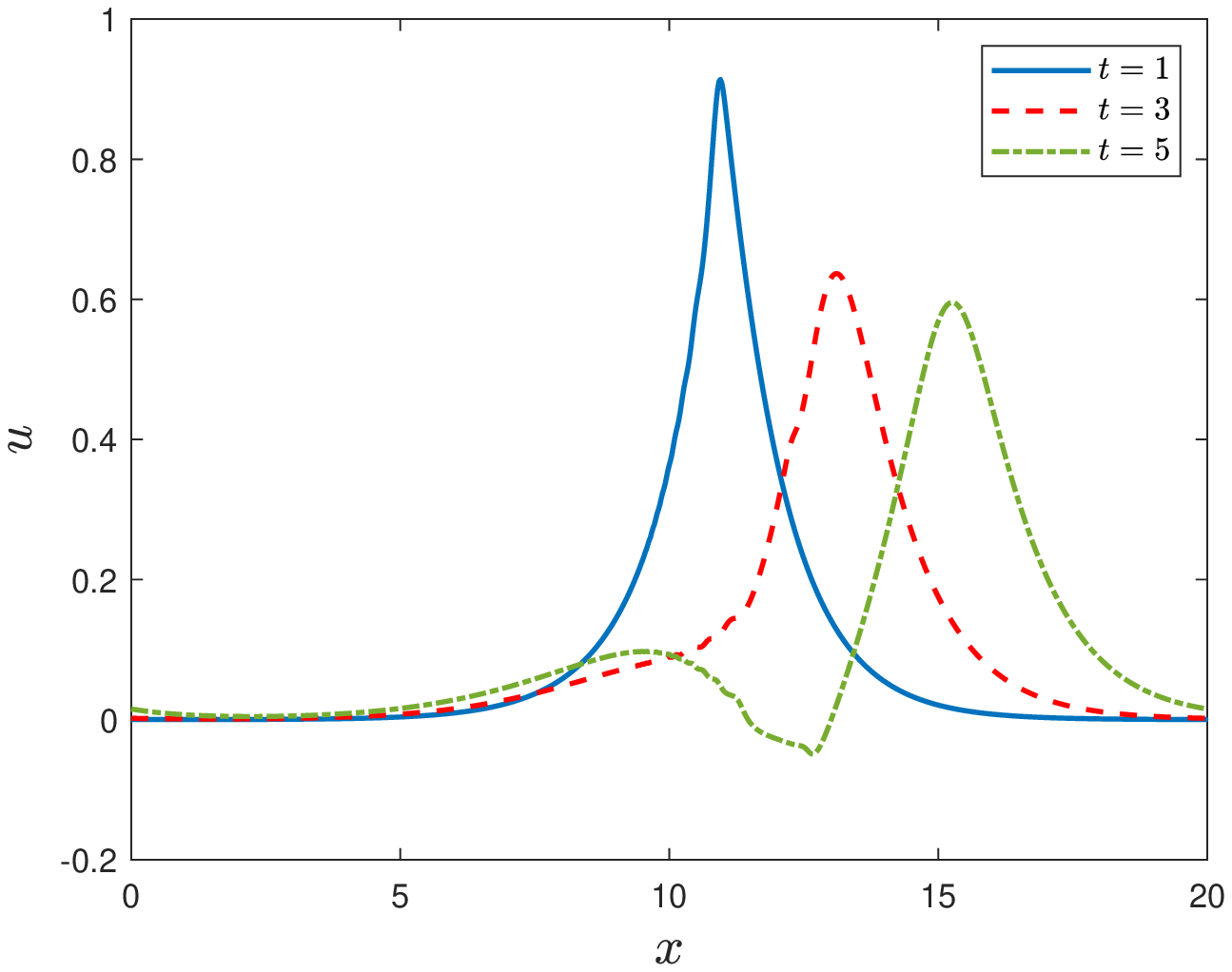}
	}\hspace{-6mm}\subfigure[Heights]{\centering
		\includegraphics[width=0.40\textwidth]{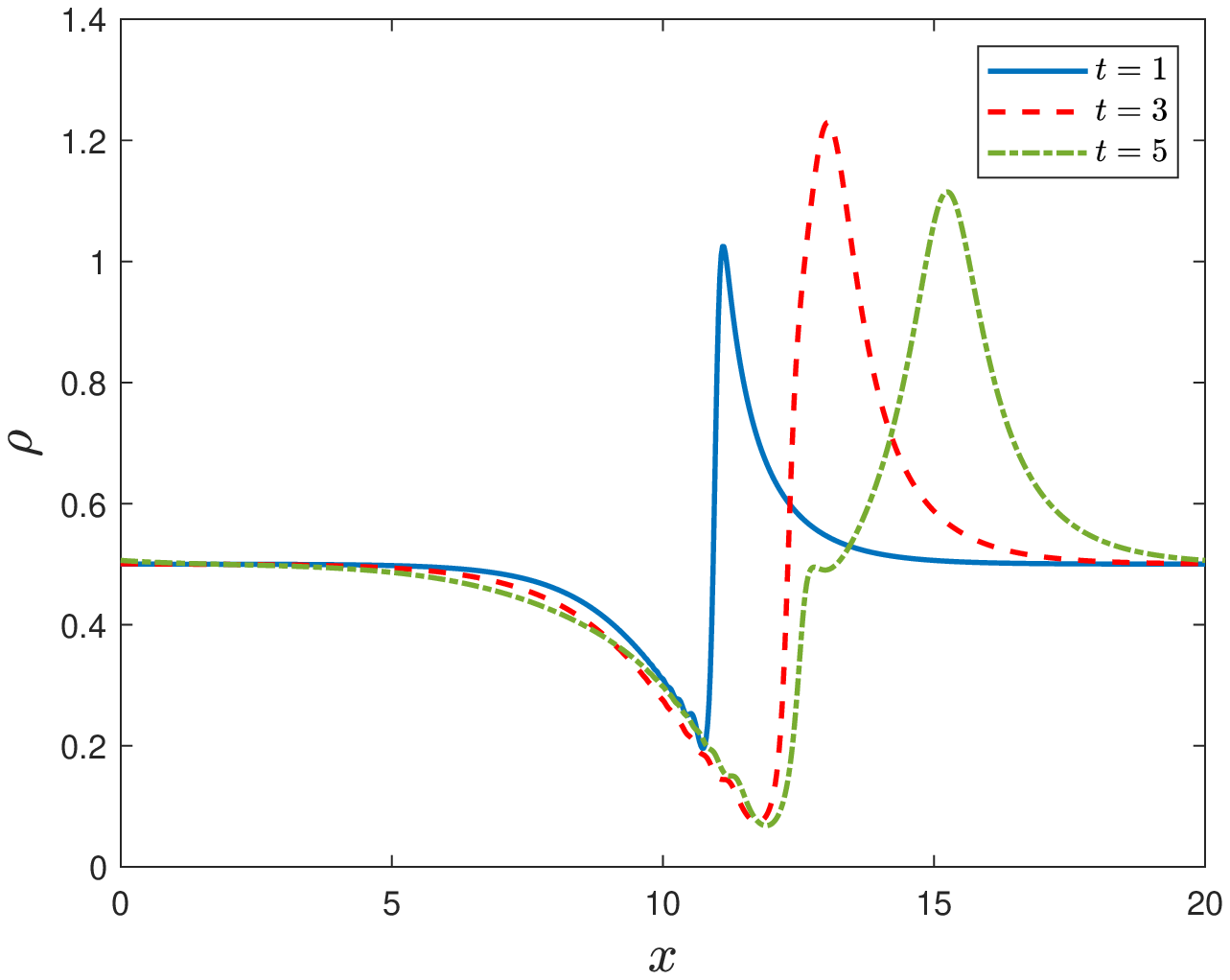}
	}
	\caption{Velocities $u(x,t)$ and heights $\rho(x,t)$ for the R2CH system in \textbf{Case} (\uppercase\expandafter{\romannumeral1}) computed by scheme \eqref{eq3.1}--\eqref{eq3.3} at three different times with stepsizes $h=0.025$ and $\tau = 0.0005$.} \label{fig6}
\end{figure}

\begin{figure}[htbp]
	\centering
	\subfigure[Velocities]{\centering
		\includegraphics[width=0.4\textwidth]{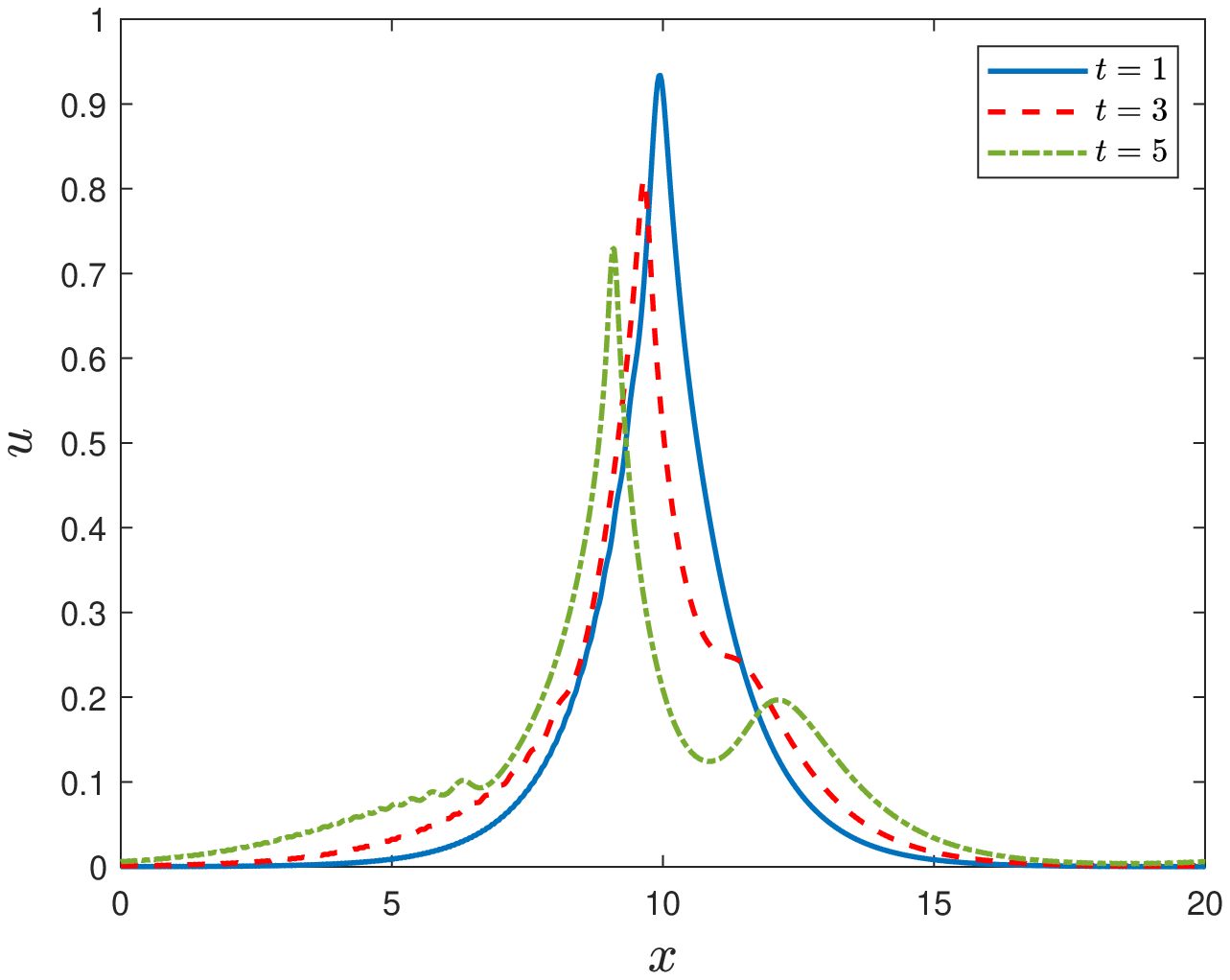}
	}\hspace{-6mm}\subfigure[Heights]{\centering
		\includegraphics[width=0.4\textwidth]{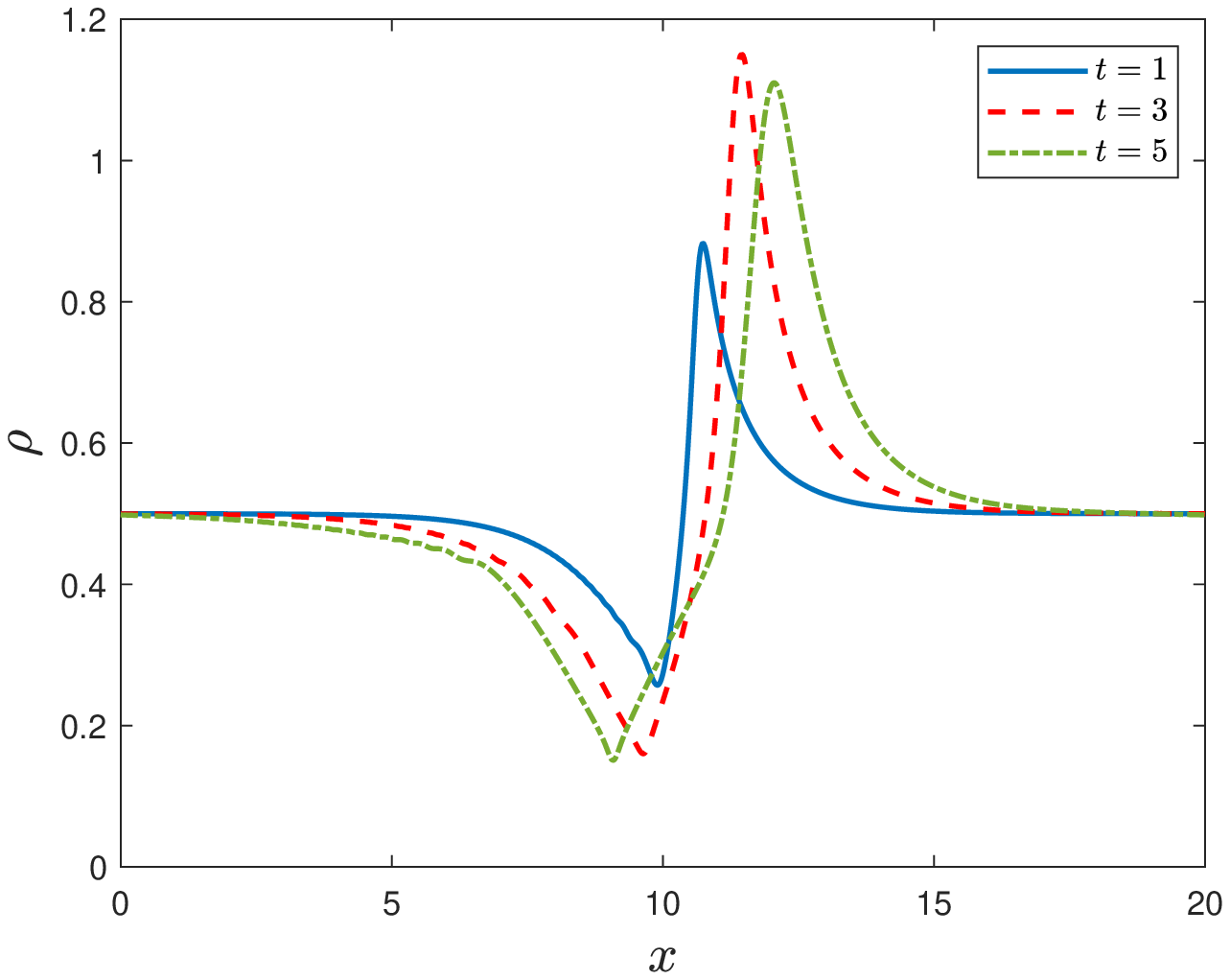}
	}
	\caption{Velocities $u(x,t)$ and heights $\rho(x,t)$  for the R2CH system in \textbf{Case} (\uppercase\expandafter{\romannumeral3}) computed by scheme \eqref{eq3.1}--\eqref{eq3.3} at three different times with stepsizes $h=0.025$ and $\tau = 0.0005$. } \label{fig7}
\end{figure}

\begin{figure}[htbp]
	\centering
	\subfigure[\textbf{Case} (\uppercase\expandafter{\romannumeral1})]{\centering
		\includegraphics[width=0.4\textwidth]{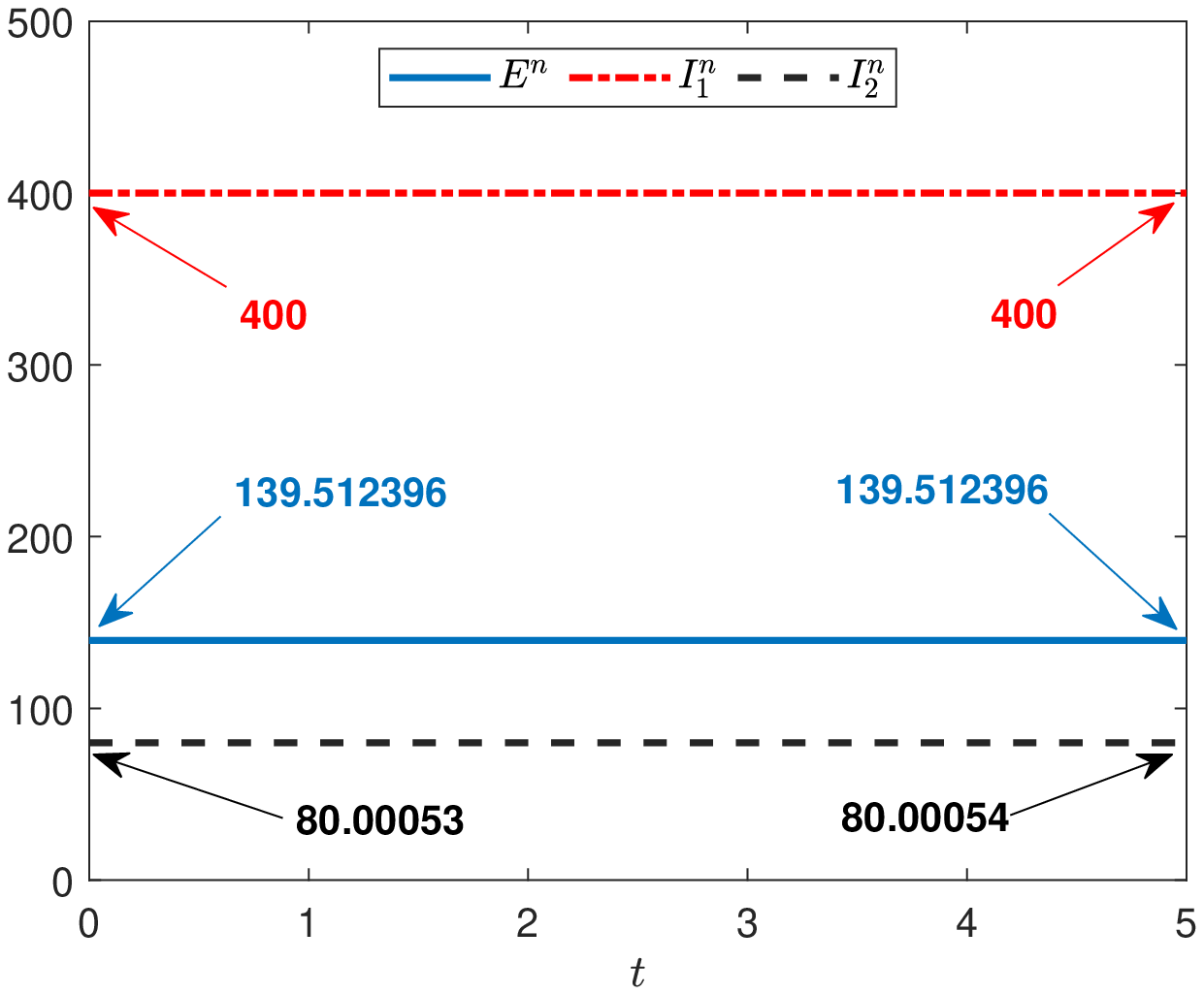}
	}\hspace{-6mm}\subfigure[\textbf{Case} (\uppercase\expandafter{\romannumeral3})]{\centering
		\includegraphics[width=0.4\textwidth]{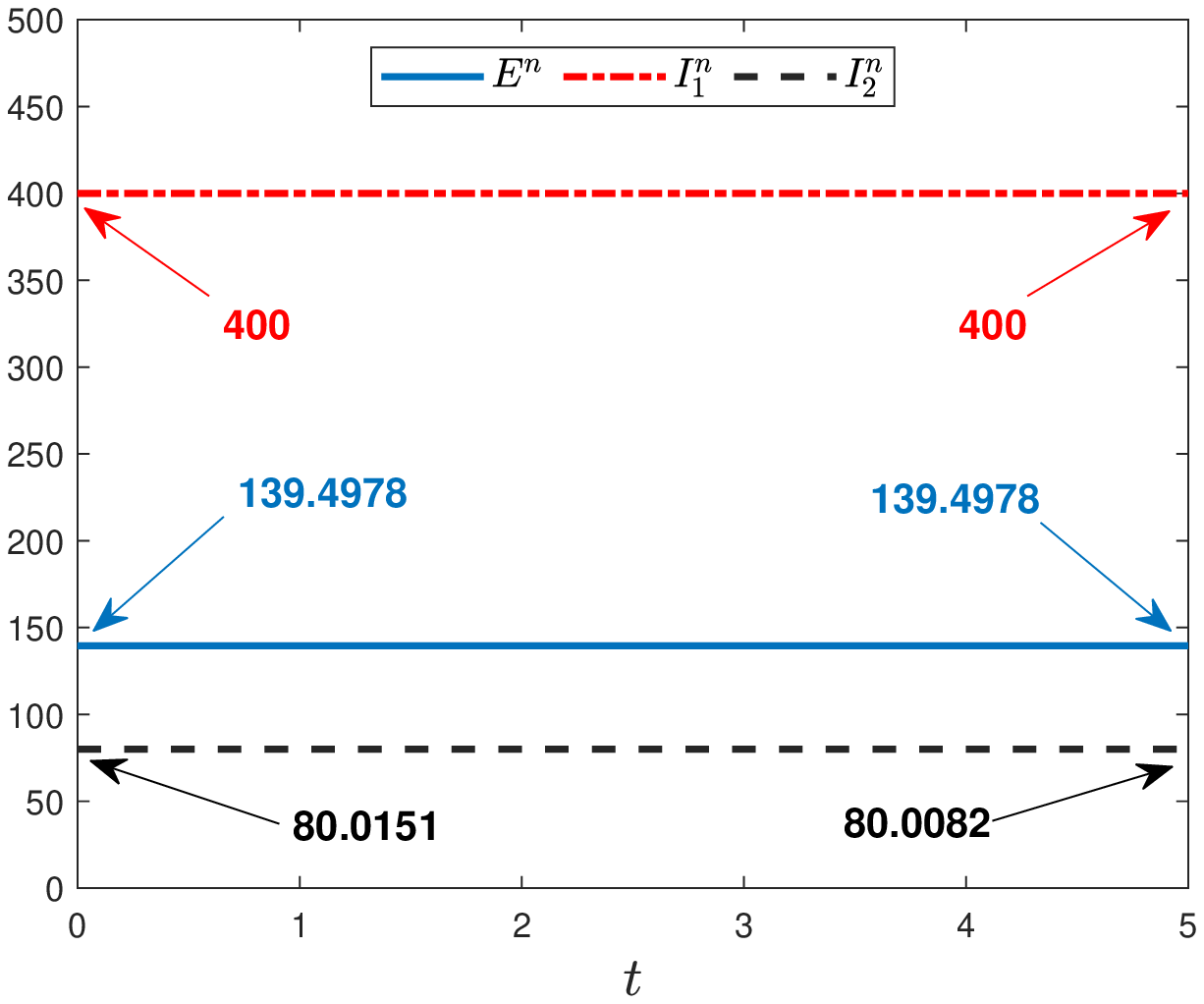}
	}
	\caption{Conserved quantities defined in \eqref{eq2.5(b)}--\eqref{eq2.7(b)} for the R2CH system in \textbf{Case} (\uppercase\expandafter{\romannumeral1}) and \textbf{Case} (\uppercase\expandafter{\romannumeral3}) with stepsizes $h=0.025$ and $\tau = 0.0005$.} \label{fig8}
\end{figure}

\begin{example}[\textbf{Peakon anti-peakon interaction--I}\cite{LQZS2017}] \label{Exam5.4} We consider a piecewise smooth initial conditions
	\begin{align*}
		u(x,0) = \left\{
		\begin{array}{ll}
			\displaystyle\frac{1}{2\sinh (\frac{1}{4})}\sinh(x),\quad  0 \leqslant x \leqslant \frac{1}{4},   \\  [3\jot]
			\displaystyle\frac{1}{\sinh(-\frac{1}{2})}\sinh(x-\frac{1}{2}), \quad \frac{1}{4} < x \leqslant \frac{3}{4},  \qquad \rho(x,0) = 1.5.\\  [3\jot]
			\displaystyle\frac{1}{2\sinh(\frac{1}{4})}\sinh(x-1), \quad \frac{3}{4} < x < 1, \\
		\end{array}
		\right.
	\end{align*}
\end{example}
To test the influence of $\Omega$ on the solution of the R2CH system \eqref{eq1.3}--\eqref{eq1.4}, the parameters in \textbf{Cases} (\uppercase\expandafter{\romannumeral1}) and (\uppercase\expandafter{\romannumeral2}) are selected for calculation in this example. Figure \ref{fig11} and Figure \ref{fig12} respectively display the profiles of solutions of the velocity variable $u(x,t)$ and the height variable $\rho(x,t)$ at different instants of time, which enables us to observe how the solutions evolve over time. In addition, it can be observed from Figure \ref{fig11}(a) and Figure \ref{fig12}(a) that $u(x,t)$ at initial time is jump discontinuous at $x=1/4$ and $x=3/4$. Compared with \textbf{Case} (\uppercase\expandafter{\romannumeral1}), the symmetry is broken in \textbf{Case} (\uppercase\expandafter{\romannumeral2}).  Figure \ref{fig13} shows the results of discrete conserved quantities defined in \eqref{eq2.5(b)}--\eqref{eq2.7(b)} for \textbf{Cases} (\uppercase\expandafter{\romannumeral1}) and  (\uppercase\expandafter{\romannumeral2}). It is observed that the total momentum is zero when the solution is symmetric, while it is nonzero when the solution is asymmetric. Moreover, we portray the evolution of the peakon-antipeakon for the velocity $u(x,t)$ and the height $\rho(x,t)$ in Figures \ref{fig14}--\ref{fig15}. It can be seen from Figure \ref{fig14} that the evolution of solutions presents a certain periodicity in the long time simulation,
while the evolution of solutions is drawn in a short time in Figure \ref{fig15} since it will blow up in a slightly long time.

\begin{figure}[htbp]
	\centering
	\subfigure[Velocities]{\centering
		\includegraphics[width=0.4\textwidth]{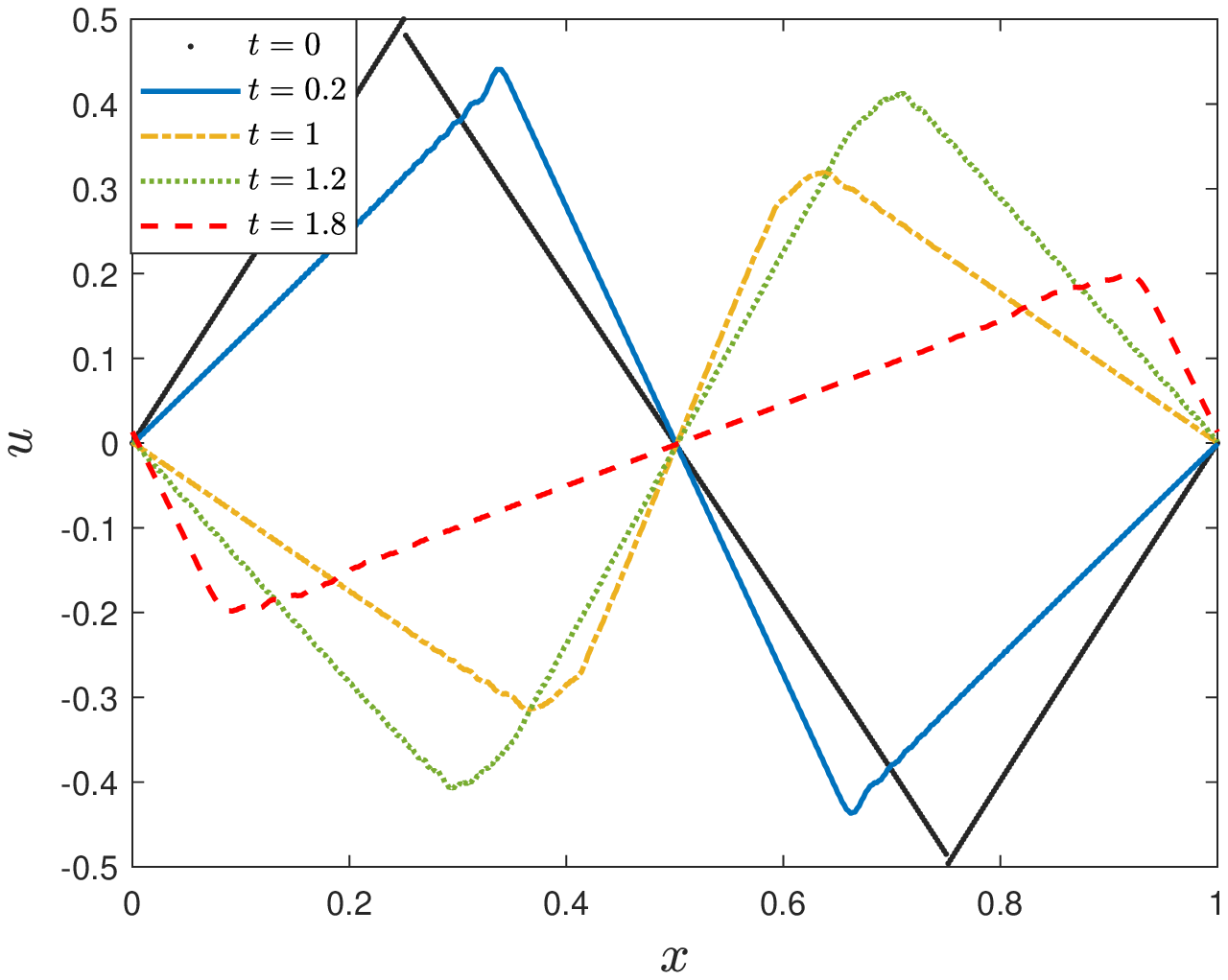}
	}\hspace{-6mm}\subfigure[Heights]{\centering
		\includegraphics[width=0.4\textwidth]{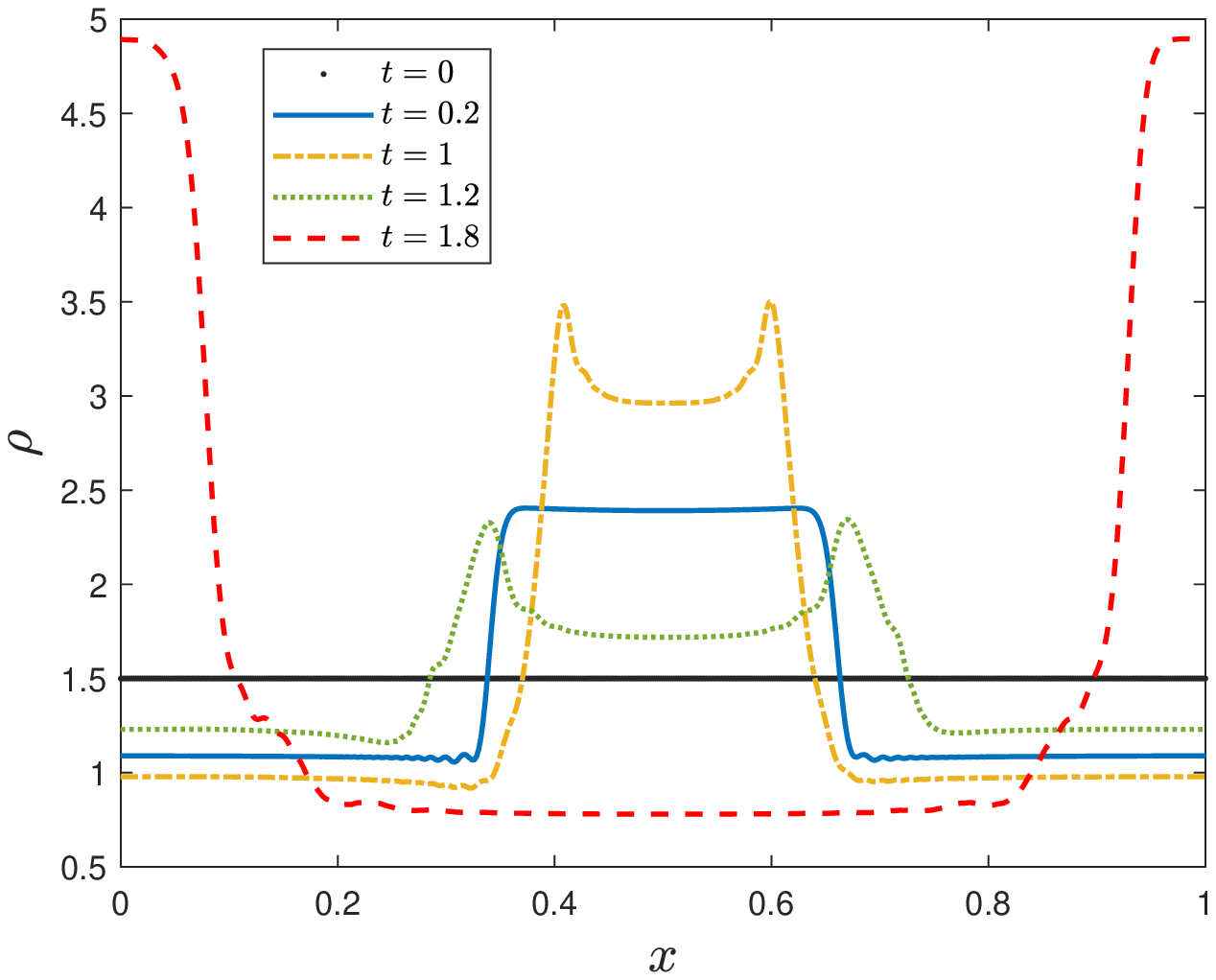}
	}
	\caption{ Velocities $u(x,t)$ and heights $\rho(x,t)$ for the R2CH system in \textbf{Case} (\uppercase\expandafter{\romannumeral1}) computed by scheme \eqref{eq3.1}--\eqref{eq3.3} at five different times with stepsizes $h=0.002$ and $\tau = 0.001$.} \label{fig11}
\end{figure}
\begin{figure}[htbp]
	\centering
	\subfigure[Velocities]{\centering
		\includegraphics[width=0.4\textwidth]{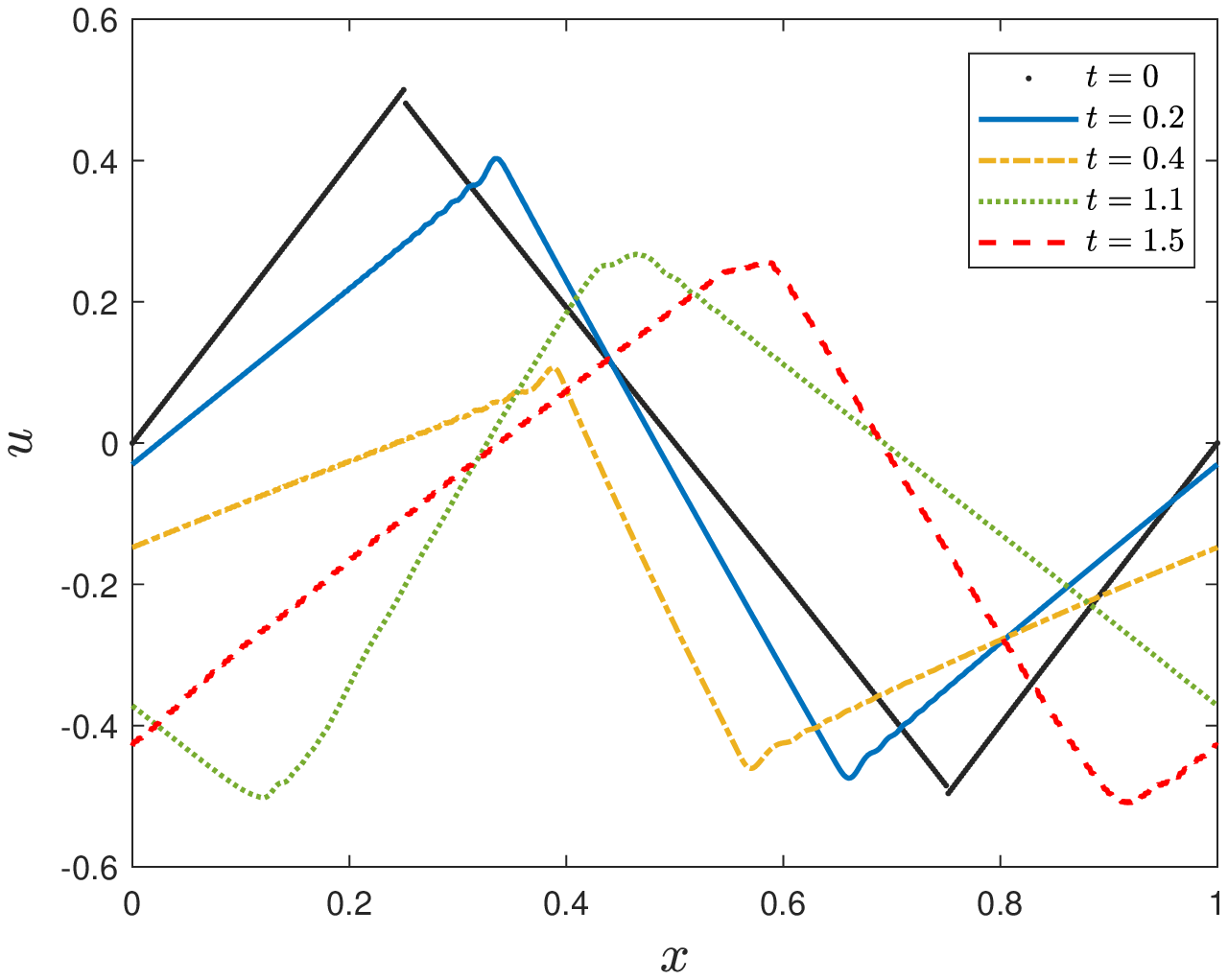}
	}\hspace{-6mm}\subfigure[Heights]{\centering
		\includegraphics[width=0.4\textwidth]{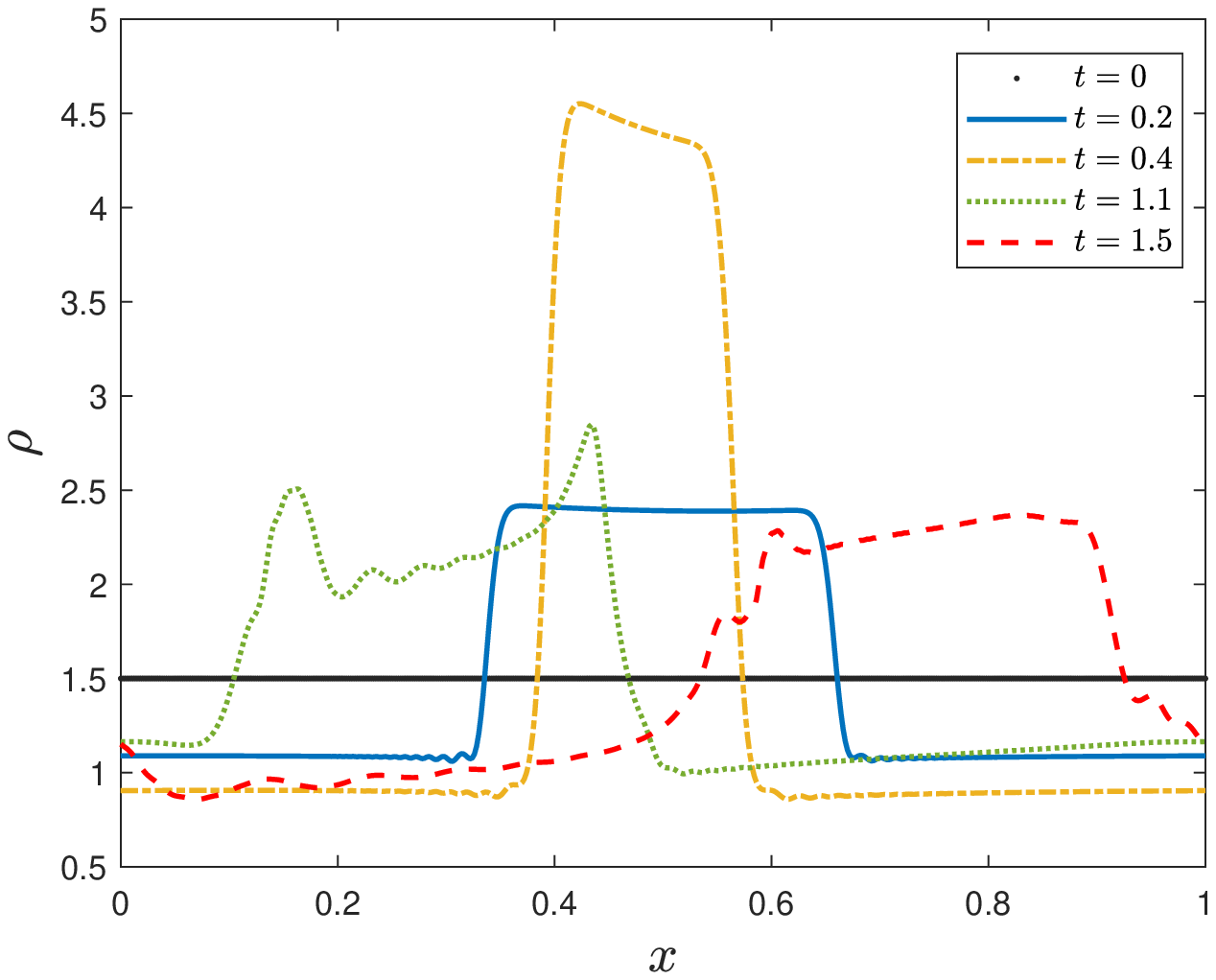}
	}
	\caption{ Velocities $u(x,t)$ and heights $\rho(x,t)$ for the R2CH system in \textbf{Case} (\uppercase\expandafter{\romannumeral2}) computed by scheme \eqref{eq3.1}--\eqref{eq3.3} at five different times with stepsizes $h=0.002$ and $\tau = 0.0005$.} \label{fig12}
\end{figure}
\begin{figure}[htbp]
	\centering
	\subfigure[\textbf{Case} (I)]{\centering
		\includegraphics[width=0.4\textwidth]{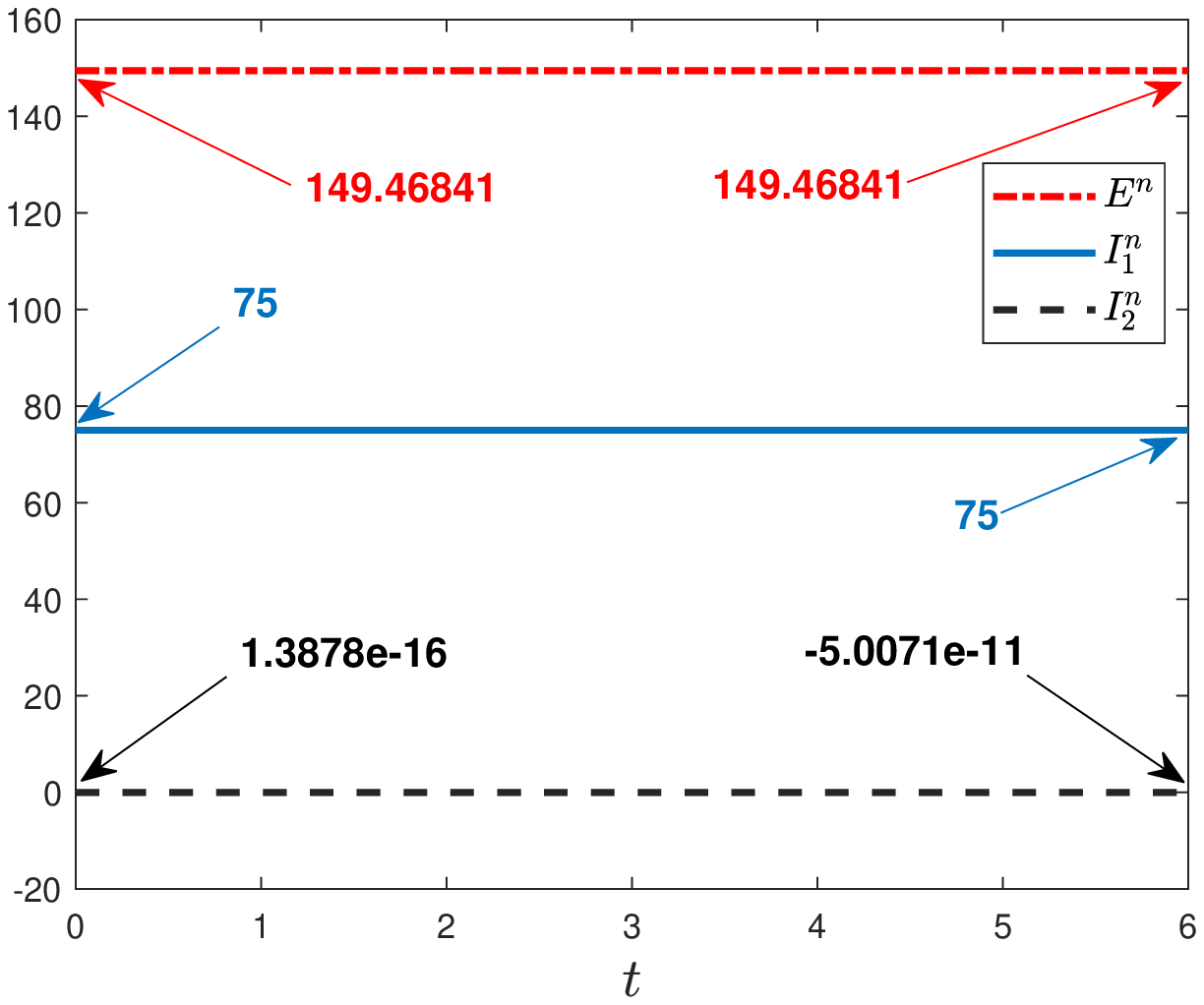}
	}\hspace{-6mm}\subfigure[\textbf{Case} (II)]{\centering
		\includegraphics[width=0.4\textwidth]{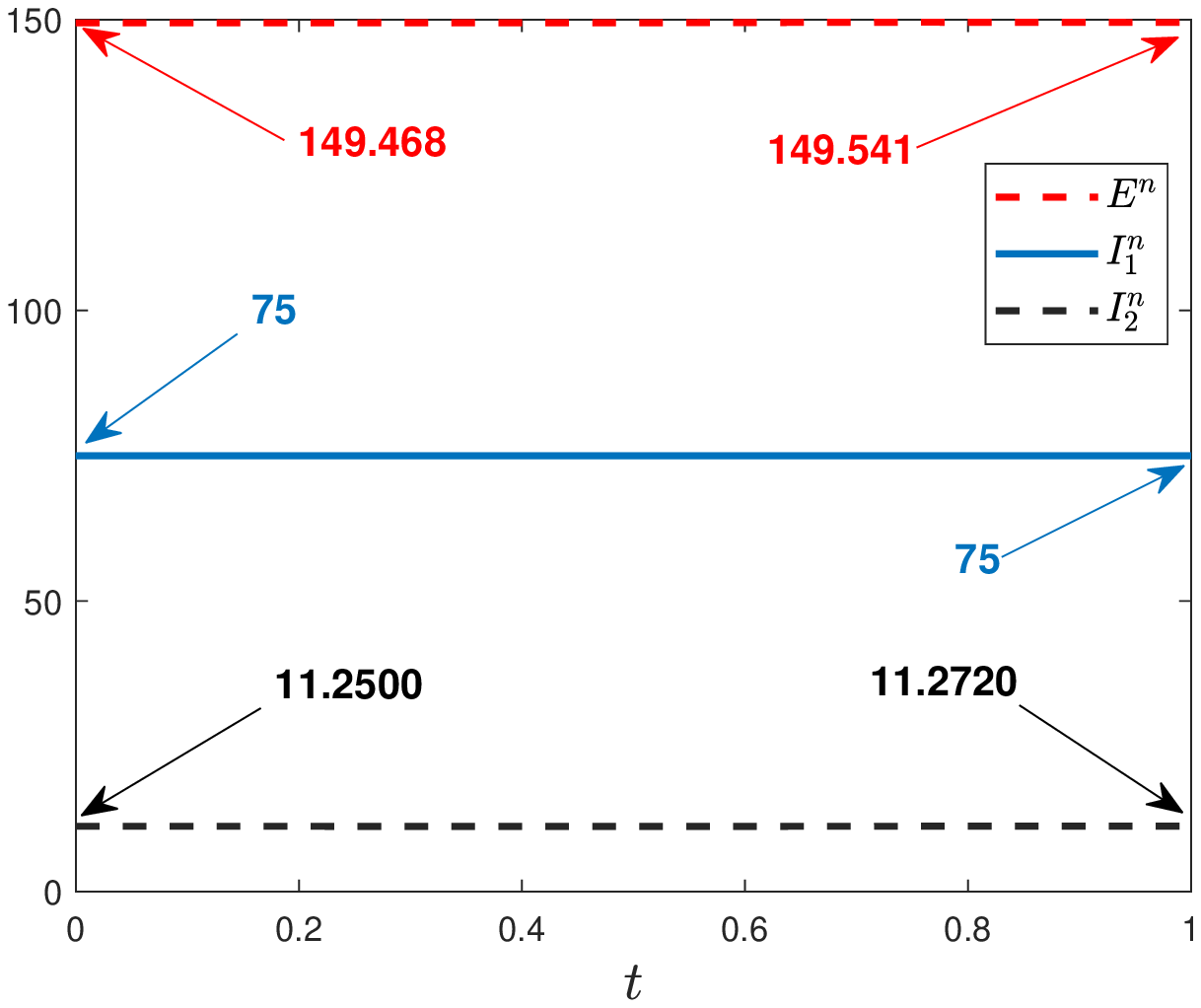}
	}
	\caption{Conserved quantities defined in \eqref{eq2.5(b)}--\eqref{eq2.7(b)} for the R2CH system in \textbf{Case} (\uppercase\expandafter{\romannumeral1}) and \textbf{Case} (\uppercase\expandafter{\romannumeral2})  with stepsizes $h=0.02$ and $\tau = 0.0005$.}  \label{fig13}
\end{figure}
\begin{figure}[htbp]
	\centering
	\subfigure[Velocity variable, view(-20,60)]{\centering
		\includegraphics[width=0.4\textwidth]{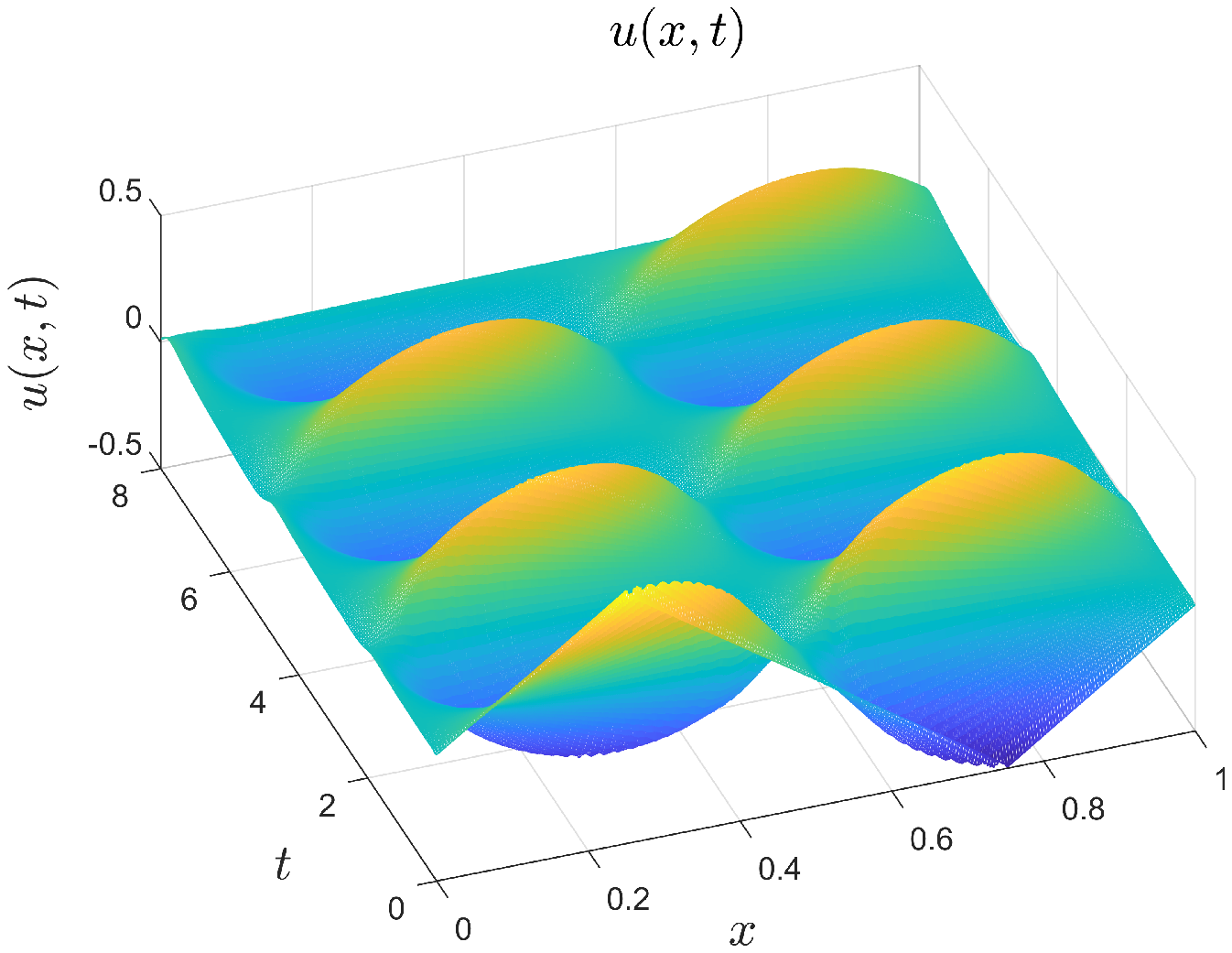}
	}\subfigure[Height variable, view(-20,60)]{\centering
		\includegraphics[width=0.4\textwidth]{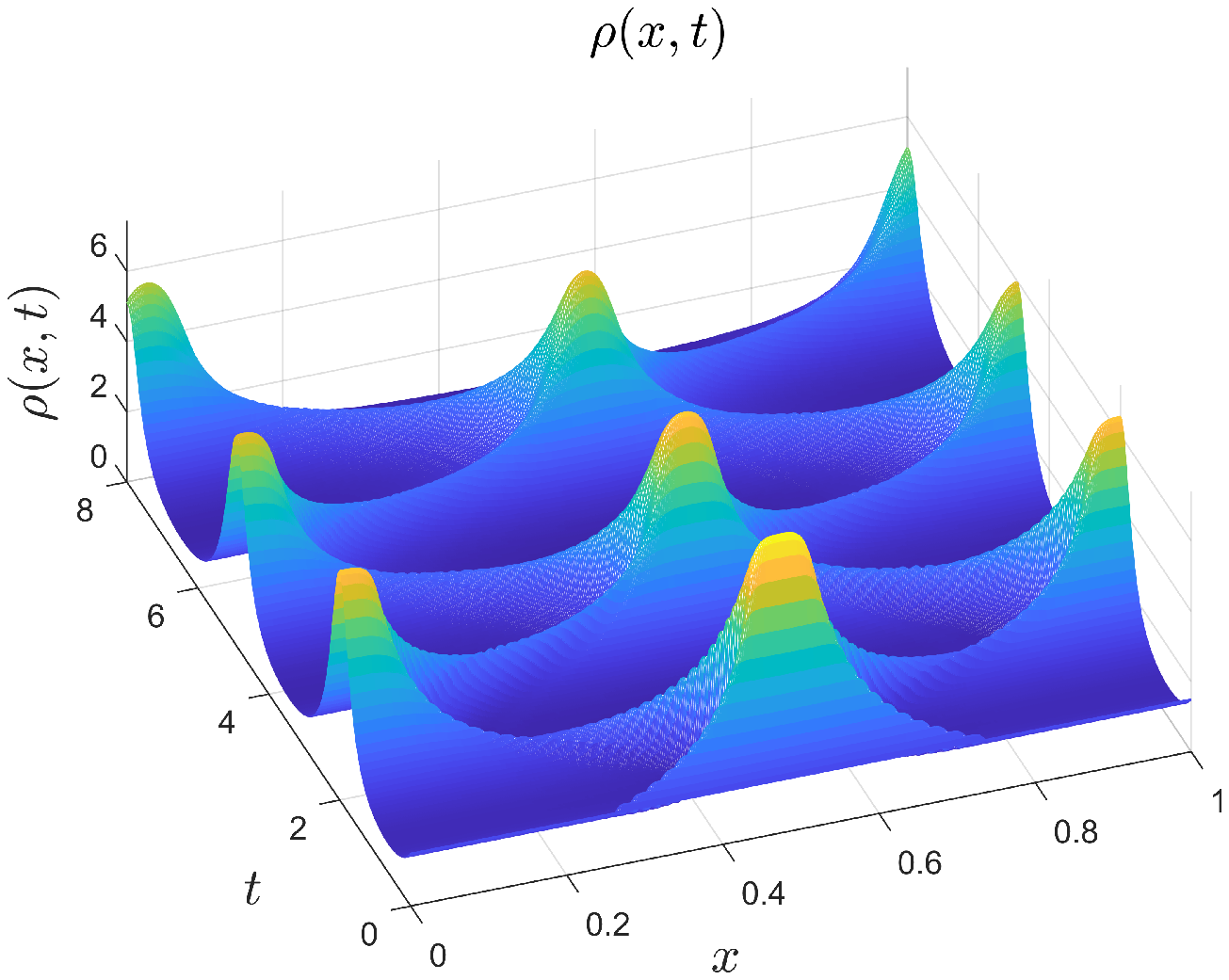}
	}
	\caption{The predicted peakon solutions for the R2CH system in \textbf{Case} (\uppercase\expandafter{\romannumeral1})
		show the evolution of the velocity $u(x,t)$ and the height $\rho(x,t)$ with stepsizes $h=0.002$ and $\tau=0.001$. }     \label{fig14}
\end{figure}
\begin{figure}[htbp]
	\centering
	\subfigure[Velocity variable, view(-20,60)]{\centering
		\includegraphics[width=0.4\textwidth]{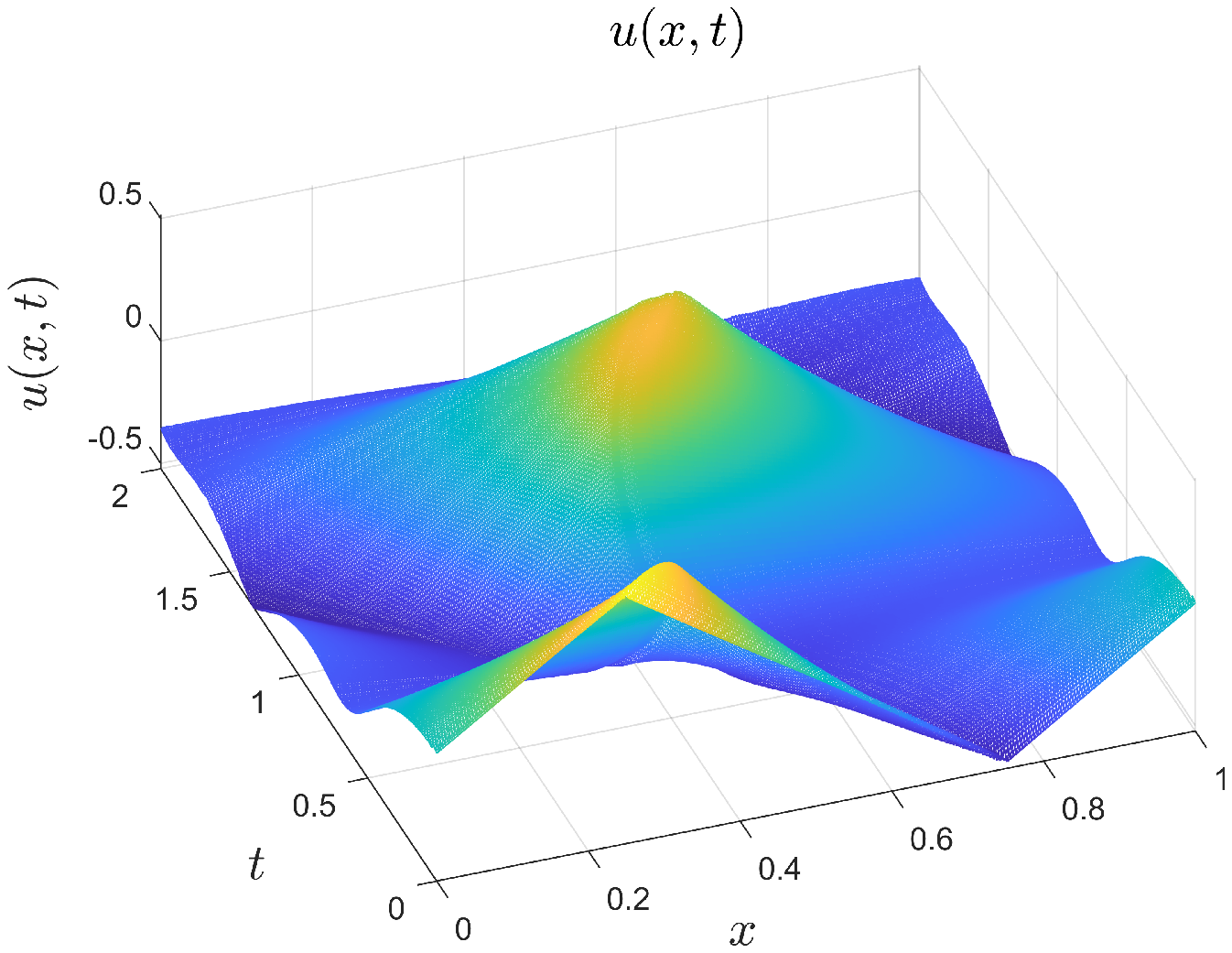}
	}\subfigure[Height variable, view(-20,60)]{\centering
		\includegraphics[width=0.4\textwidth]{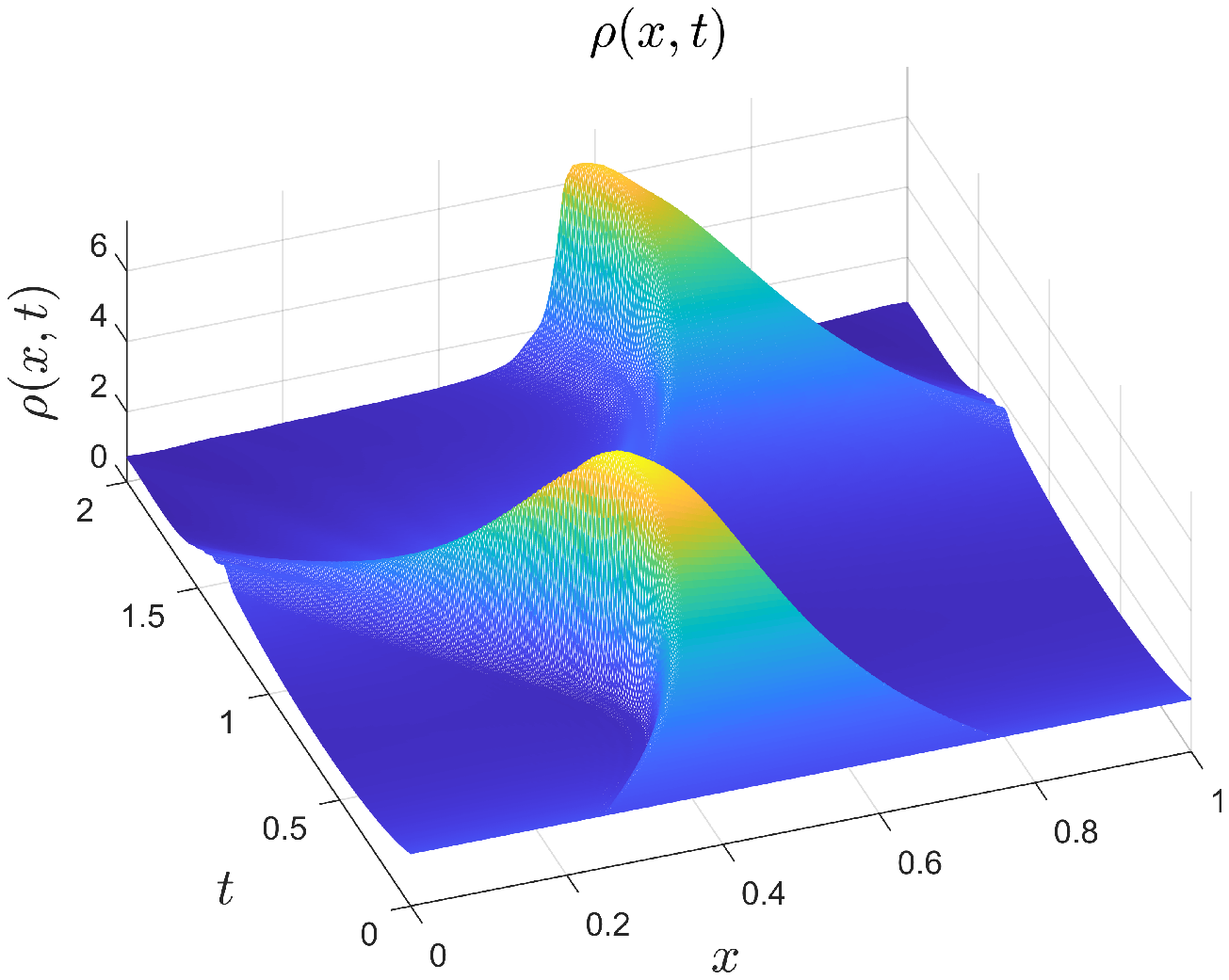}
	}\caption{The predicted peakon solutions for the R2CH system in \textbf{Case} (\uppercase\expandafter{\romannumeral2})
		show the evolution of the velocity $u(x,t)$ and the height $\rho(x,t)$ with stepsizes $h=0.002$ and $\tau=0.001$. }   \label{fig15}
\end{figure}

\begin{example}
	[\textbf{Peakon anti-peakon \label{Exam5.5} interaction--II}\cite{CKL2020}, \cite{LP2016}]
	Here we study the case of a peakon anti-peakon interaction over a long time.
	For this purpose, we consider the following initial data
	\begin{align*}
		u(x,0) = p_1 e^{-|x-x_1|} + p_2 e^{-|x-x_2|}, \quad  \rho (x,0) = 0.5, 
	\end{align*}
	where $x_1 = -5$ and $x_2 = 5$ respectively represent the position of the peak and the trough,
	$p_1 = 1$, $p_2 = -1$ and the computation domain is set to be $[-20,20]$.
\end{example}
We depict the evolution of the solutions for the R2CH system \eqref{eq1.3}--\eqref{eq1.4}, where the parameters are selected as \textbf{Cases} (\uppercase\expandafter{\romannumeral1}) and (\uppercase\expandafter{\romannumeral3}) by taking $h=0.05$ and $\tau=0.0005$ for calculation. Figures \ref{fig16}--\ref{fig19} portray the evolution behaviors of velocity $u(x,t)$ and height $\rho(x,t)$ under these two cases. It easily observes that for sufficiently long periods of time, the peak solutions will have an elastic collision, so that we obtain a dissipative solution for the R2CH system \eqref{eq1.3}--\eqref{eq1.4}. By comparing Figures \ref{fig16}--\ref{fig17} with Figures \ref{fig18}--\ref{fig19}, we find that \textbf{Case} (\uppercase\expandafter{\romannumeral1}) still maintains symmetric, while the symmetry in \textbf{Case} (\uppercase\expandafter{\romannumeral3}) is obviously broken. Similar to Example \ref{Exam5.4}, we assert that the total momentum remains zero in \textbf{Case} (\uppercase\expandafter{\romannumeral1}) because the two peak solutions have the same magnitude but opposite signs. The momentum in \textbf{Case} (\uppercase\expandafter{\romannumeral3}) is nonzero because the solution in \textbf{Case} (\uppercase\expandafter{\romannumeral3}) loses its symmetry. We verify this through the calculation, which is not listed here for brevity. It is worth mentioning that, since $\rho(x,0) > 0$, it can be proved that $\rho(x,t)$ remains strictly positive and that the solution retains the regularity of the initial data, see e.g., \cite{GHR2012}. Particularly, $\rho(x,t)$ remains bounded. We observe that the numerical solution for the height $\rho(x,t)$ still keep positive. Ultimately, the evolution of the velocity $u(x,t)$ and the height $\rho(x,t)$ for these two cases in a longer time ($t=35$) is depicted in Figures \ref{fig20}--\ref{fig21}. We clearly see the symmetry in \textbf{Case} (\uppercase\expandafter{\romannumeral1}) and asymmetry in \textbf{Case} (\uppercase\expandafter{\romannumeral3}). These phenomena for the asymmetrical case are displayed firstly in current paper.

\begin{figure}[htbp]
	\subfigure[$t=1$]{\centering
		\includegraphics[width=0.35\textwidth]{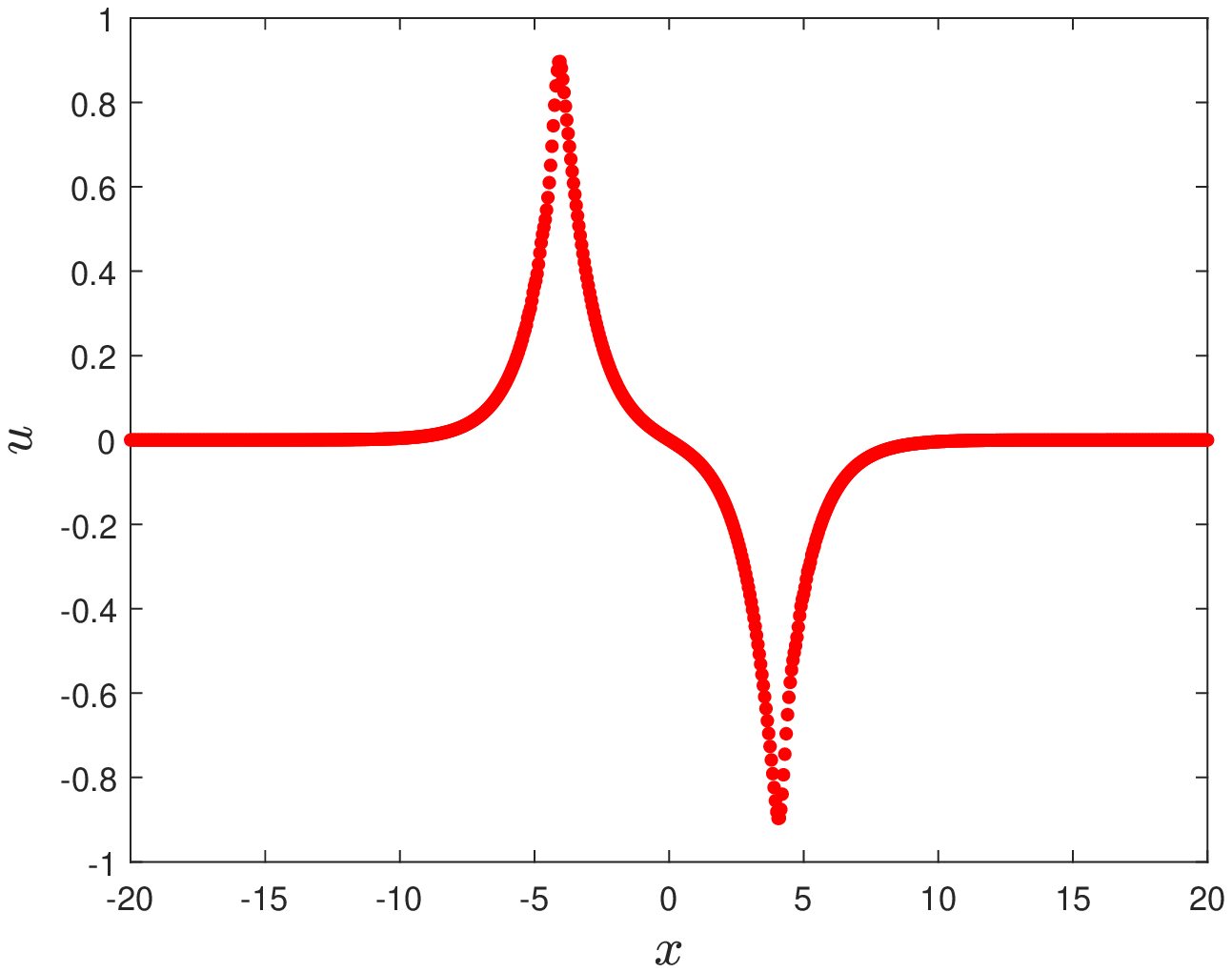}
	}\hspace{-5.3mm}\subfigure[$t=3$]{\centering
		\includegraphics[width=0.35\textwidth]{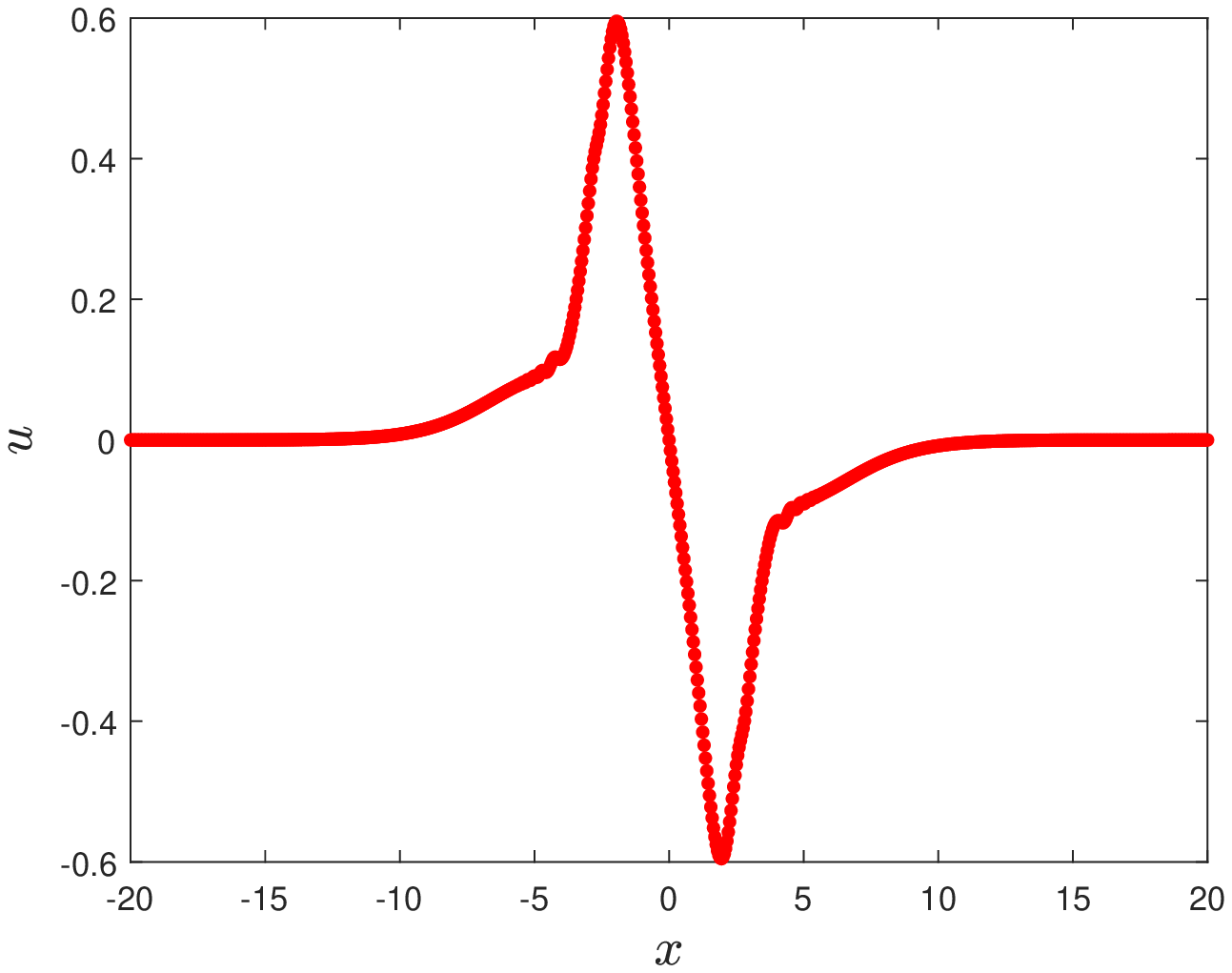}
	}\hspace{-5.3mm}\subfigure[$t=5$]{\centering
		\includegraphics[width=0.35\textwidth]{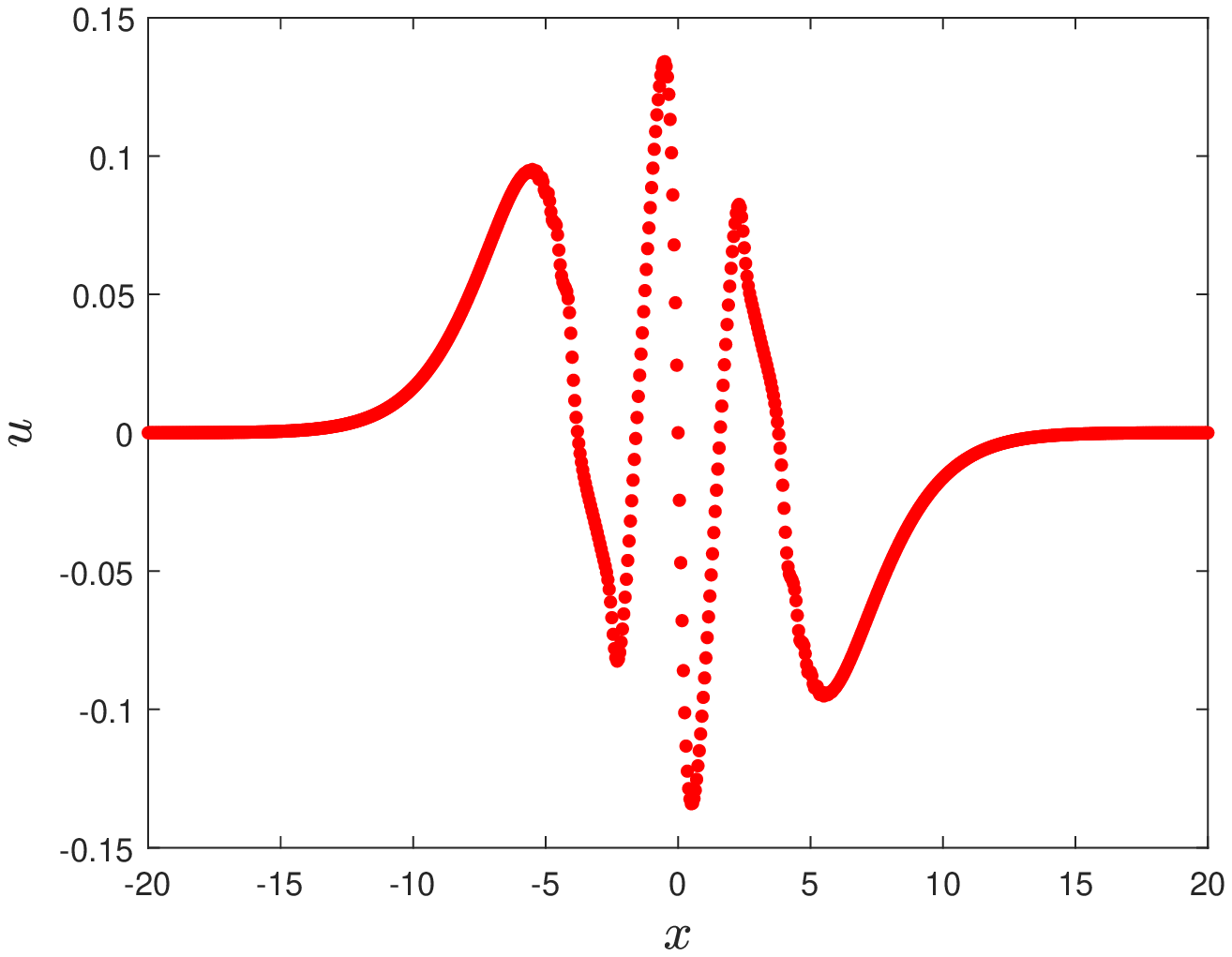}
	}\vspace{-2mm}
	\subfigure[$t=6$]{\centering
		\includegraphics[width=0.35\textwidth]{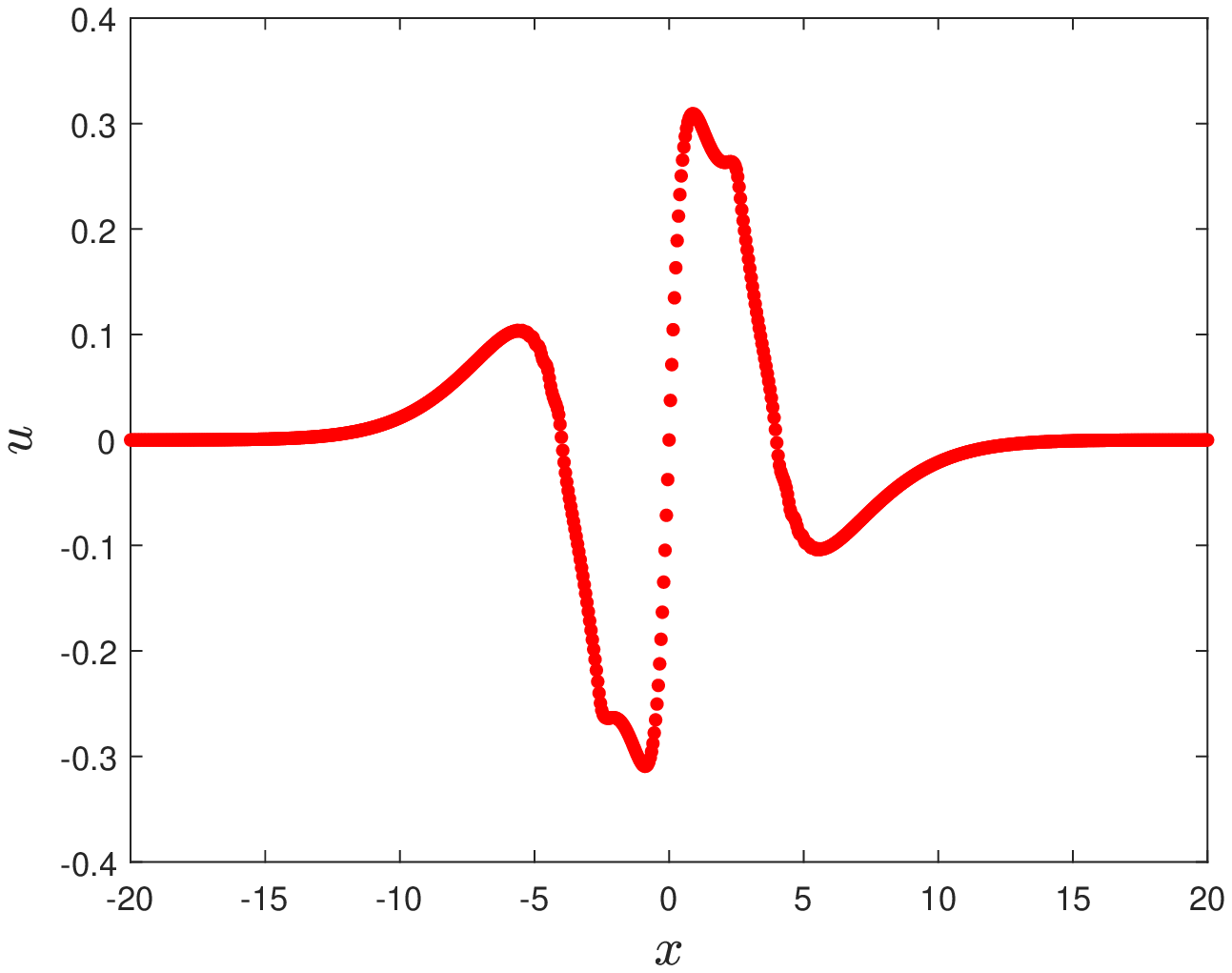}
	}\hspace{-5.3mm}\subfigure[$t=8$]{\centering
		\includegraphics[width=0.35\textwidth]{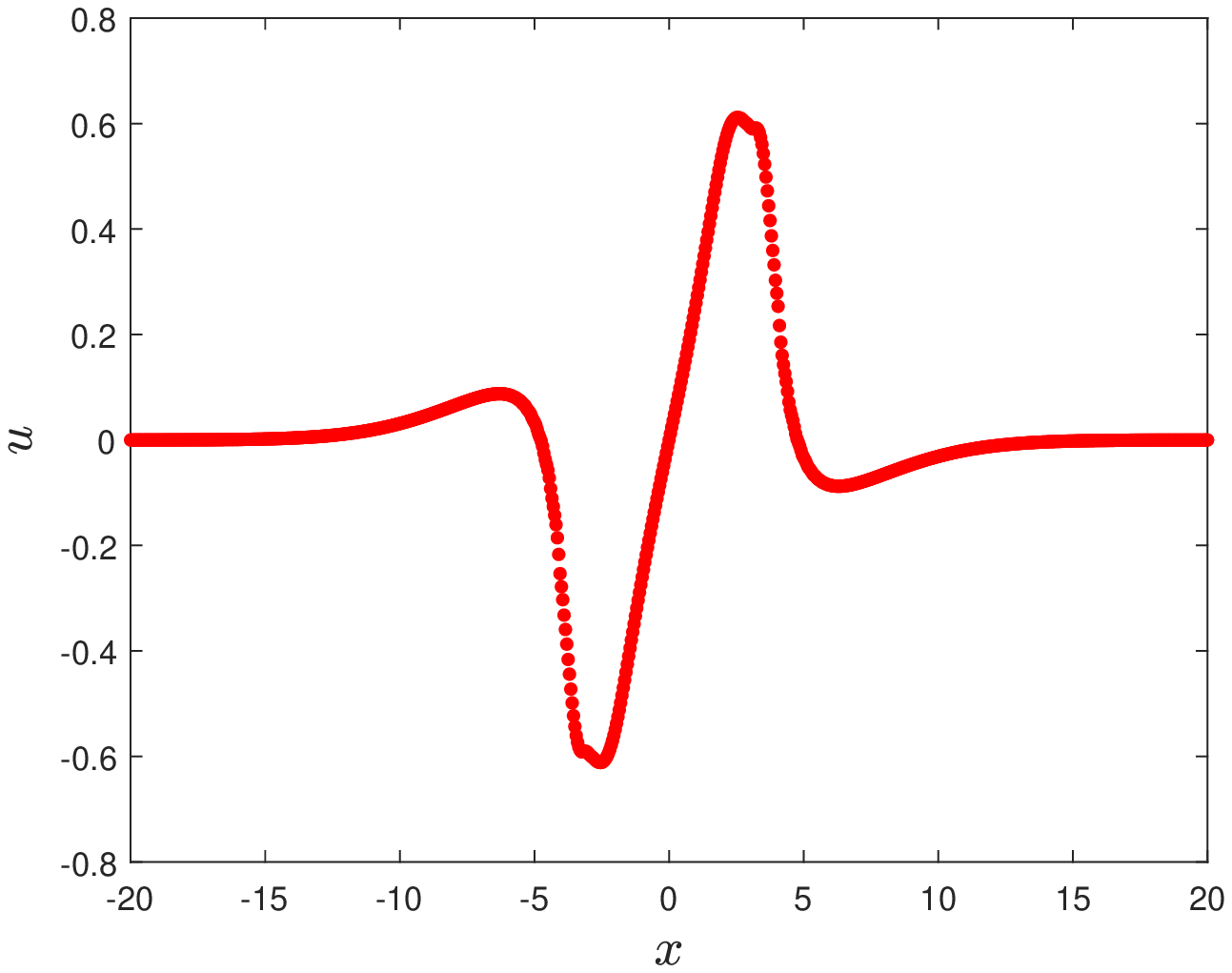}
	}\hspace{-5.3mm}\subfigure[$t=10$]{\centering
		\includegraphics[width=0.35\textwidth]{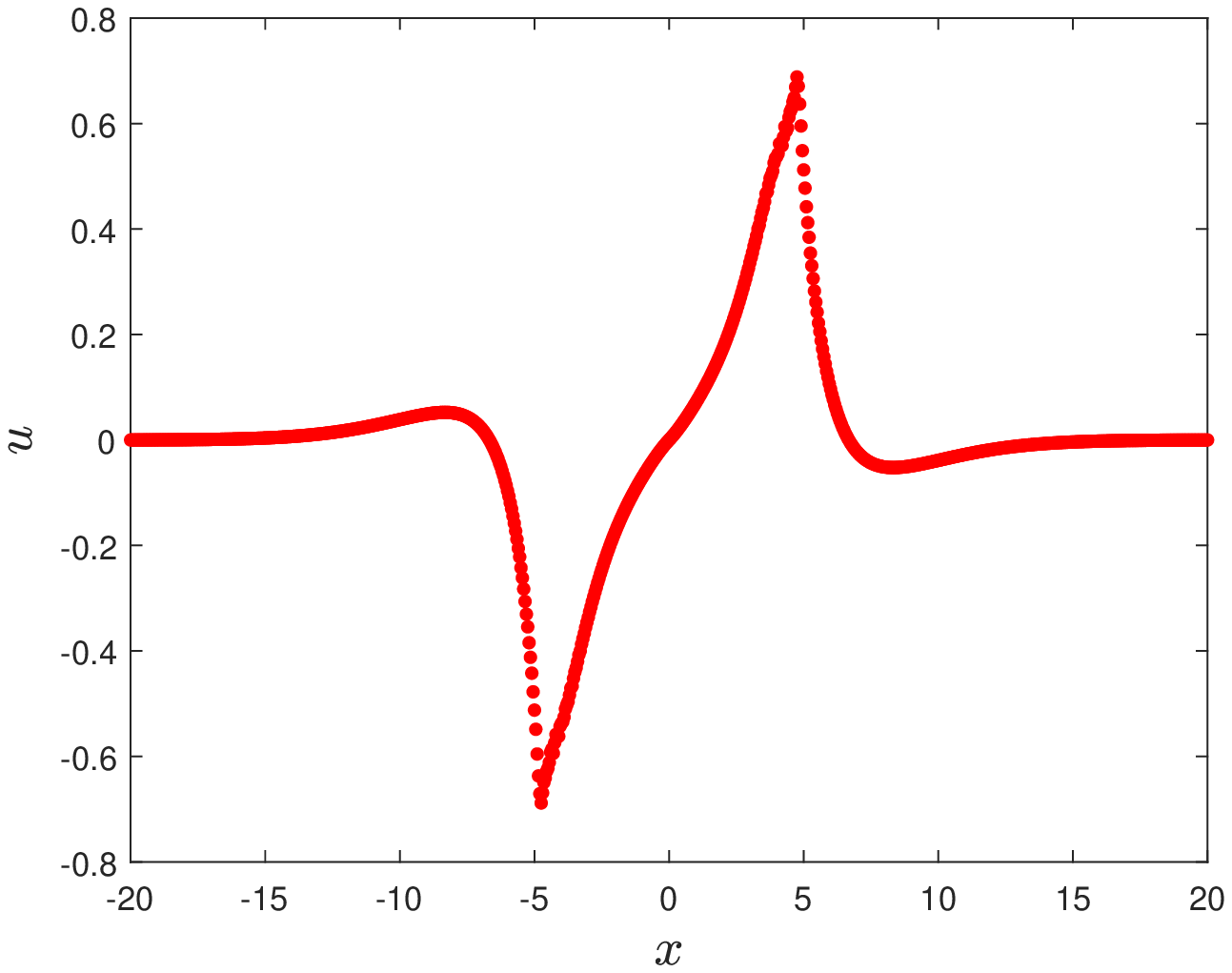}
	}\vspace{-2mm}
	\caption{The velocities $u(x,t)$ for the R2CH system in \textbf{Case} (\uppercase\expandafter{\romannumeral1}) at different times with stepsizes $h=0.05$ and $\tau = 0.0005$.} \label{fig16}
\end{figure}
\begin{figure}[htbp]
\vspace{-2mm}
	\subfigure[$t=1$]{\centering
		\includegraphics[width=0.35\textwidth]{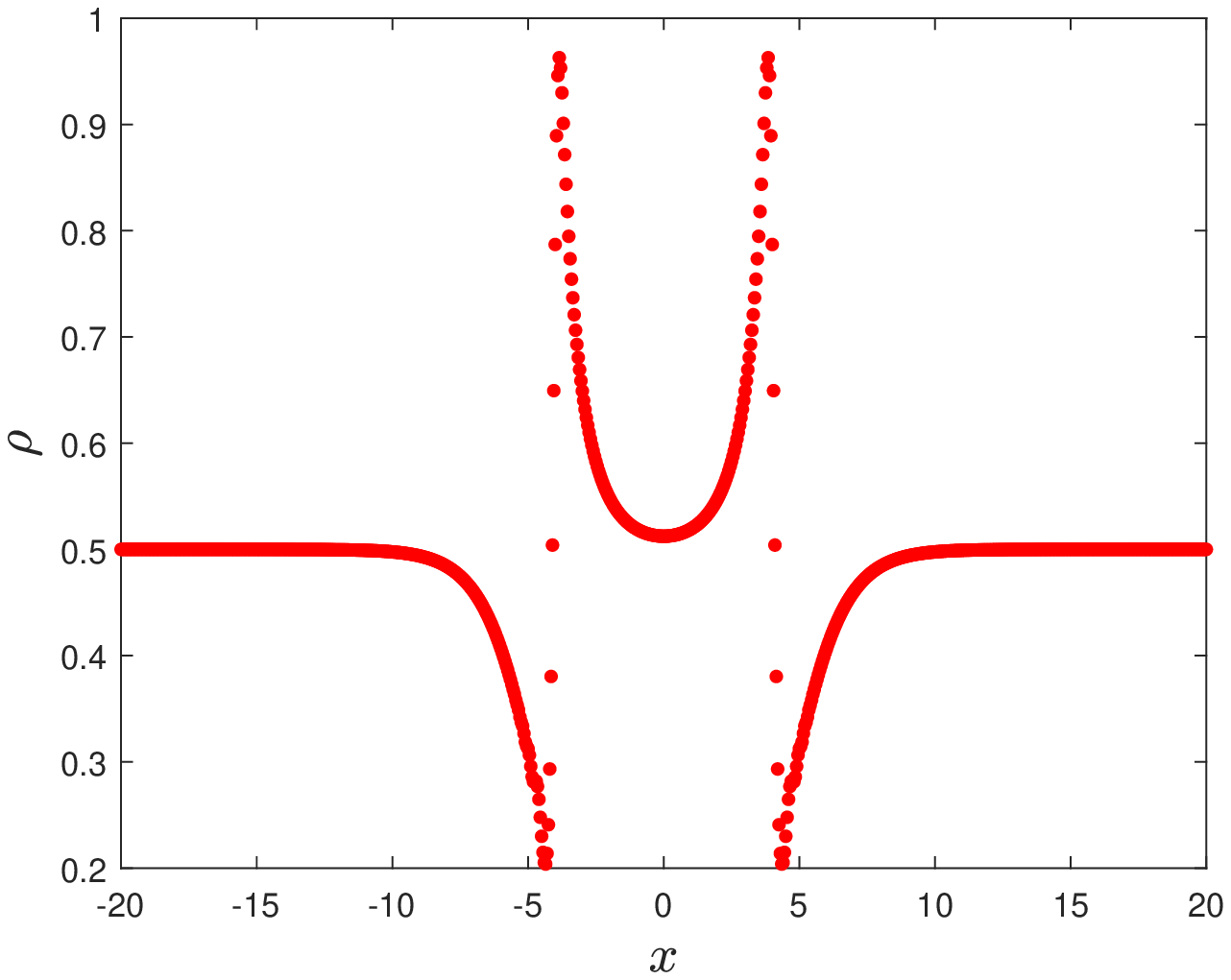}
	}\hspace{-5.3mm}\subfigure[$t=3$]{\centering
		\includegraphics[width=0.35\textwidth]{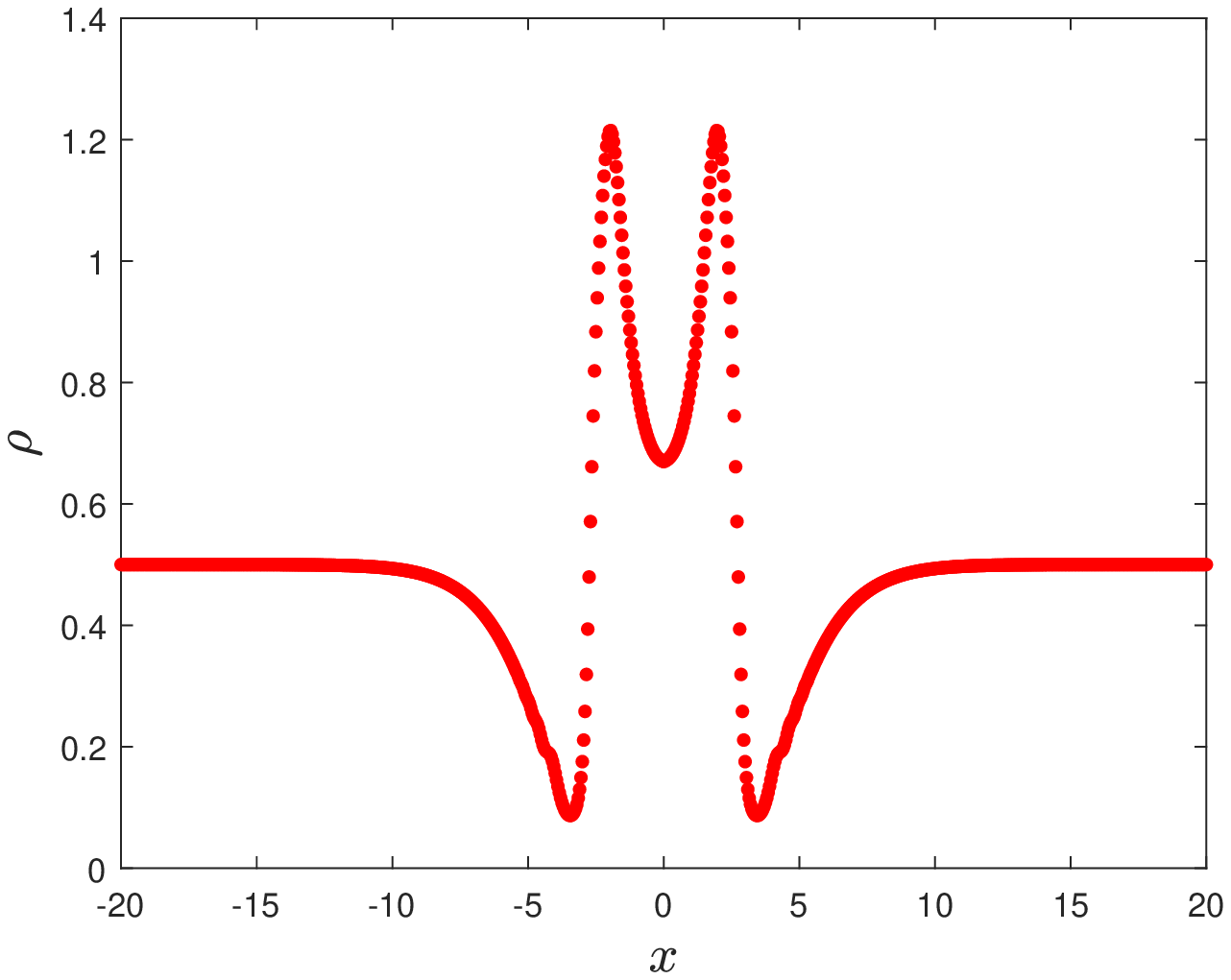}
	}\hspace{-5.3mm}\subfigure[$t=5$]{\centering
		\includegraphics[width=0.35\textwidth]{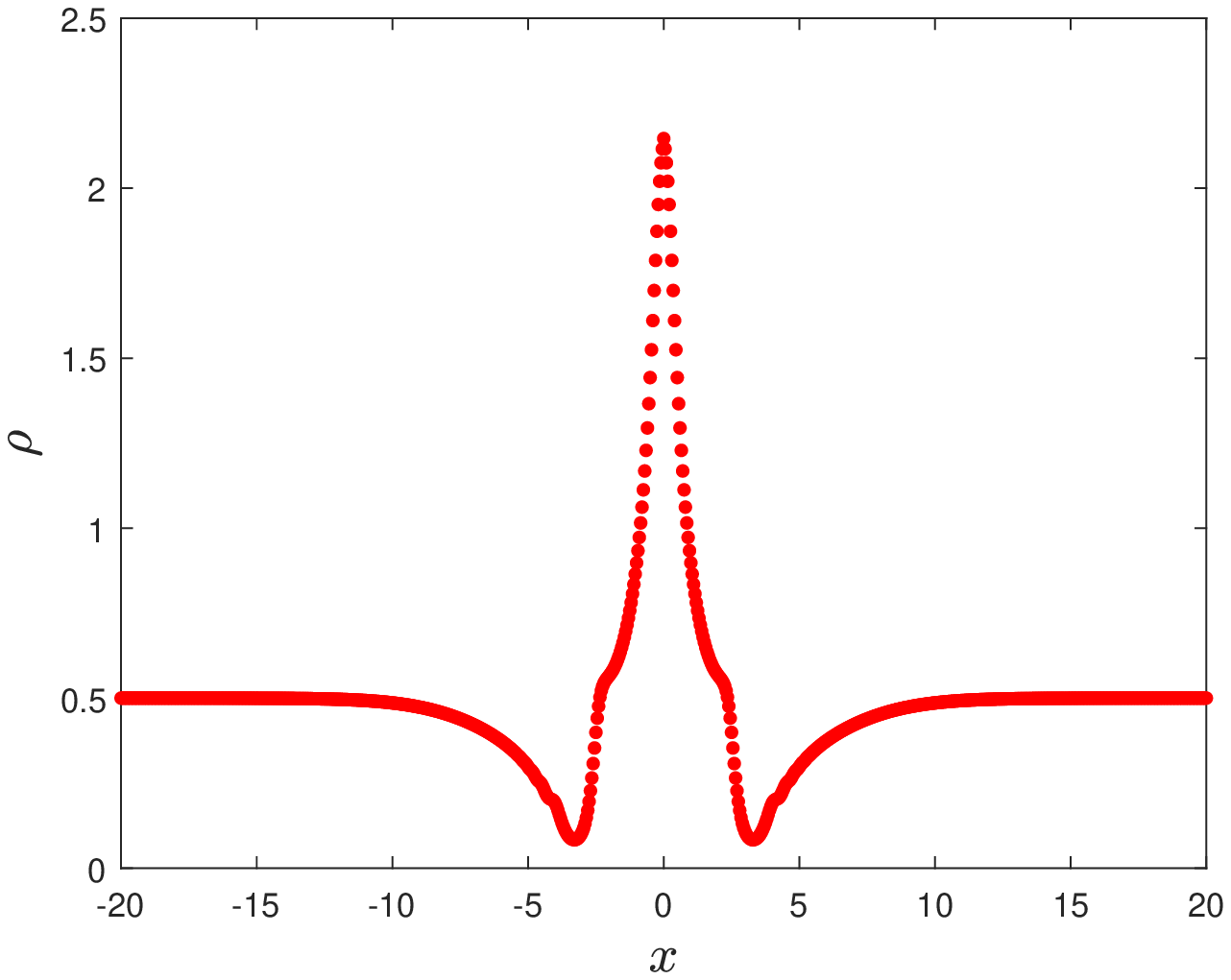}
	} \vspace{-3mm}
	\subfigure[$t=6$]{\centering
		\includegraphics[width=0.35\textwidth]{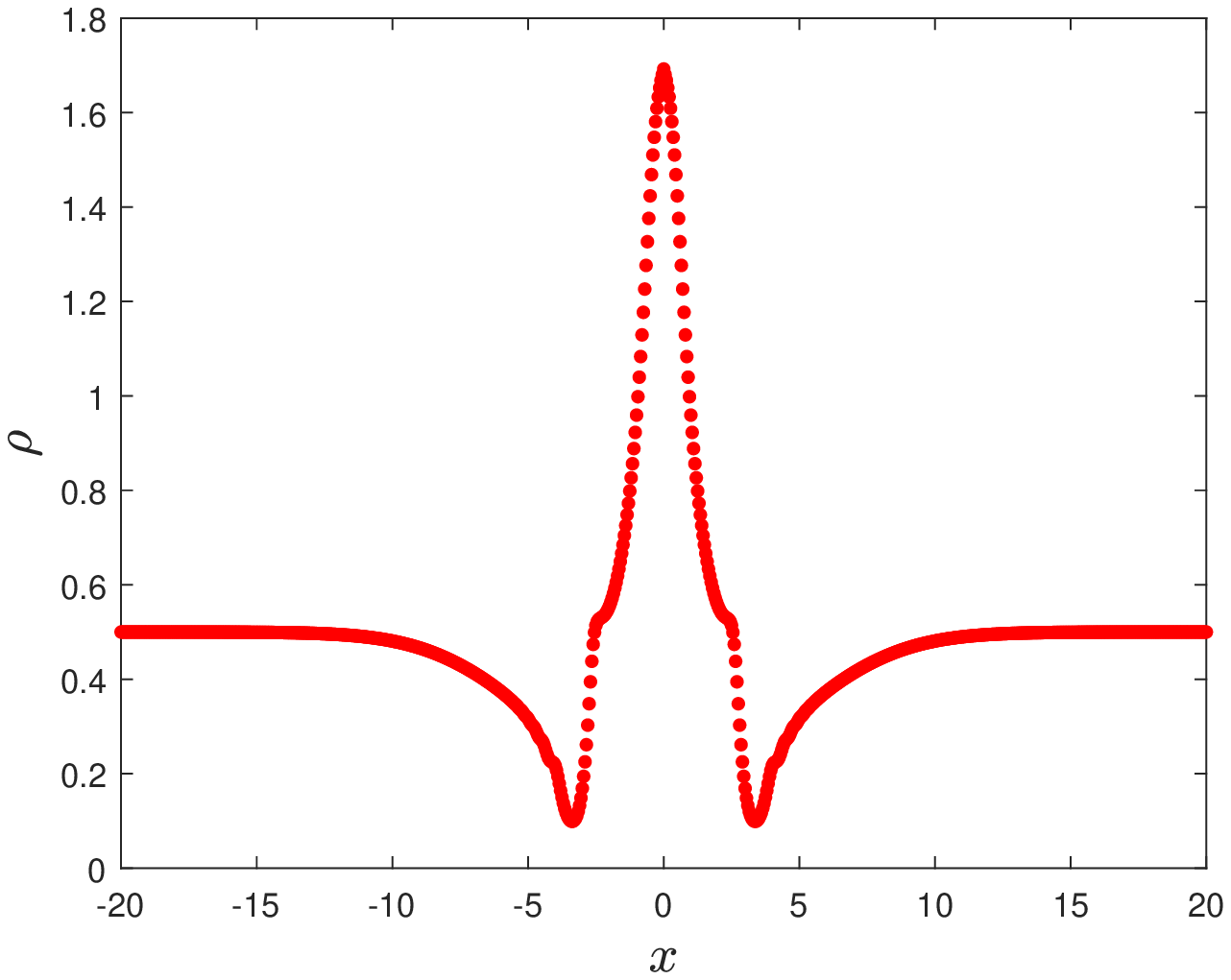}
	}\hspace{-5.3mm}\subfigure[$t=8$]{\centering
		\includegraphics[width=0.35\textwidth]{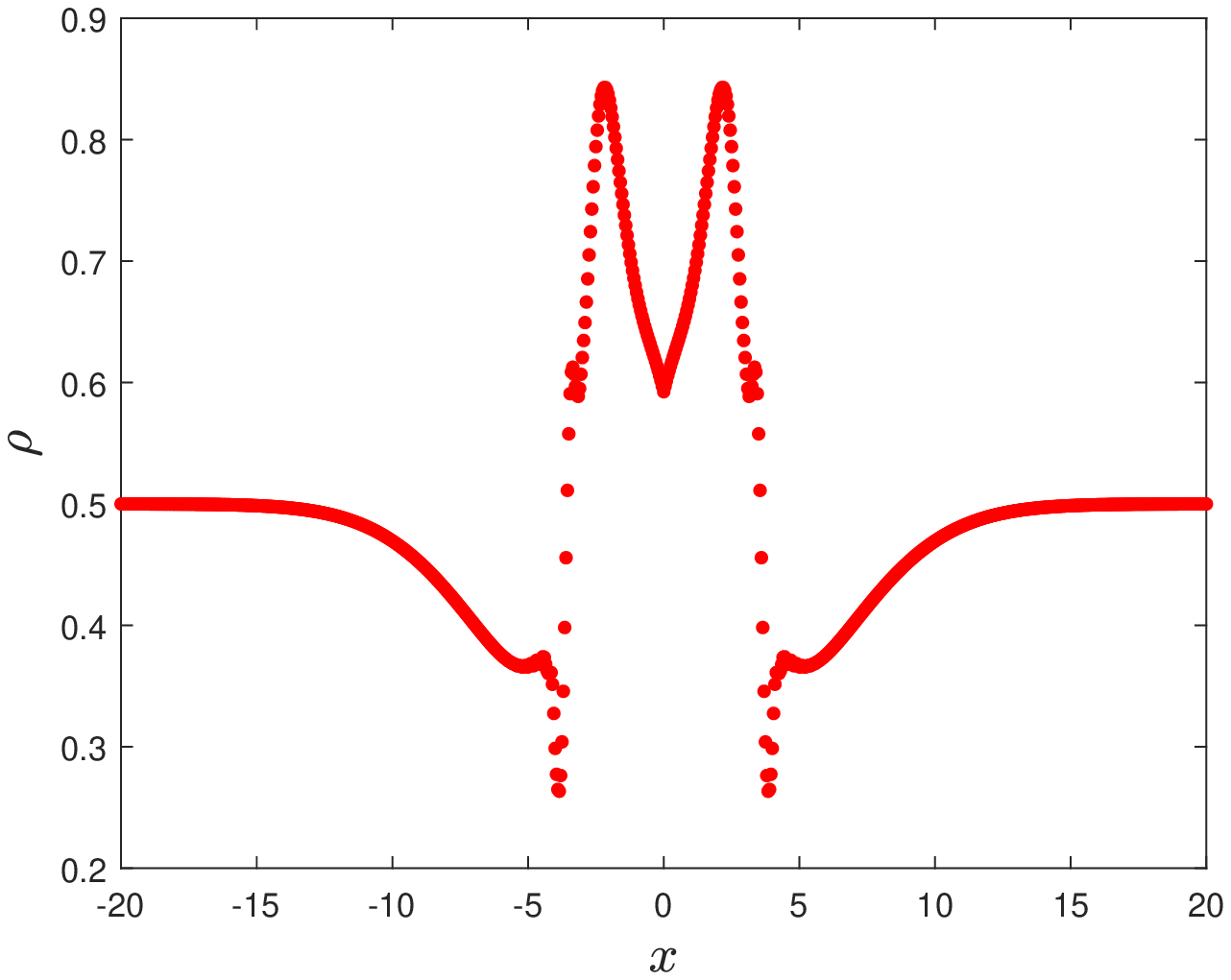}
	}\hspace{-5.3mm}\subfigure[$t=10$]{\centering
		\includegraphics[width=0.35\textwidth]{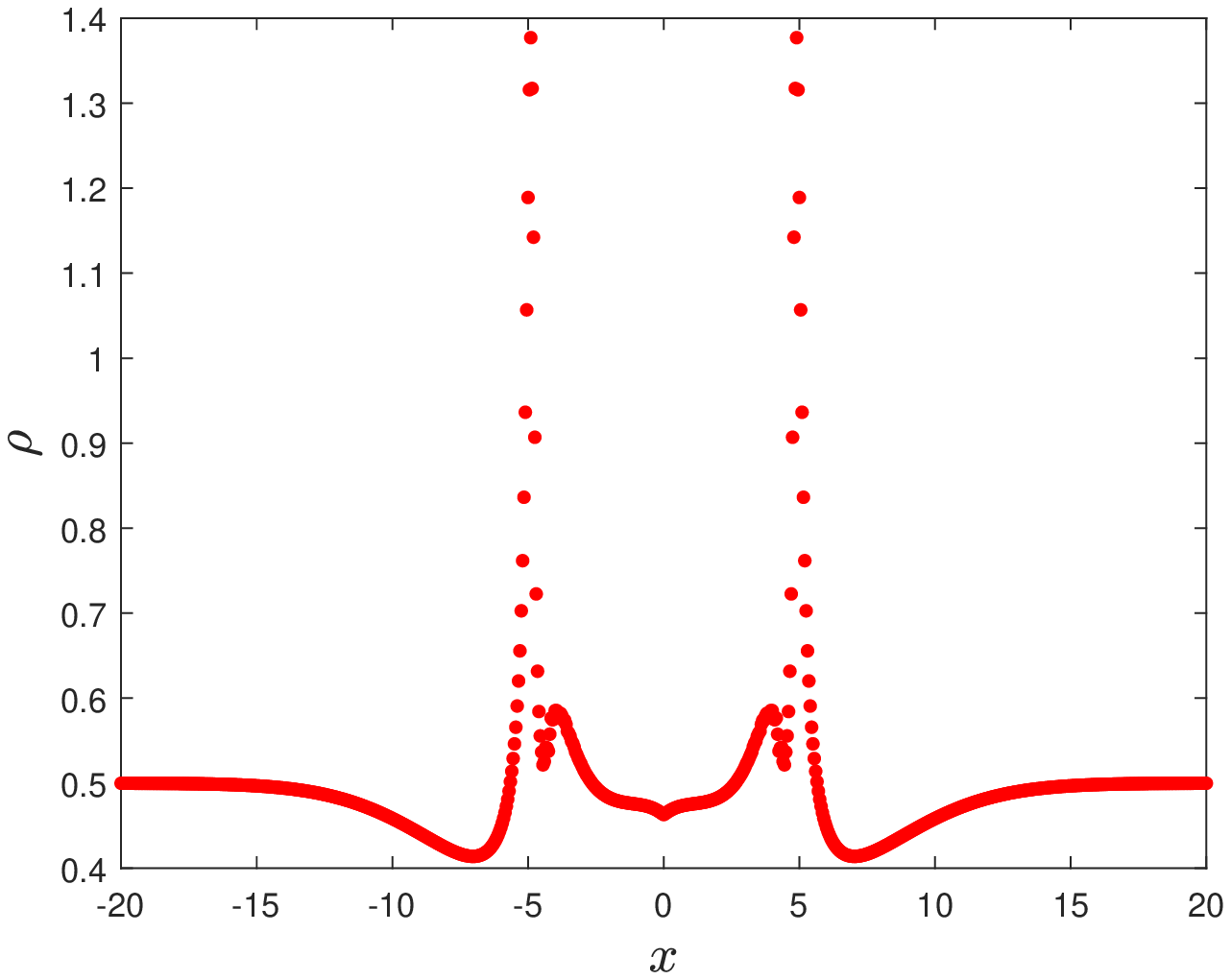}
	} \vspace{-2mm}
	\caption{The heights $\rho(x,t)$ for the R2CH system in \textbf{Case} (\uppercase\expandafter{\romannumeral1}) at different times with stepsizes $h=0.05$ and $\tau = 0.0005$.} \label{fig17}
\end{figure}
\begin{figure}[htbp]
	\subfigure[$t=1$]{\centering
		\includegraphics[width=0.35\textwidth]{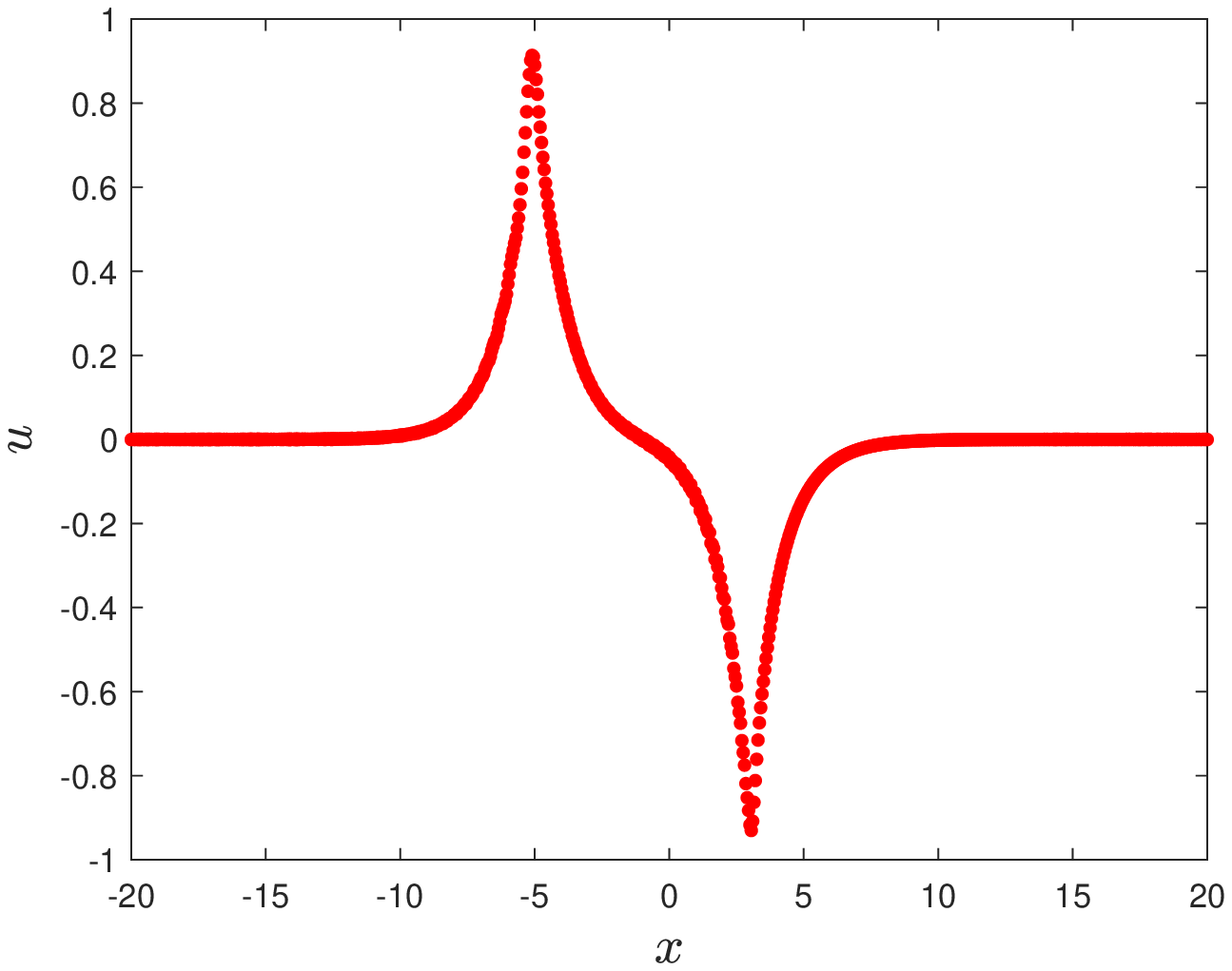}
	}\hspace{-5.3mm}\subfigure[$t=3$]{\centering
		\includegraphics[width=0.35\textwidth]{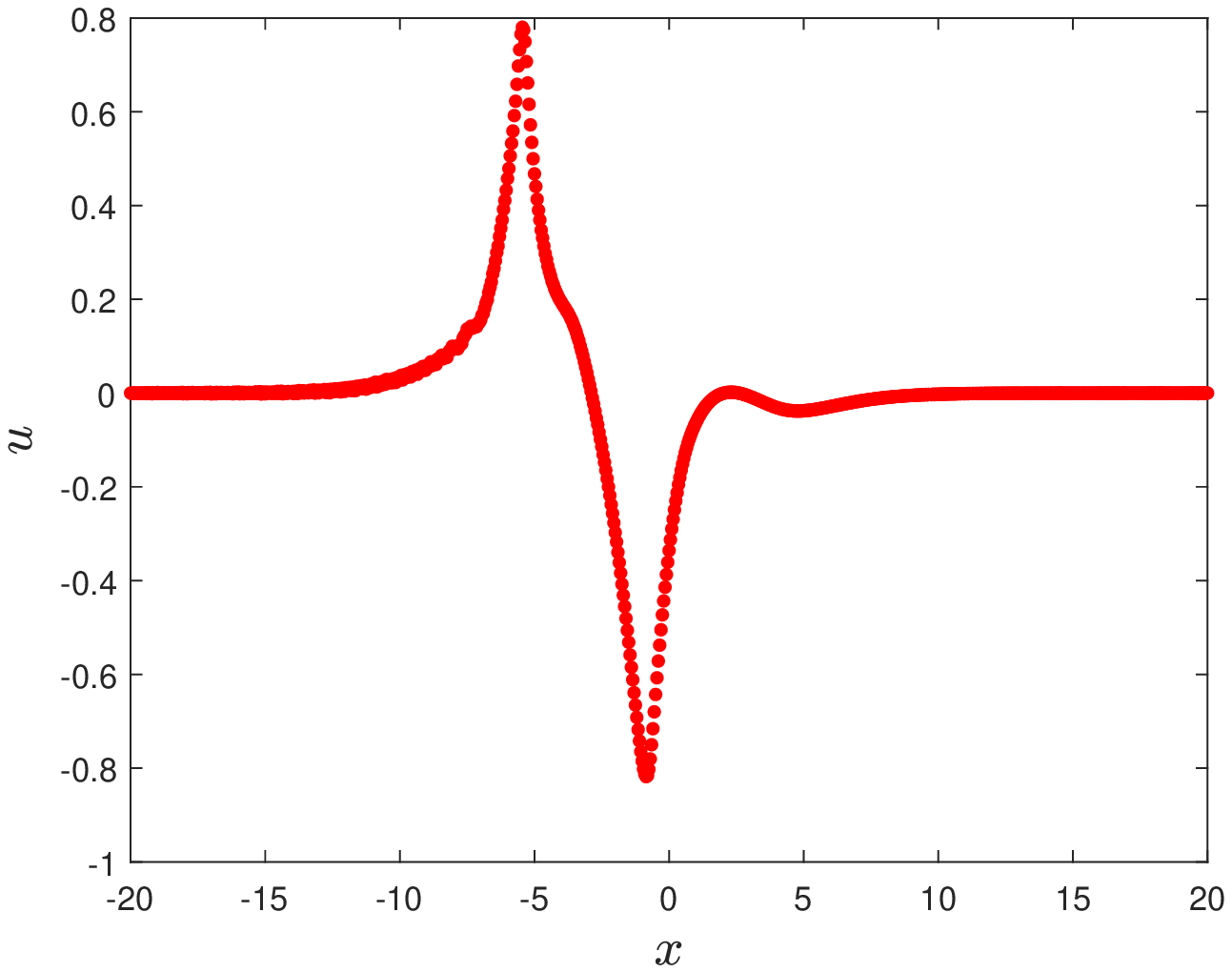}
	}\hspace{-5.3mm}\subfigure[$t=5$]{\centering
		\includegraphics[width=0.35\textwidth]{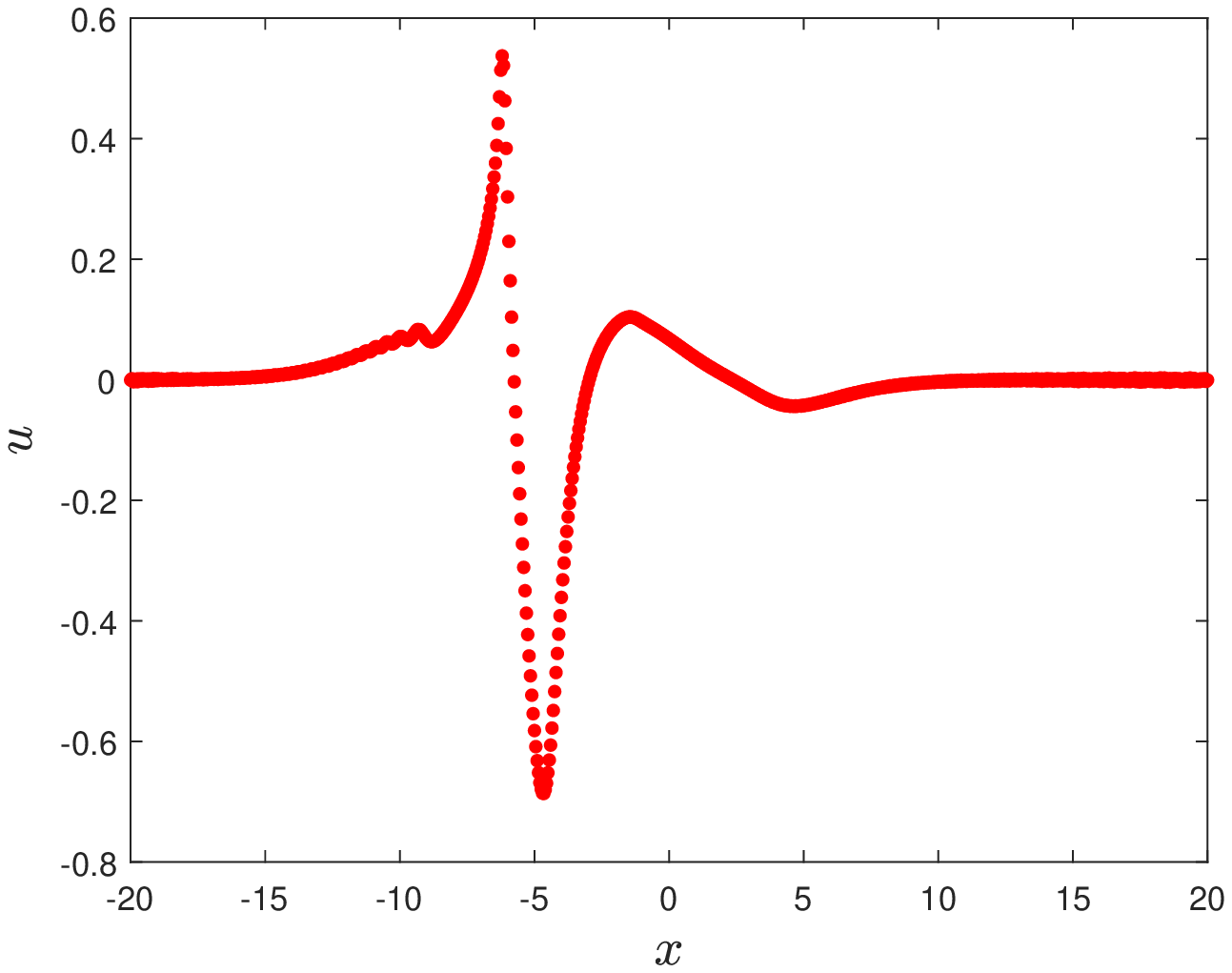}
	}  \vspace{-3mm}
	\subfigure[$t=6$]{\centering
		\includegraphics[width=0.35\textwidth]{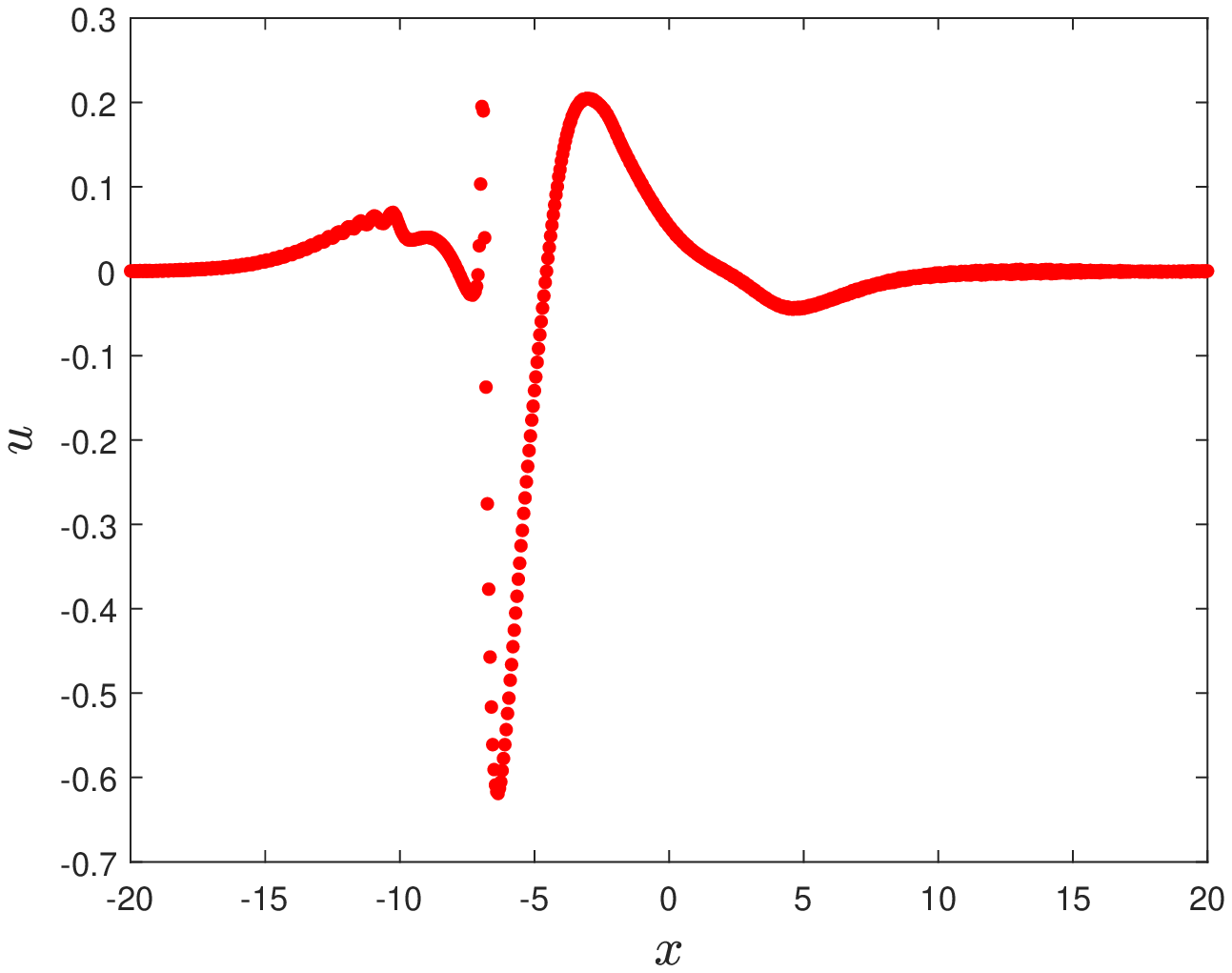}
	}\hspace{-5.3mm}\subfigure[$t=8$]{\centering
		\includegraphics[width=0.35\textwidth]{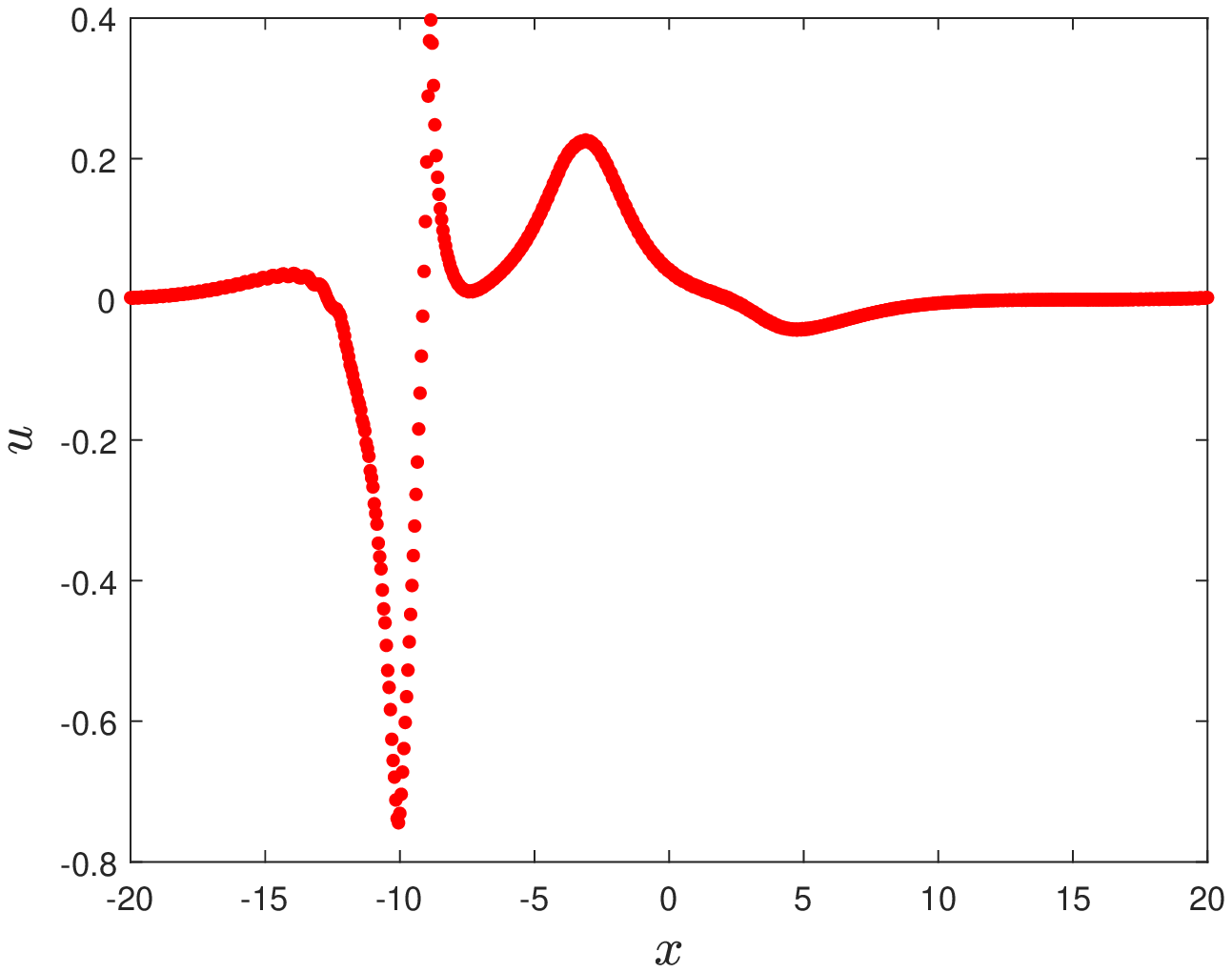}
	}\hspace{-5.3mm}\subfigure[$t=10$]{\centering
		\includegraphics[width=0.35\textwidth]{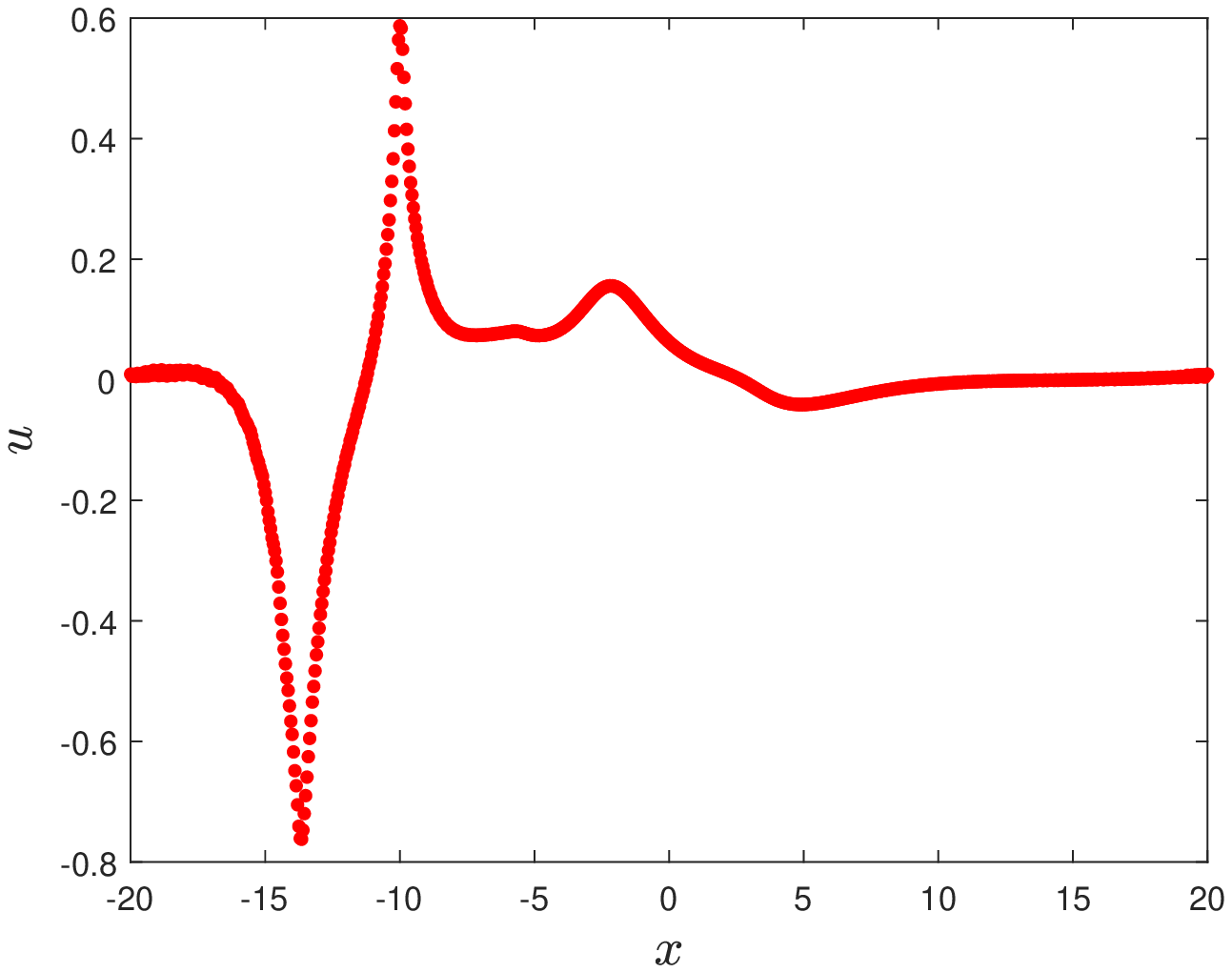}
	} \vspace{-2mm}
	\caption{The velocities $u(x,t)$ for the R2CH system in \textbf{Case} (\uppercase\expandafter{\romannumeral3}) at different times with stepsizes $h=0.05$ and $\tau = 0.0005$.} \label{fig18}
\vspace{-2mm}
\end{figure}

\begin{figure}[htbp]
	\subfigure[$t=1$]{\centering
		\includegraphics[width=0.35\textwidth]{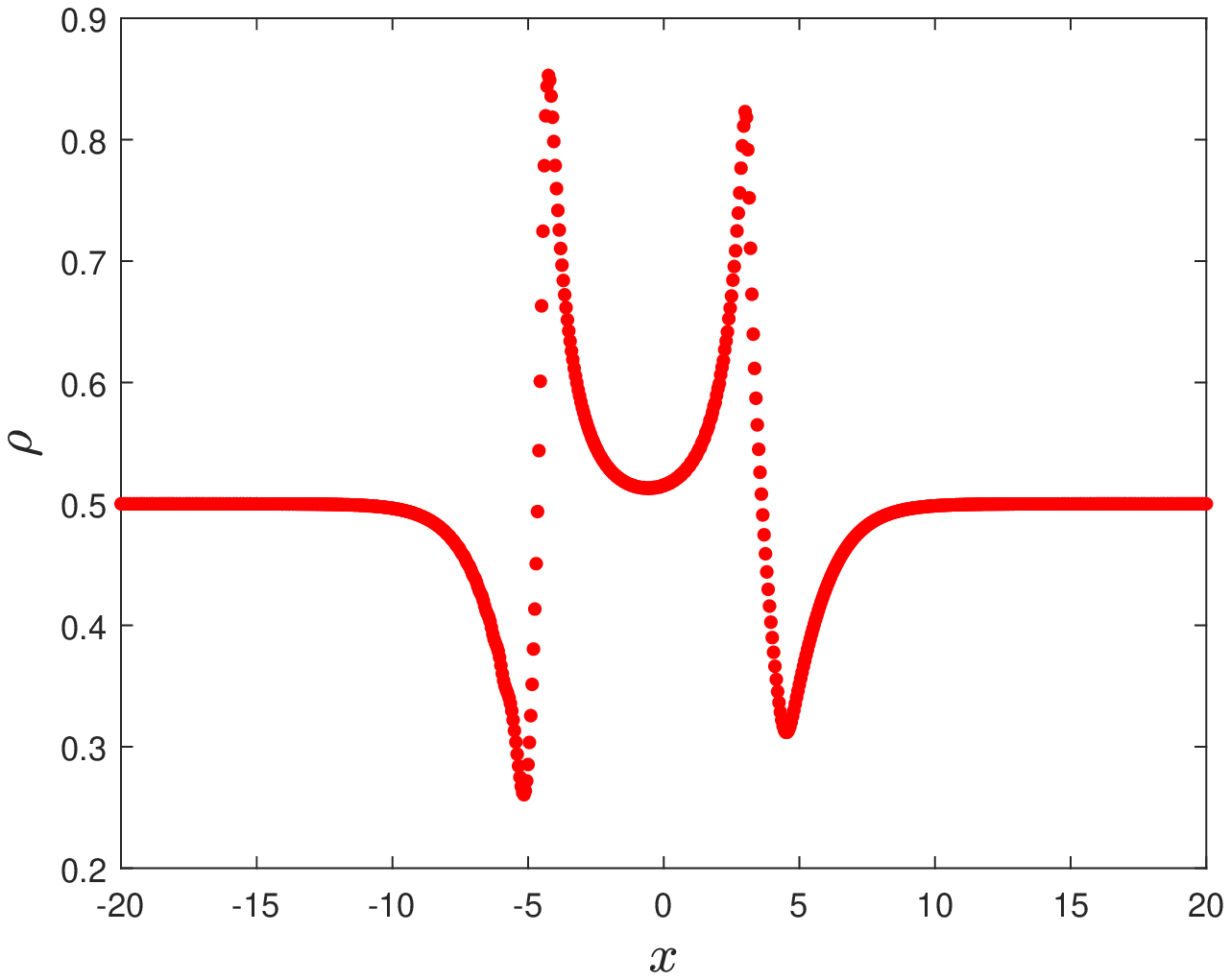}
	}\hspace{-5.3mm}\subfigure[$t=3$]{\centering
		\includegraphics[width=0.35\textwidth]{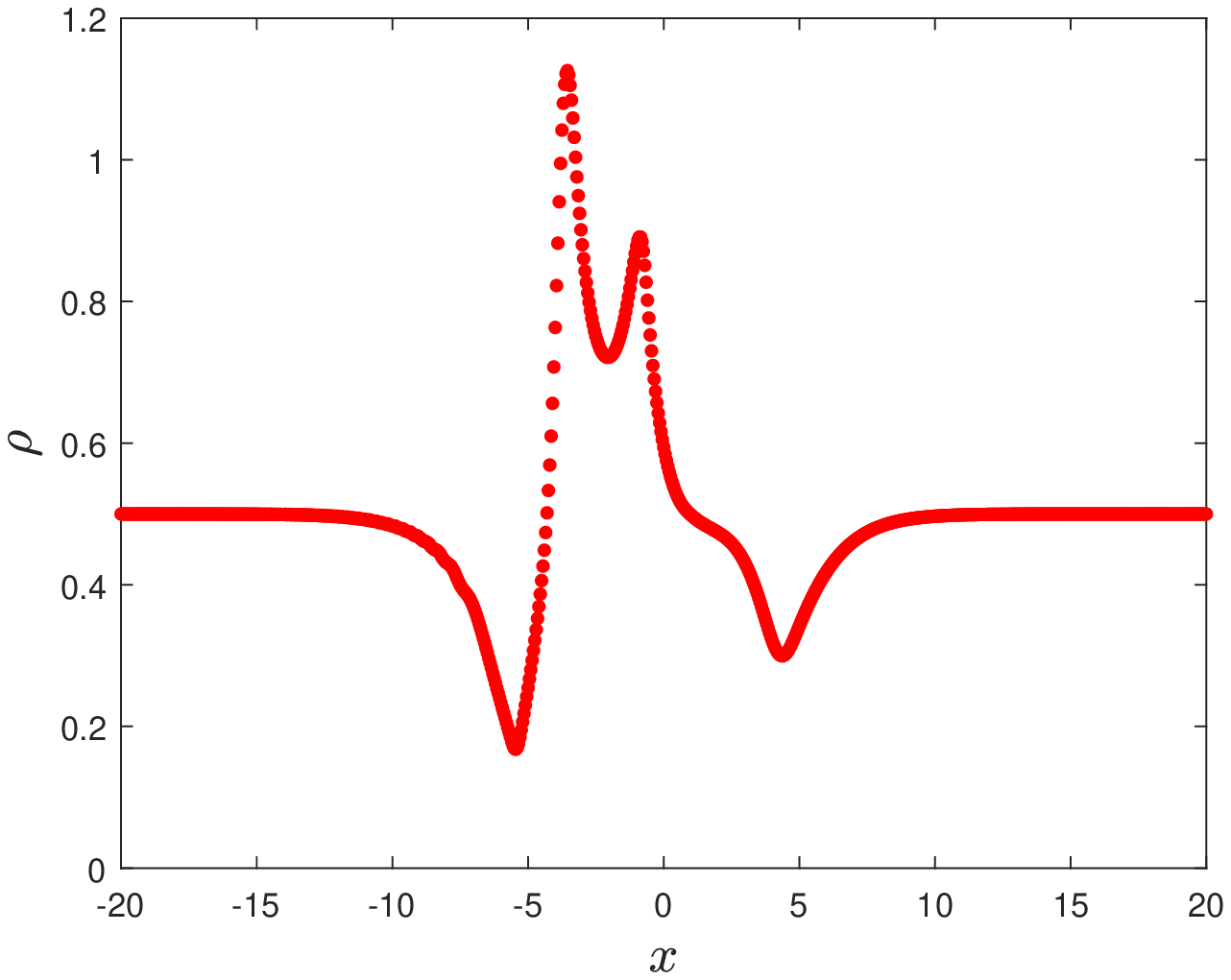}
	}\hspace{-5.3mm}\subfigure[$t=5$]{\centering
		\includegraphics[width=0.35\textwidth]{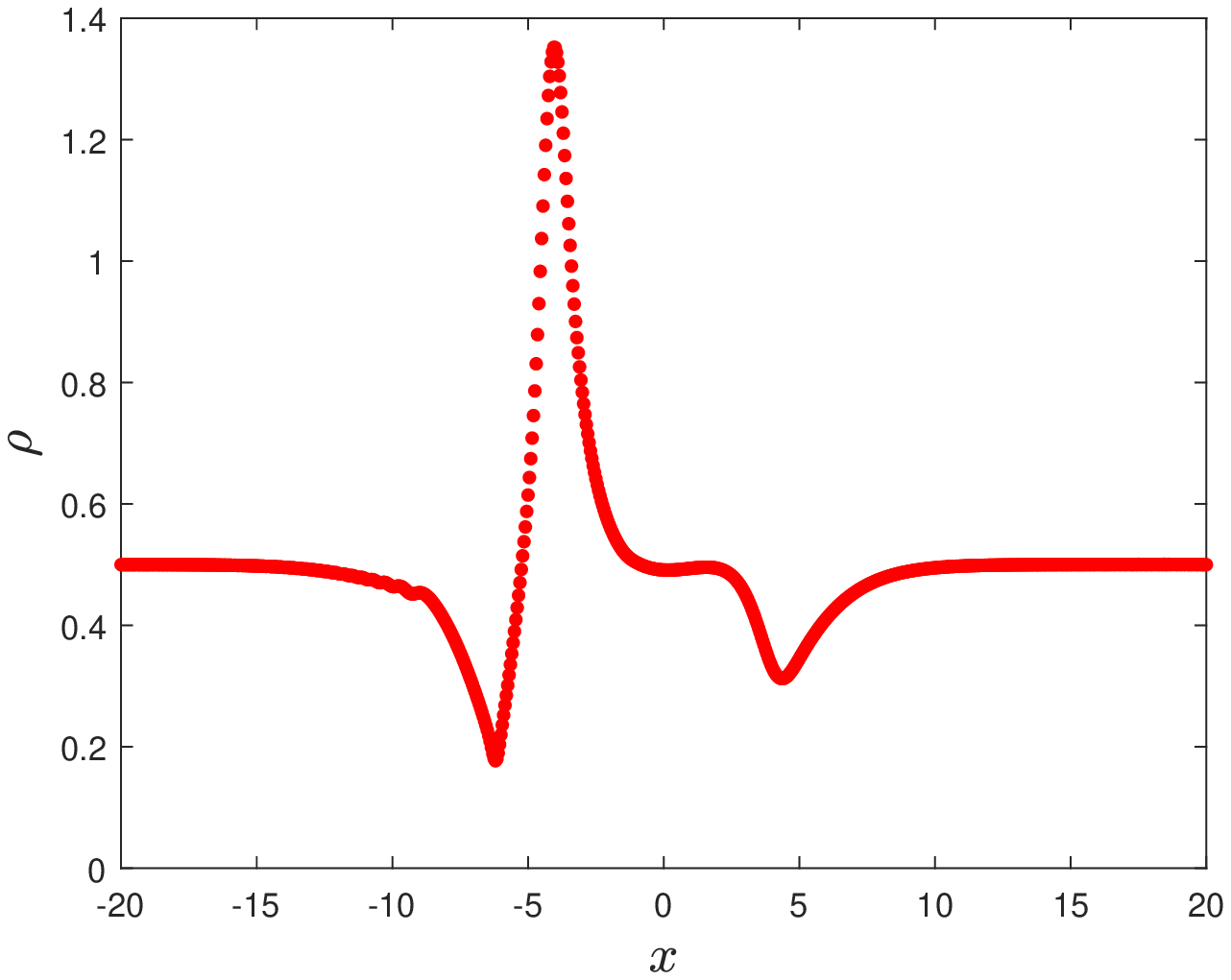}
	} \vspace{-2mm}
	\subfigure[$t=6$]{\centering
		\includegraphics[width=0.35\textwidth]{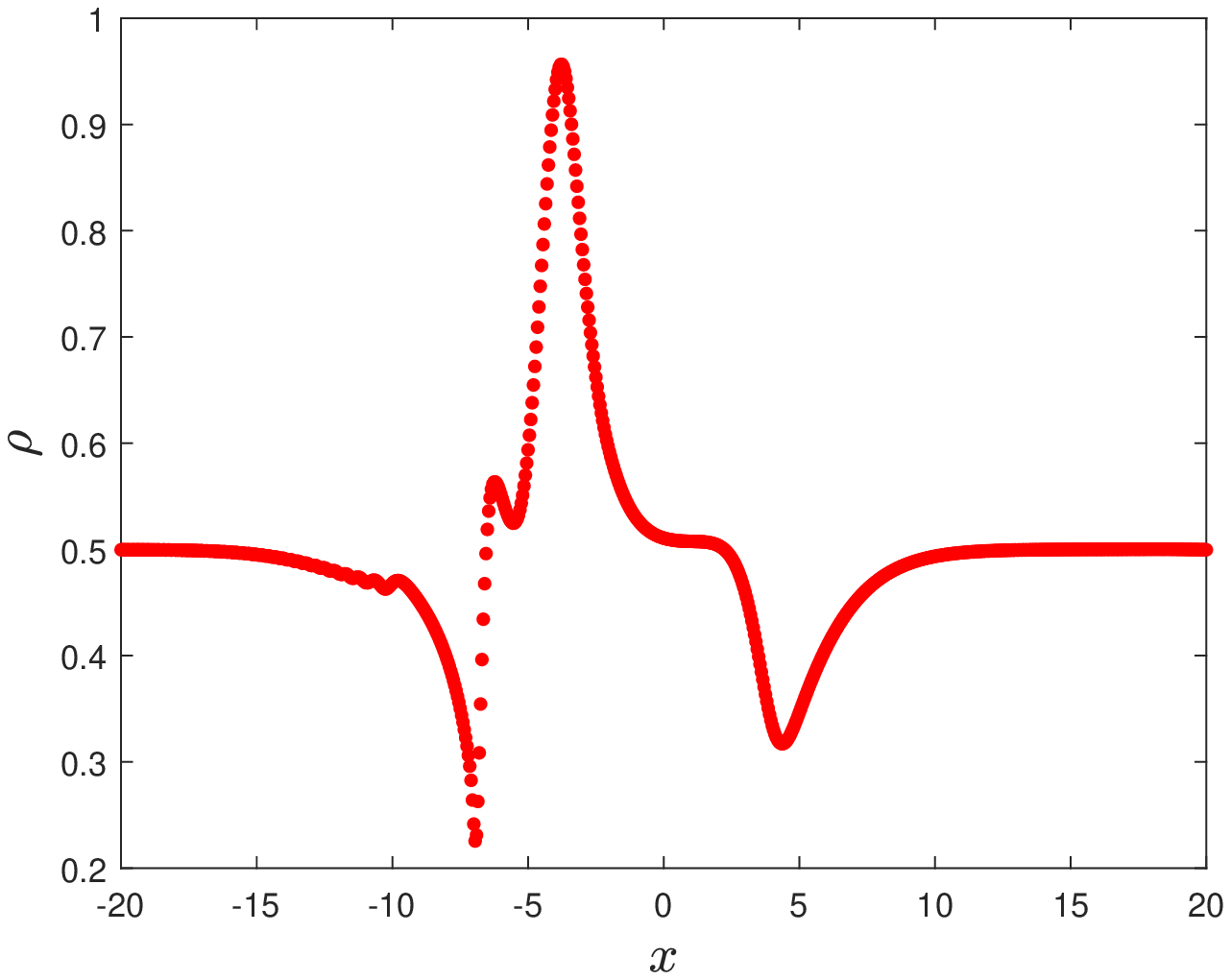}
	}\hspace{-5.3mm}\subfigure[$t=8$]{\centering
		\includegraphics[width=0.35\textwidth]{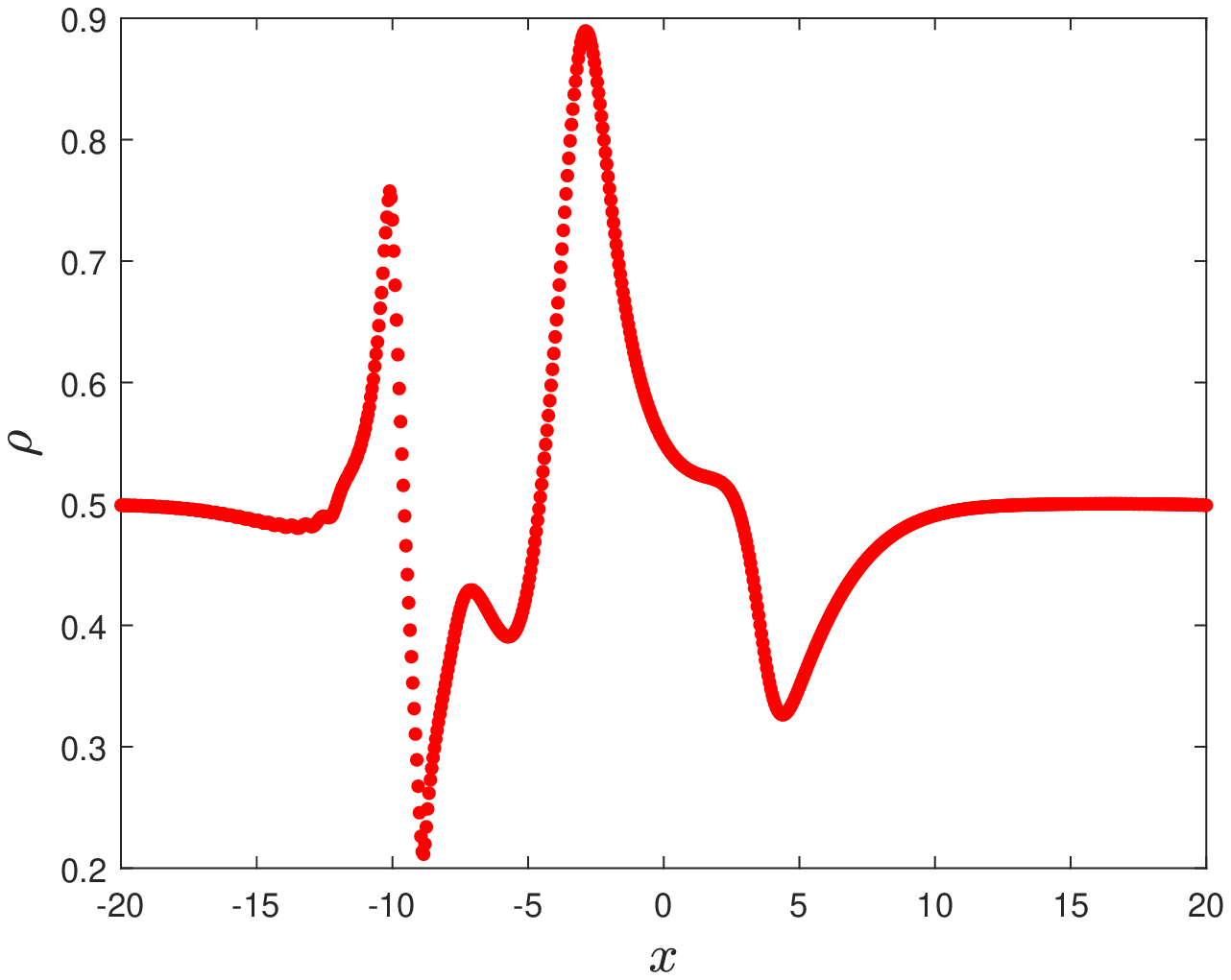}
	}\hspace{-5.3mm}\subfigure[$t=10$]{\centering
		\includegraphics[width=0.35\textwidth]{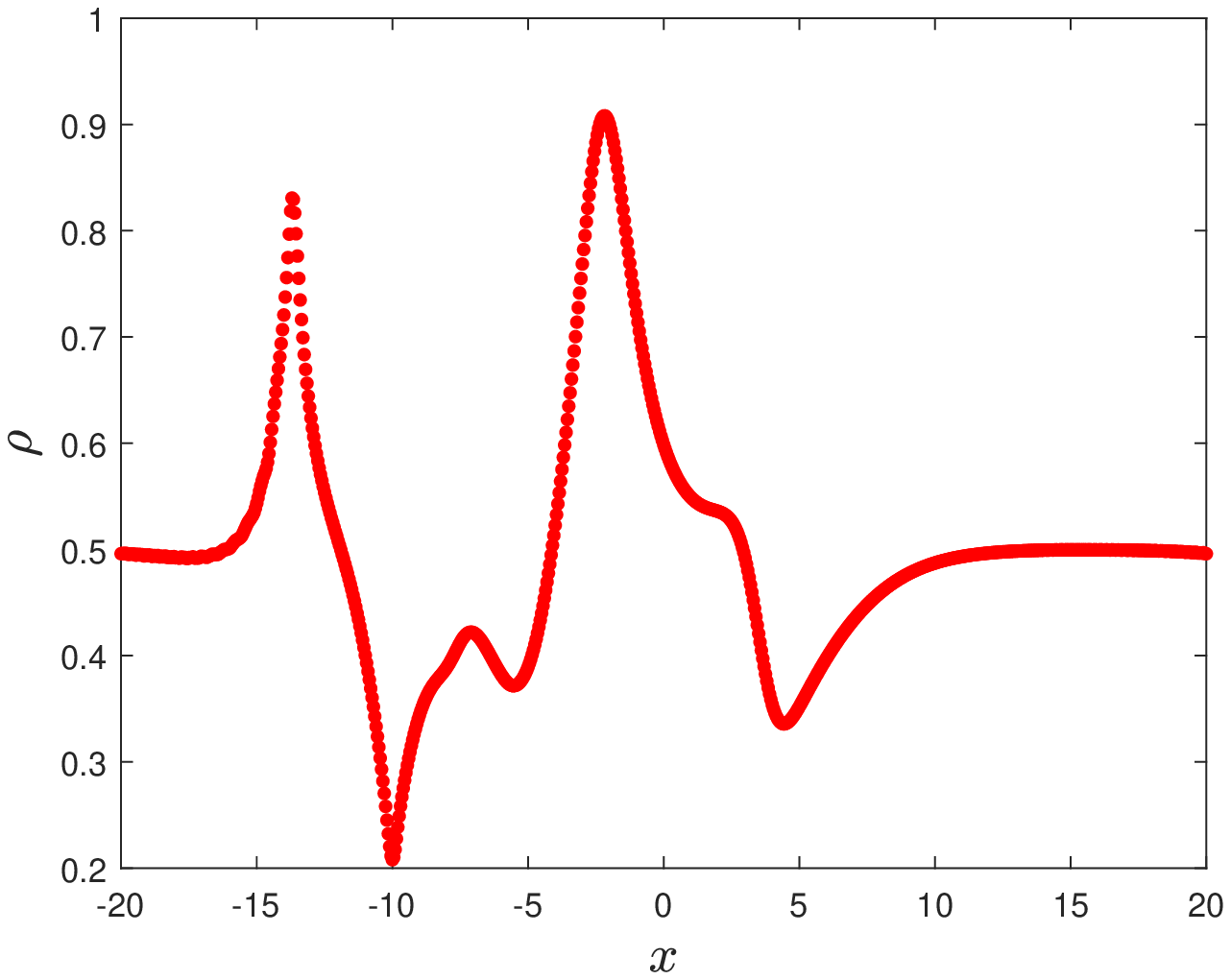}
	} \vspace{-2mm}
	\caption{The heights $\rho(x,t)$ for the R2CH system in \textbf{Case} (\uppercase\expandafter{\romannumeral3}) at different times with stepsizes $h=0.05$ and $\tau = 0.0005$.}  \label{fig19}
\end{figure}

\begin{figure}[htbp]
	\centering
	\subfigure[Velocity variable, view(0,90)]{\centering
		\includegraphics[width=0.4\textwidth]{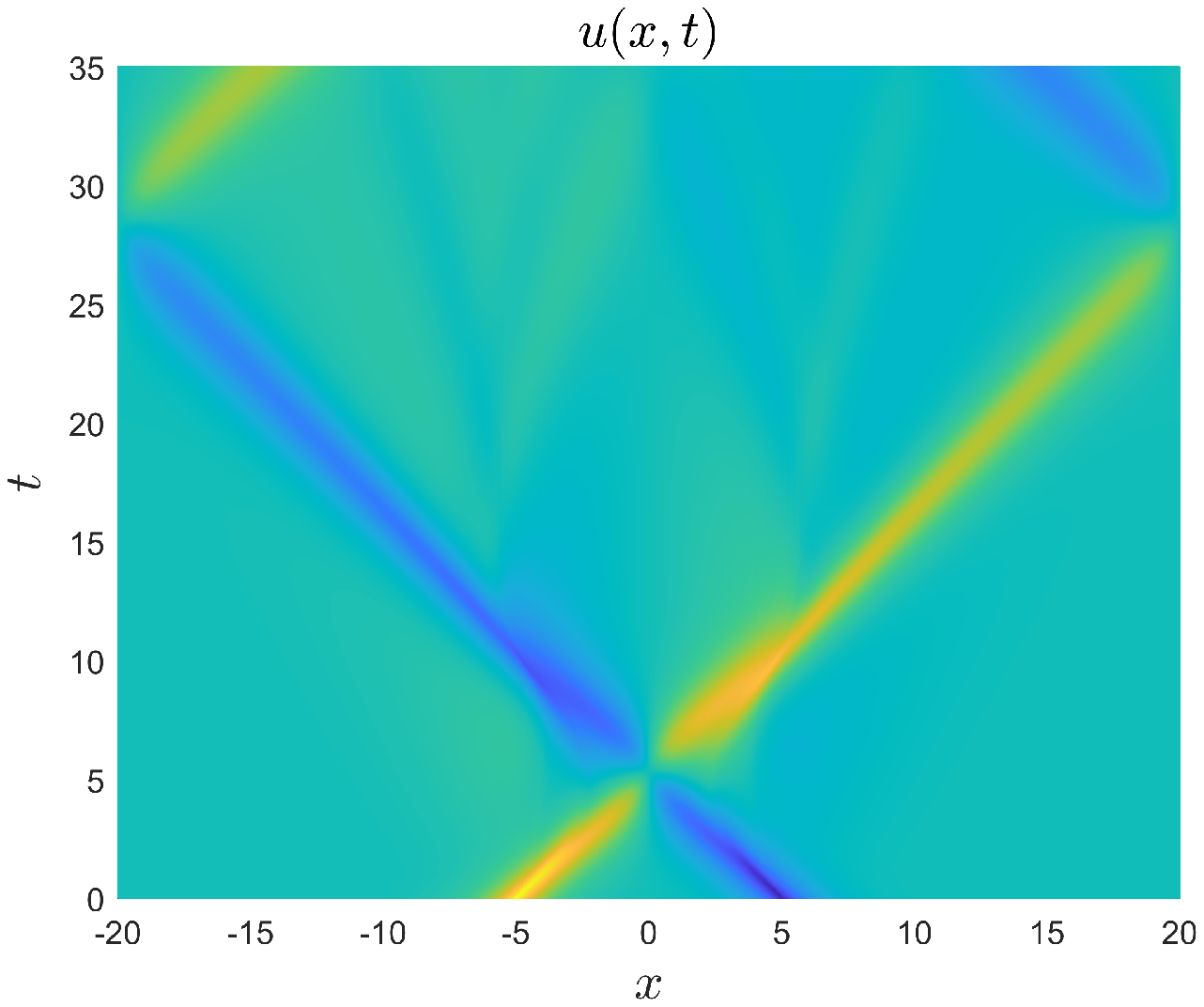}
	}
	\subfigure[Height variable, view(0,90)]{\centering
		\includegraphics[width=0.4\textwidth]{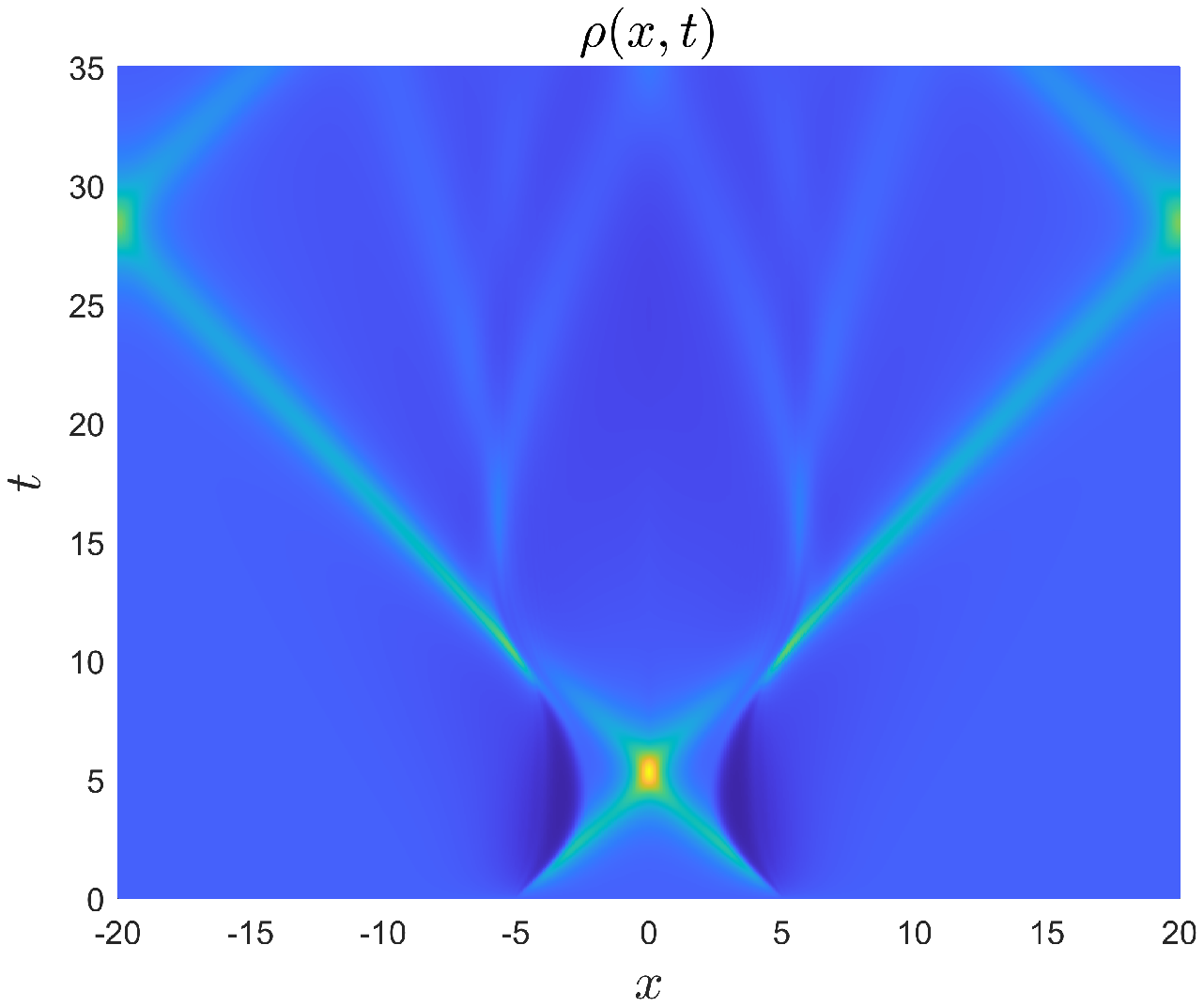}
	}
	\caption{The predicted peakon solutions for the R2CH system in \textbf{Case} (\uppercase\expandafter{\romannumeral1}) show the evolution of the velocity $u(x,t)$ and the height $\rho(x,t)$ with $t = 35$. } \label{fig20}
\end{figure}

\begin{figure}[htbp]
	\centering
	\subfigure[Velocity variable, view(0,90)]{\centering
		\includegraphics[width=0.4\textwidth]{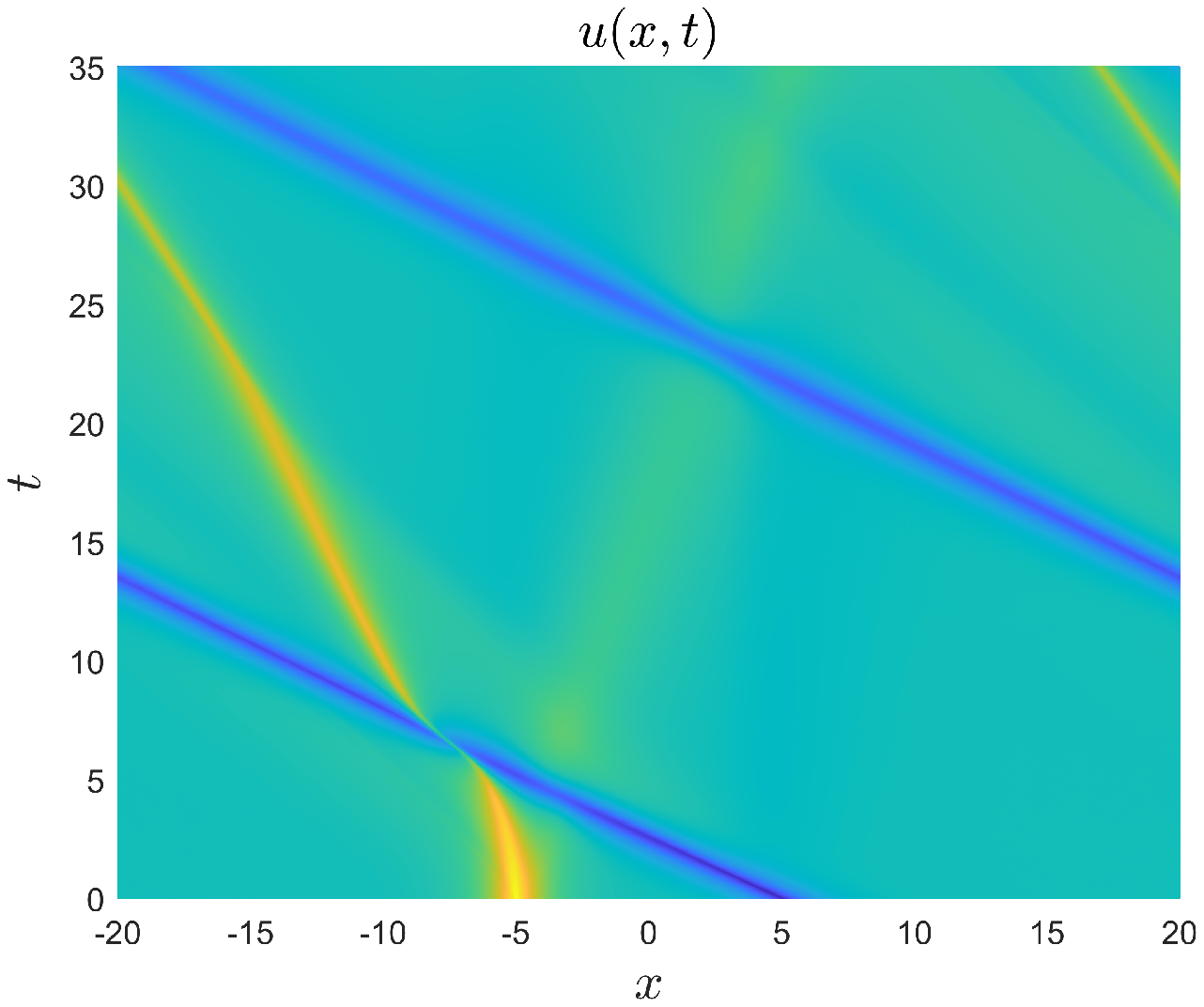}
	}
	\subfigure[Height variable, view(0,90)]{\centering
		\includegraphics[width=0.4\textwidth]{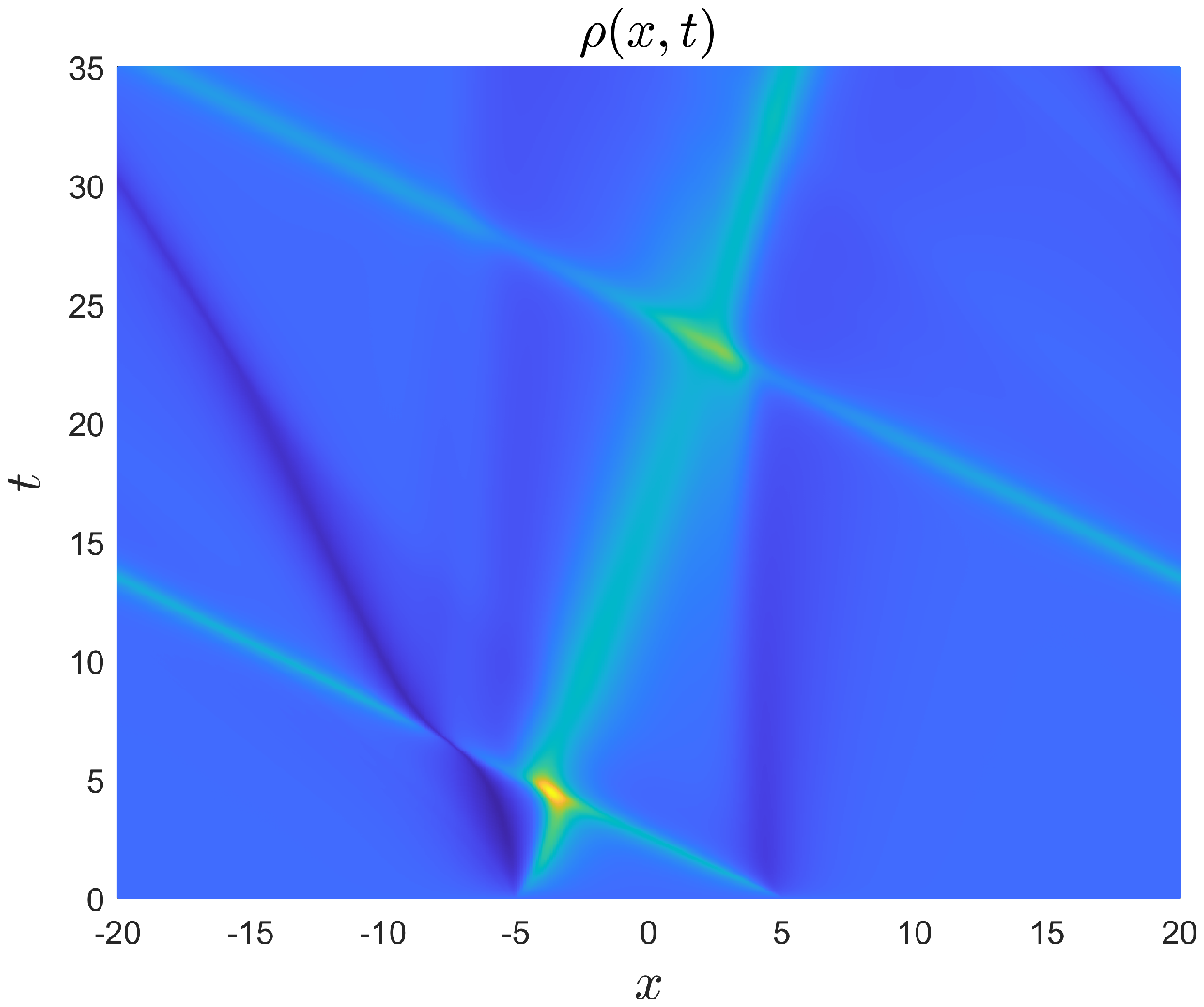}
	}
	\caption{The predicted peakon solutions for the R2CH system in \textbf{Case} (\uppercase\expandafter{\romannumeral3}) show the
		evolution of the velocity $u(x,t)$ and the height $\rho(x,t)$ with $t = 35$. } \label{fig21}
\end{figure}
\subsection{Part III: the conserved quantity $H$}\label{Sec5.3}
Recalling the conserved quantity $H$ mentioned in the introduction, it can be defined discretely as
\begin{align}
	H^n = \sum_{i=1}^{M}\bigg[(u_i^n)^3 + u_i^n \Big(\frac{u_{i+1}^n - u_{i-1}^n}{2h}\Big)^2 - A(u_i^n)^2 -\mu \Big(\frac{u_{i+1}^n - u_{i-1}^n}{2h}\Big)^2 + u_i^n (\rho_i^n)^2 \bigg],\quad n=1,2,\cdots,N. \label{eq5.3}
\end{align}
We expect that the proposed scheme \eqref{eq3.1}--\eqref{eq3.3}  also can preserve the conserved quantity. Toward this end we conduct a conservation test for $H$ through two different types of numerical examples including a smooth initial data problem and two nonsmooth initial data problem. For the calculation of Example \ref{Exam5.1}, the parameters are selected from \textbf{Case} (\uppercase\expandafter{\romannumeral1}) and \textbf{Case} (\uppercase\expandafter{\romannumeral2}). For the nonsmooth initial data problem, Example \ref{Exam5.4} and Example \ref{Exam5.5} are selected. All the parameters in \textbf{Case} (\uppercase\expandafter{\romannumeral1}) are used. Due to the symmetry of initial values, it immediately obtains that $H^0=0$ according to \eqref{eq5.3}. Figure \ref{fig22} depicts the errors between $H^n$ and $H^0$ at different times, and we can find that $H^n$ can approximate to $H$ numerically.
\begin{figure}[htbp]
	\centering
	\subfigure[$H^n-H^0$ in {\rm Example} \ref{Exam5.1}]{\centering
		\includegraphics[width=0.35\textwidth]{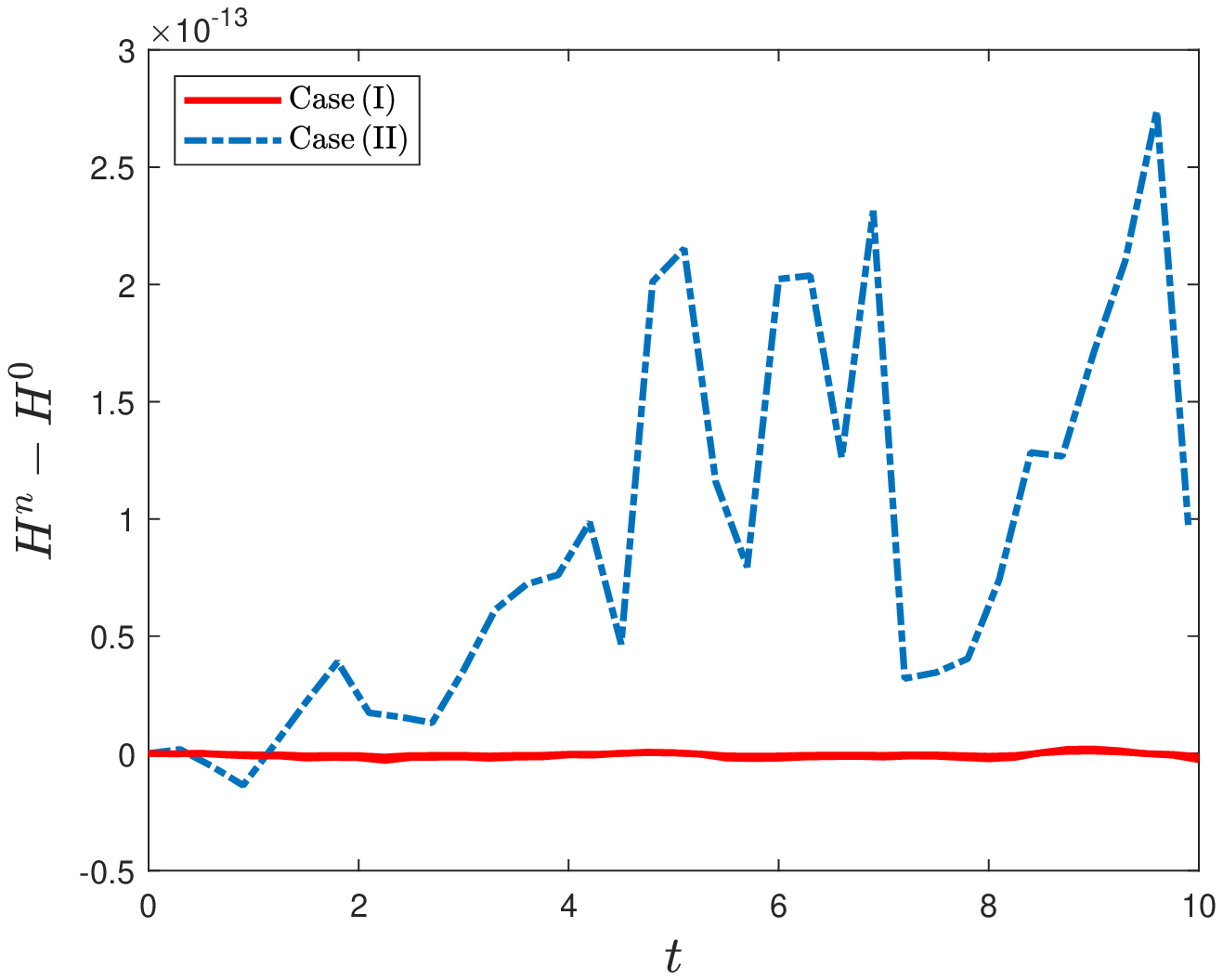}
	}\hspace{-5.3mm}\subfigure[$H^n-H^0$ in {\rm Example} \ref{Exam5.4}]{\centering
		\includegraphics[width=0.35\textwidth]{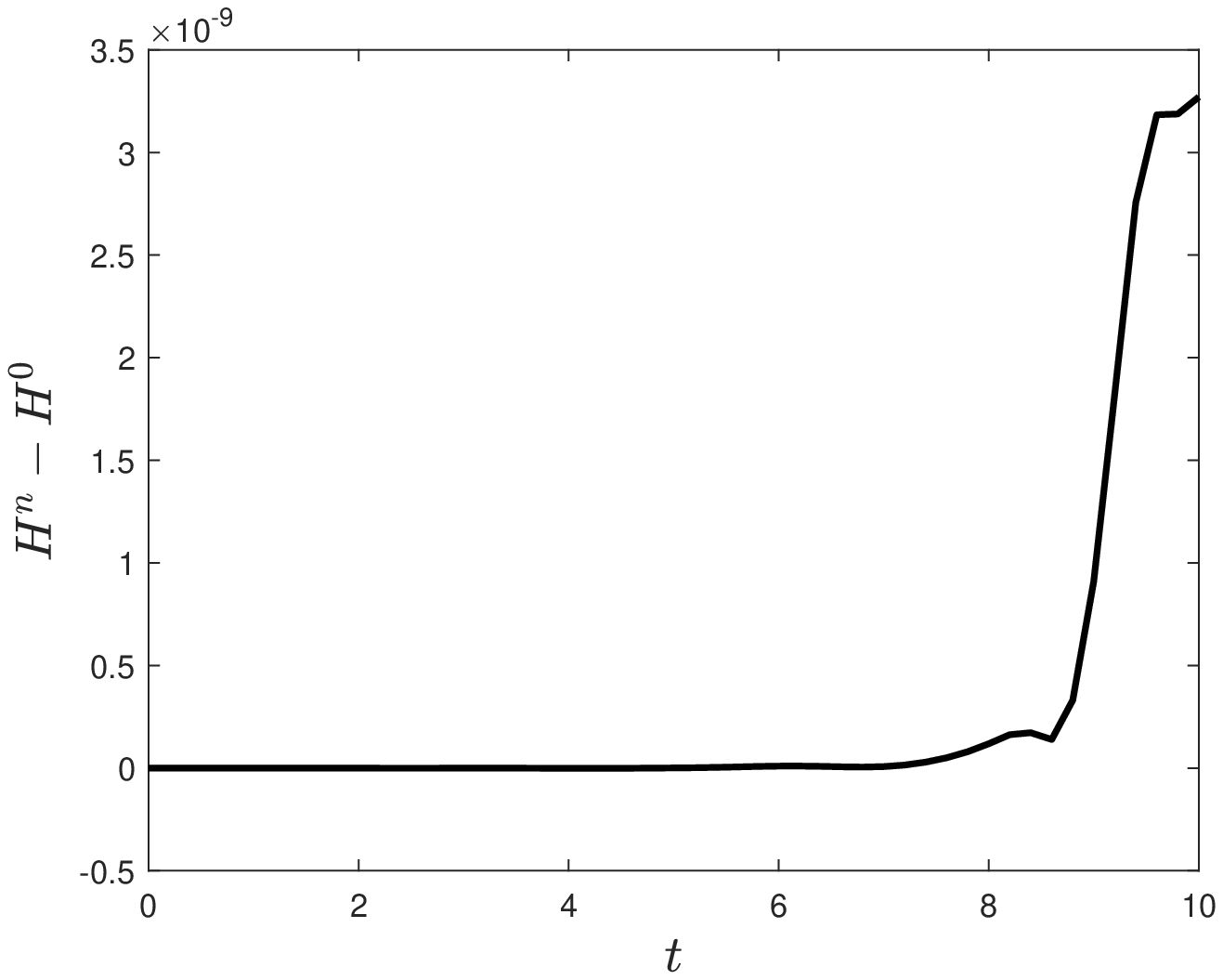}
	}\hspace{-5.3mm}\subfigure[$H^n-H^0$ in {\rm Example} \ref{Exam5.5}]{\centering
		\includegraphics[width=0.35\textwidth]{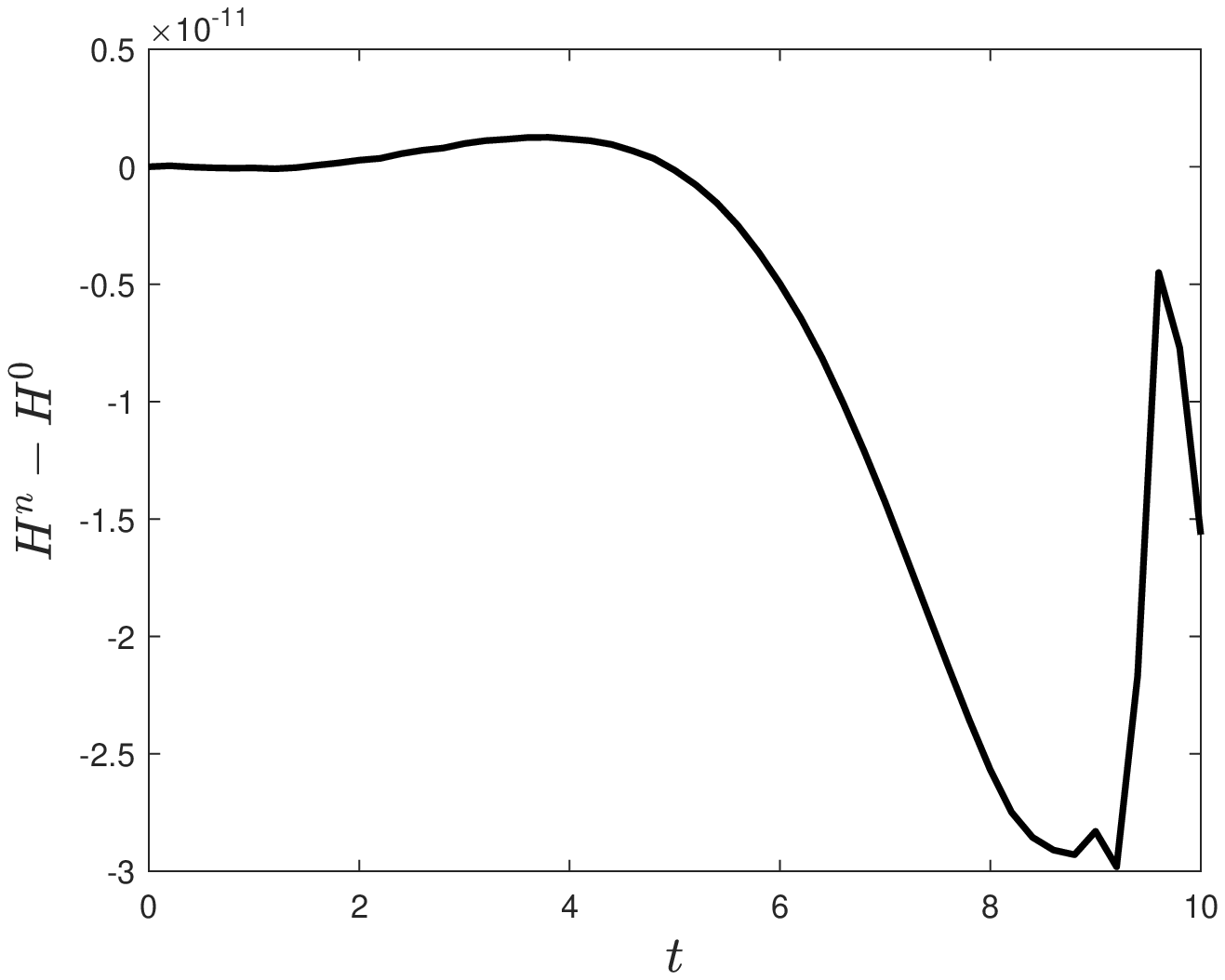}}
	\caption{(a) shows the computed errors of the conserved quantity $H$ with $h=0.5$ and $\tau=0.005$ in \textbf{Case} (\uppercase\expandafter{\romannumeral1}) and \textbf{Case} (\uppercase\expandafter{\romannumeral2})  in subsection \ref{sec5_1} for {\rm Example} \ref{Exam5.1}; (b) and (c) show the computed errors of the conserved quantity $H$ with $h=0.2$ and $\tau=0.0025$ in \textbf{Case} (\uppercase\expandafter{\romannumeral1}) for {\rm Example} \ref{Exam5.4} and {\rm Example} \ref{Exam5.5} in subsection \ref{sec5_2}, respectively. }  \label{fig22}
\end{figure}

\section{Concluding remarks}\label{sec6}
In this paper, we have carried out extensively numerical study on the R2CH system under both the smooth and nonsmooth initial data based on a well-designed conservative finite difference discretization. The evolution of several asymmetric and non-smooth solitary wave solutions are depicted for the first time. The numerical conservation laws enable the calculation schemes to accurately capture the evolution of the smooth/nonsmooth solutions of the benchmark problems. In particular, the numerical threshold technique plays a key role in grasping drastic change of peakon solutions.
Last but surely not the least, there is still plenty of rooms for further study. Especially, the strict error estimate of the difference scheme is a challenging topic, which is not covered in this paper.

\section*{Appendix}
\setcounter{equation}{0}
\renewcommand\theequation{A.\arabic{equation}} 
Noticing when $\sigma=1$ and $\Omega=0$, the system \eqref{eq1.1}--\eqref{eq1.2} can be written as a system of hyperbolic type
\bcase
u_t + uu_x + \partial_x G * \Big(u^2 + \frac{1}{2}u_x^2 - Au + \mu u_{xx} + \frac{1}{2}\rho^2\Big) = 0, \label{eqA.1a}  \\
\rho_t + u\rho_x = - u_x\rho, \quad x \in R, t > 0,   \label{eqA.1b}
\ecase
then we show the system above has the conservation law
\begin{align*}
	H = \int_{R} (u^3 + uu_{x}^2 - Au^2 - \mu u_{x}^2 + u \rho^2 ) {\rm d} x.   
\end{align*}
\begin{proof}
	For simplicity, we set
	\begin{align}
		g(x) = G*\Big(u^2 + \frac{1}{2}u_x^2 - Au + \mu u_{xx}\Big), \quad h(x) = G*\Big(\frac{1}{2}\rho^2\Big).  \label{eqA.3}
	\end{align}
	Then \eqref{eqA.1a} takes the equivalent form
	\begin{align}
		u_t + uu_x + g_x + h_x = 0.  \label{eqA.4}
	\end{align}
	Using \eqref{eqA.4} and noticing
	$$\int_{R}u\rho\rho_t\mathrm{d}x = -\int_{R}(u\rho)(u\rho)_x\mathrm{d}x = 0,$$
	we have
	\begin{align}
		&\frac{{\rm d}}{{\rm d}t} \int_{R} (u^3 + uu_{x}^2 - Au^2 - \mu u_{x}^2 + u \rho^2 ) \mathrm{d}x  \notag \\
		&\quad = \int_{R}(3u^2u_t + u_tu_{x}^2 + 2uu_xu_{xt} -2Auu_t -2\mu u_xu_{xt} + u_t \rho^2 ) \mathrm{d}x  \notag \\
		&\quad = \int_{R}(3u^2 + u_x^2 -2Au)u_t \mathrm{d}x + \int_{R}(2uu_x-2\mu u_x)u_{xt}\mathrm{d}x + \int_{R}u_t\rho^2\mathrm{d}x \notag \\
		&\quad = -\int_{R}(3u^2 + u_x^2 -2Au)(uu_x + g_x + h_x) \mathrm{d}x  \notag\\
		&\qquad\, -\int_{R}(2uu_x-2\mu u_x)(u_x^2 + uu_{xx}+g_{xx}+h_{xx})\mathrm{d}x \notag\\
		&\qquad\, -\int_{R}(uu_x+g_x+h_x)\rho^2\mathrm{d}x  \notag\\
		&\quad \triangleq   K_1 + K_2 + K_3.   \label{eqA.5}
	\end{align}
	Below, we calculate each term at the right hand of \eqref{eqA.5}. First, it easily obtains that
	\begin{align*}
		K_1
		&= -\int_{R}(3u^2 + u_x^2 -2Au)(uu_x + g_x + h_x) \mathrm{d}x \notag\\
		&= -\int_{R}uu_x^3\mathrm{d}x - \int_{R}(3u^2+u_x^2-2Au)g_x\mathrm{d}x - \int_{R} (3u^2+u_x^2-2Au)h_x\mathrm{d}x.
	\end{align*}
	Using \eqref{eqA.3}, we have
	\begin{align*}
		K_2
		&= -\int_{R}(2uu_x-2\mu u_x)(u_x^2 + uu_{xx}+g_{xx}+h_{xx})\mathrm{d}x \notag\\
		&= - \int_{R}(2uu_x^3 + 2u^2u_xu_{xx})\mathrm{d}x + \int_{R}(2\mu u_x^3 + 2\mu uu_xu_{xx})\mathrm{d}x  \notag\\
		&\quad\, - \int_{R}(2uu_x - 2\mu u_x)\Big(g-u^2-\frac{1}{2}u_x^2+Au-\mu u_{xx}\Big)\mathrm{d}x \notag\\
		&\quad\,- \int_{R}(2uu_x - 2\mu u_x)\Big(h-\frac{1}{2}\rho^2\Big)\mathrm{d}x \notag\\
		&=-\int_{R}2uu_xg\mathrm{d}x + \int_{R}uu_x^3\mathrm{d}x + \int_{R}2\mu u_xg\mathrm{d}x - \int_{R}\mu u_x^3\mathrm{d}x  \notag\\
		&\quad\; - \int_{R}2uu_xh\mathrm{d}x + \int_{R}uu_x\rho^2\mathrm{d}x + \int_{R}2\mu u_xh\mathrm{d}x - \int_{R}\mu u_x\rho^2\mathrm{d}x.
	\end{align*}
	Similarly, we have
	\begin{align*}
		K_3
		= -\int_{R}\rho^2(uu_x + g_x + h_x) \mathrm{d}x
		= -\int_{R}uu_x\rho^2\mathrm{d}x - \int_{R}2(h-h_{xx})g_x\mathrm{d}x.
	\end{align*}
	Adding $K_1$, $K_2$ and $K_3$ together, we obtain
	\begin{align*}
		\frac{\mathrm{d}H}{\mathrm{d}t}
		&= -\int_{R}(2u^2 + u_x^2 -2Au + 2\mu u_{xx})g_x\mathrm{d}x + \int_{R}2\mu u_{xx}g_x\mathrm{d}x + \int_{R} 2\mu u_x g\mathrm{d}x  \notag\\
		&\quad\, -\int_{R}(2u^2+u_x^2-2Au+2\mu u_{xx})h_x\mathrm{d}x  + \int_{R}2\mu u_{xx}h_x\mathrm{d}x + \int_{R}2\mu u_xh\mathrm{d}x \notag\\
		&\quad\, -\int_{R}\mu u_x^3\mathrm{d}x - \int_{R}2\mu u_x(h-h_{xx})\mathrm{d}x  - \int_{R}2(h-h_{xx})g_x\mathrm{d}x \notag\\
		&= -\int_{R}2(g-g_{xx})g_x\mathrm{d}x - \int_{R}2\mu u_xg_{xx}\mathrm{d}x + \int_{R}2\mu u_xg\mathrm{d}x \notag\\
		&\quad\, -\int_{R}2(g-g_{xx})h_x\mathrm{d}x -\int_{R}2\mu u_xh_{xx} \mathrm{d}x + \int_{R}2\mu u_x h\mathrm{d}x\notag\\
		&\quad\, -\int_{R}\mu u_x^3\mathrm{d}x - \int_{R}2\mu u_x(h-h_{xx})\mathrm{d}x  - \int_{R}2(h-h_{xx})g_x\mathrm{d}x \notag\\
		&= - \int_{R}2\mu u_xg_{xx}\mathrm{d}x + \int_{R}2\mu u_xg\mathrm{d}x -\int_{R}\mu u_x^3\mathrm{d}x  \notag\\
		&= -\int_{R} 2\mu u_x\Big(g-u^2-\frac{1}{2}u_x^2 + Au-\mu u_{xx}\Big)\mathrm{d}x+ \int_{R}2\mu u_xg\mathrm{d}x -\int_{R}\mu u_x^3\mathrm{d}x \notag\\
		&= 0,
	\end{align*}
	in which
	\begin{align*}
		-\int_{R}u^2u_xu_{xx}\mathrm{d}x =  \int_{R} uu_x^3 \mathrm{d}x \quad \text{and} \quad
		-\int_{R}2\mu uu_xu_{xx}\mathrm{d}x = \int_{R}\mu u_x^3\mathrm{d}x
	\end{align*}
	have been used several times during the calculation.
\end{proof}



\end{document}